\documentclass[11pt]{amsart}

\usepackage{amsmath, amsfonts, amssymb,amsthm}
\usepackage{amstext}
\usepackage{mathrsfs}

\usepackage{float,epsf,subfigure}
\usepackage[all,cmtip]{xy}
\usepackage{hyperref}
\usepackage{algorithm}
\usepackage{algorithmic}
\usepackage{mathtools}
\usepackage{float}
\usepackage{bm}
\theoremstyle{plain}
\newtheorem{theorem}{Theorem}[section]

\newtheorem{cor}[theorem]{Corollary}
\newtheorem{def-thm}[theorem]{Definition-Theorem}
\newtheorem{lemma}[theorem]{Lemma}

\newtheorem{defi}[theorem]{Definition}

\newtheorem*{tha}{Theorem A}

\newtheorem*{thb1}{Theorem B1}

\newtheorem*{thb4}{Theorem B4}
\newtheorem*{thc1}{Theorem C1}
\newtheorem*{corc2}{Corollary C2}
\newtheorem*{thd1}{Theorem D1}
\newtheorem*{cord2}{Corollary D2}
\newtheorem*{thd3}{Theorem D3}
\newtheorem*{cord4}{Corollary D4}

\newtheorem*{corb2}{Corollary B2}
\newtheorem*{corb3}{Corollary B3}
\newtheorem*{corb5}{Corollary B5}
\newtheorem*{corb6}{Corollary B6}

\theoremstyle{definition}

\newtheorem{remark}[theorem]{Remark}

\def\min{\mathop{\mathrm{min}}}

\begin{document}
\title[Nevanlinna Theory of Algebroid Functions]{Nevanlinna Theory of Algebroid Functions on Complete K\"ahler Manifolds}
\author[X.-J. Dong]
{Xianjing Dong}

\address{School of Mathematical Sciences \\ Qufu Normal University \\ Qufu, 273165,  P. R. China}
\email{xjdong05@126.com}


\subjclass[2010]{32H30; 32H25} \keywords{Algebroid functions; Nevanlinna theory;   second main theorem;  Picard type  theorem.}
\date{}
\maketitle \thispagestyle{empty} \setcounter{page}{1}

\begin{abstract}  In this paper, we generalize the classical Nevanlinna theory of algebroid functions from $\mathbb C$ to a complete K\"ahler manifold with either non-negative Ricci curvature or non-positive sectional curvature.  As its applications, we establish some Picard type  theorems and  five-value type theorems for algebroid functions under certain  conditions. 
     \end{abstract}  
 
\vskip\baselineskip
\vskip\baselineskip

\tableofcontents
\newpage
\setlength\arraycolsep{2pt}
\medskip

\section{Introduction}
\vskip\baselineskip

\subsection{Motivation}~

The notion of algebroid functions (see \cite{HX, SG})  introduced by H. Poincar\'e is an extension    of   meromorphic functions.  Recall    that a $\nu$-valued algebroid function  $w$ on a  domain  $D\subseteq\mathbb C$   is defined by an irreducible algebraic equation  
$$A_\nu(z) w^\nu+A_{\nu-1}(z)w^{\nu-1}+\cdots+A_0(z)=0, \ \ \ \ A_\nu\not\equiv0,$$
where $A_0,\cdots,A_\nu$ are holomorphic functions without common zeros in $D.$ 
An algebroid function  is viewed     as a special multi-valued analytic function,  
while  which was  first put forward, 
 G. Darboux  deemed   that  it is  an important class of functions. 
To see that,  for example, it  is known  that a $\nu$-valued algebroid function $w$ defines  a $\nu$-leaf Riemann surface  $\mathcal D$ which  is a  
 $\nu$-sheeted ramified analytic  covering of $D.$ 
 A further fact states   that $w$ can lift to a single-valued meromorphic function $f$ on $\mathcal D$ via a natural projection $\pi: \mathcal D\to  D$ such that   $w=f\circ\pi^{-1},$ which is called    the 
 uniformization of $w.$

The most concerned  problem on algebroid functions seems to be the value distribution. In the 1930s after R. Nevanlinna \cite{Nev} founded two fundamental  theorems concerning the  value distribution theory of meromorphic functions in 1925,  many famous scholars  such as 
G. R\'emoundos, H. Selberg, 
  E. Ullrich and G. Valiron, etc. began to study the value distribution theory of algebroid functions. 
  For   example, G. R\'emoundos \cite{GR} proved  a Picard's theorem  
   which     
   claims that   
   a nonconstant $\nu$-valued algebroid function on $\mathbb C$  omits $2\nu$ distinct  values at most in 
  $\mathbb P^1(\mathbb C).$ 
  Afterwards,   G. Valiron  \cite{GV}  showed    without details 
      a five-value type theorem in  unicity theory,  
       which claims   that  two nonconstant $\nu$-valued algebroid functions
   on $\mathbb C$
    will be identical if they share $4\nu$+1 distinct  values ignoring multiplicities in  $\mathbb P^1(\mathbb C).$
Since then, a large number of  scholars  
have devoted themselves to the study of  
    Nevanlinna theory, Ahlfors theory of covering surfaces, 
  unicity theory  and normal family theory, etc. for algebroid functions. 
 We refer  readers   to H. Cartan \cite{HC}, He-Xiao \cite{HX}, 
 Niino-Ozawa \cite{NO}, 
 J. Noguchi \cite{nogu}, 
  M. Ozawa \cite{Oz}, 
M. Ru \cite{ru0, ru}, Sun-Gao \cite{SG},  J. Suzuki \cite{JS},  L. Selberg \cite{LS1, LS2}, N. Toda \cite{To1, To2} and E. Ullrich \cite{EU}, etc. 

Algebroid function theory  in one complex variable 
 didn't mature until  the last 30-40  years. 
 Up to now,    there seems to be  very little literature regarding   algebroid functions in  several complex variables,  
    except for the  limited  work by  
  Hu-Yang \cite{HY}. 
    Let $w, v$ denote  any two  non-constant  algebroid functions  
        on a non-compact Hermitian or K\"ahler manifold $M.$    
                                 We are inspired by the following three   fundamental    problems: 
 \begin{enumerate}
   \item[$\bullet$]  How to  realize the uniformization of $w$?
    
  \item[$\bullet$]  How many Picard's exceptional values  can  $w$ have at most?
    
    \item[$\bullet$] How many values  should   $w, v$ share  at least to guarantee   $w\equiv v$?
\end{enumerate}

In the present  paper, we  answer  these  problems. What is more important is that  
we  develop    the classical  Nevanlinna theory of algebroid functions on $\mathbb C$ to   
 complete  non-compact   K\"ahler manifolds.

\subsection{Main Results}~

\noindent\textbf{A. Uniformization}

Let $M$ be a  non-compact Hermitian manifold. 
We     generalize      the notion of  algebroid functions from $\mathbb C$ to $M$ (see Definition \ref{Def}). 
Let $w$ be a $\nu$-valued algebroid function on $M.$  
Given  a $\lambda$-leaf  function element $(x_0, u)$ of  $w.$ 
 When $x_0$ is not a singular branch point of $u,$   we  give a  local Pusieux expansion  of  $u$ 
   at $x_0,$ i.e., in  a local holomorphic coordinate
     $z:=(\hat z_m, z_m),$  we prove  that $u$  can be  expanded locally into the following  convergent  Pusieux series (see Theorem \ref{ses})
        $$u(z)=B_{0}(\hat z_m)+B_{\tau}(\hat z_m)z_m^{\frac{\tau}{\lambda}}+B_{\tau+1}(\hat z_m)z_m^{\frac{\tau+1}{\lambda}}+\cdots$$
   at $x_0$  with $z(x_0)=\textbf{0},$ 
 where  $B_0, B_\tau, \cdots$ are  meromorphic functions. 
 By  this Pusieux series, we construct a $\nu$-leaf complex manifold $\mathcal M$  and show   that 
  \begin{tha}\label{lift}
 Let $\mathcal M$ be the $\nu$-leaf complex manifold by  a $\nu$-valued algebroid function $w$ on $M.$ 
 Then, $w$ can lift to a meromorphic function $f$ on $\mathcal M$ via the natural projection $\pi: \mathcal M\to M$
  such that $w=f\circ\pi^{-1}.$ 
  \end{tha}

\noindent\textbf{B. Nevanlinna theory}     

In this part and  the following Parts  C, D,  we use    $M$  to stand for   a complete non-compact  K\"ahler manifold. 
Let $w$ be any $\nu$-valued algebroid function on $M.$
Refer to the  definition in  Section \ref{sec53}, we have 
 the   Nevanlinna's functions 
   $T(r, w), m(r, w), N(r, w), \overline N(r, w)$  of $w$
    and $T(r,\mathscr R)$ on $\Delta(r),$   
  where $\mathscr R$ is  the Ricci form of $M,$ 
    and $\Delta(r)$ is a precompact domain containing a fixed  point $o\in M.$ 
      To unify  the form, we write   
$$\overline N(r, w)=\overline N\Big(r, \frac{1}{w-\infty}\Big).$$
 The simple defect  of $w$  with respect to $a\in\mathbb P^1(\mathbb C)$ is defined    by
$$
\bar\delta_w(a)= 1-\underset{{r\rightarrow\infty}}\limsup\frac{\overline N(r,\frac{1}{w-a})}{T(r,w)}.
$$
  
\noindent\textbf{(b1)}  \emph{$M$ has non-negative Ricci curvature} 
 
 Assume that $M$ is non-parabolic.  Set 
 \begin{equation}\label{Hr}
H(r)=\frac{V(r)}{r^2}\int_{r}^\infty\frac{tdt}{V(t)}, 
 \end{equation} 
 where $V(r)$ denotes  the Riemannian volume of the geodesic ball centered at $o$ with radius $r.$
 
\begin{thb1}[Second Main Theorem]
 Let $w$ be a nonconstant   $\nu$-valued algebroid function on $M.$  Let $a_1, \cdots, a_q$ be distinct values in 
  $\mathbb P^1(\mathbb C).$ Then  for any $\delta>0,$ there exists a subset $E_\delta\subseteq(0, \infty)$ of  finite Lebesgue measure such that 
     \begin{eqnarray*}
&&  (q-2\nu)T(r,w)+T(r, \mathscr R) \\
 &\leq& \sum_{j=1}^q\overline{N}\Big(r,\frac{1}{w-a_j}\Big)+O\big(\log^+T(r,w)+\log H(r)+\delta\log r\big)
      \end{eqnarray*}
holds for all $r>0$ outside $E_\delta,$ where $H(r)$ is defined by $(\ref{Hr}).$
\end{thb1}

When $M=\mathbb C^m,$  we  have  $\log H(r)=O(1).$ Hence,  Theorem B1  covers    the classical second main theorem for   $\mathbb C^m$ with $m\geq2$ (see Corollary \ref{haoba}). 

\begin{corb2}[Defect Relation] Assume the same conditions as in Theorem {\rm B1}. Then    
$$\sum_{j=1}^q\bar\delta_w(a_j)\leq 2\nu,$$
if one of the following conditions is satisfied$:$

 $(a)$ $M$ satisfies the volume growth condition 
\begin{equation}\label{cond}
\lim_{r\to\infty}\frac{\log H(r)}{\log r}=0,
\end{equation}
where $H(r)$ is defined by $(\ref{Hr});$ 

$(b)$ $w$ is  of non-polynomial type growth, i.e., 
$w$ satisfies the growth condition 
$$\lim_{r\to\infty}\frac{\log r}{T(r,w)}=0.$$
\end{corb2}

Note that   $r^2=o(V(r))$ as $r\to\infty,$ due to  
$$\int_1^\infty\frac{tdt}{V(t)}<\infty.$$
The  condition (\ref{cond})  is  relaxed. For example, it is satisfied  when  $V(r)=O(r^\mu)$ with $\mu>2,$ or  when   $V(r)=O(r^\mu\log^{\nu}r)$ with $\mu>2$   or  $\mu= 2$  and $\nu>1.$ 

\begin{corb3}[Picard Type Theorem]
Let $w$ be a  nonconstant $\nu$-valued algebroid function on $M.$ Then,  $w$ can omit  $2\nu$ distinct values at most  in $\mathbb P^1(\mathbb C)$  if one of the conditions in Corollary ${\rm{B}}2$ is satisfied.   
\end{corb3}

\noindent\textbf{(b2)}  \emph{$M$ has non-positive sectional  curvature} 

  Assume   that  $M$ has curvature  satisfying   
\begin{equation}\label{curvature}
 K\leq-\sigma^2, \ \ \ \ {\rm{Ric}}\geq-(2m-1)\tau^2
 \end{equation}
 for    constants $\sigma,\tau$ with  that  either $\sigma>0$ or $\sigma=\tau=0,$ where $K, {\rm{Ric}}$ are  the   sectional curvature and   Ricci curvature of $M,$ respectively. 
 Set
   \begin{equation}\label{chi}
 \chi(s, t)=\left\{
                \begin{array}{ll}
t, \  \ & s=0; \\
  \frac{\sinh s t}{s}, \ \ & s\not=0.
              \end{array}
              \right.
               \end{equation}

\begin{thb4}[Second Main Theorem]
 Let $w$ be a nonconstant   $\nu$-valued algebroid function on $M.$  Let $a_1, \cdots, a_q$ be distinct values in 
  $\mathbb P^1(\mathbb C).$ Then  for any $\delta>0,$ there exists a subset $E_\delta\subseteq(0, \infty)$ of finite Lebesgue measure such that 
     \begin{eqnarray*}
&&  (q-2\nu)T(r,w)+T(r, \mathscr R) \\
 &\leq& \sum_{j=1}^q\overline{N}\Big(r,\frac{1}{w-a_j}\Big)+O\left(\log^+T(r,w)+(\tau-\sigma)r+\delta\log\chi(\tau, r)\right)
      \end{eqnarray*}
holds for all $r>0$ outside $E_\delta,$ where $\chi(\tau, r)$ is defined by $(\ref{chi}).$ 
\end{thb4}

Note that Theorem B4  covers    the classical second main theorem for  both  $\mathbb C^m$ and  the unit ball $\mathbb B$ with Poincar\'e metric in $\mathbb C^m$
(see Corollaries  \ref{dong1} and  \ref{dong2}). 

\begin{corb5}[Defect Relation]  Assume the same conditions as in Theorem {\rm{B4.}}  
If 
\begin{equation}\label{xjing}
\lim_{r\to\infty}\frac{(\tau-\sigma)r-T(r,\mathscr R)}{T(r,w)}=0,
\end{equation}
then  
$$\sum_{j=1}^q\bar\delta_w(a_j)\leq 2\nu.$$
  \end{corb5} 
  
    \begin{corb6}[Picard Type Theorem]
Let $w$ be a nonconstant $\nu$-valued algebroid function on $M$  satisfying   $(\ref{xjing}).$
 Then,  $w$ can omit  $2\nu$ distinct values at most  in $\mathbb P^1(\mathbb C).$
\end{corb6}

\noindent\textbf{C. Propagation  of algebraic dependence}
 
Propagation problem for algebraic dependence of   meromorphic mappings was  first studied  by L. Smiley \cite{Sm} and further studied    by 
 Y. Aihara \cite{A1, A2, A3}, S. Drouilhet \cite{Dr},  H. Fujimoto \cite{Fji1, Fji2}, Dulock-Ru \cite{Ru}  and W. Stoll \cite{Sto}, etc. 

Assume that   $M$  is either non-parabolic with non-negative Ricci curvature, or of    curvature satisfying (\ref{curvature}).

Let $Z=\sum_j\nu_j Z_j$ be an effective divisor on $M,$ where  $Z_j{^{,}s}$ are  irreducible hypersurfaces of $M.$ 
Let $k$ be a positive integer $k$ (allowed to be $+\infty$).
Set 
\begin{equation*}\label{}
{\rm{Supp}}_kZ=\bigcup_{1\leq\nu_j\leq k}Z_j.
\end{equation*}
 Let $S_1,\cdots, S_q$ be  hypersurfaces of $M$ satisfying that $\dim_{\mathbb C}S_i\cap S_j\leq m-2$ if $m\geq2$ or $S_i\cap S_j=\emptyset$ if $m=1$ for  $i\not=j.$ 
 Let $a_1,\cdots, a_q$ be  distinct values in $\mathbb P^1(\mathbb C)$ and   $k_1,\cdots, k_q$ be positive integers  (allowed to be $+\infty$).
Denote by $\mathscr G$   the set of all nonconstant algebroid functions on $M.$  Set  $S=S_1\cup\cdots\cup S_q.$ 

 Let 
\begin{equation*}
\mathscr W=\mathscr W\big(w\in \mathscr G; \ \{k_j\}; \ (M, \{S_j\}); \ (\mathbb P^1(\mathbb C),  \{a_j\})\big)
\end{equation*}
be  the set of all  $w\in\mathscr G$ such that 
$$S_j={\rm{Supp}}_{k_j}(w^*a_j), \ \ \ 1\leq j\leq q$$
and  that either 
  (\ref{cond}) is satisfied if $M$ is non-parabolic with non-negative Ricci curvature,  
or  (\ref{xjing}) is satisfied if $M$ has  curvature satisfying (\ref{curvature}). 

  Let $\mathscr O(1)$ be the point  line bundle over $\mathbb P^1(\mathbb C).$  We denote by  $\mathscr H$  the set of all indecomposable  hypersurfaces $\Sigma$ of $\mathbb P^1(\mathbb C)^l$ such that $\Sigma={\rm{Supp}}\tilde D$ for some  $\tilde D\in|\tilde {\mathscr O}(1)|,$ in which  $\tilde{\mathscr O}(1)=\pi^*_1\mathscr O(1)\otimes\cdots\otimes\pi^*_l\mathscr O(1)$
and  $\pi_k: \mathbb P^1(\mathbb C)^l\rightarrow \mathbb P^1(\mathbb C)$ is the natural projection on the $k$-th factor.
Set  
 $$\gamma=\sum_{j=1}^q\frac{k_j}{k_j+1}-\frac{k_0\nu_1\cdots\nu_l}{k_0+1}
\sum_{k=1}^l\frac{1}{\nu_k}-2\nu_{0}$$
with  $k_{0}=\max\{k_1,\cdots,k_q\}$
and 
$\nu_{0}=\max\{\nu_1,\cdots,\nu_l\}.$
\begin{thc1} Let $w^k$ be a $\nu_k$-valued algebrod function in $\mathscr W$ for $1\leq k\leq l.$  Assume that $w^1,\cdots,w^l$  are $\Sigma$-related on $S$ for some $\Sigma\in \mathscr H.$   If $\gamma>0,$
   then $w^1,\cdots,w^l$  are $\Sigma$-related on $M.$
\end{thc1}
Let $\varsigma $ be a  positive integer.  Set
$$\mathscr W_\varsigma=\big\{\text{all $\nu$-valued algebroid functions in} \  \mathscr W \ \text{satisfying} \ \nu\leq\varsigma\big\}.$$
Fix any   $w_0\in\mathscr W_\varsigma.$ Assume that 
 $w_0(M)\cap a_s\not=\emptyset$
 for at least one $s\in\{1,\cdots,q\}.$  
  Let $\mathscr W_{\varsigma}^0$ be  the set of all $w\in\mathscr W_\varsigma$ such that $w=w_0$ on  $S.$
  Put 
 $$\gamma_0=\sum_{j=1}^q\frac{k_j}{k_j+1}-\frac{2k_0\varsigma}{k_0+1}-2\varsigma.$$

 \begin{corc2}
 If $\gamma_0>0,$ then $\mathscr W_\varsigma^0$ contains exactly  one element. 
 \end{corc2}

\noindent\textbf{D. Unicity  problems}

\noindent\textbf{(d1)}  \emph{$M$ has non-negative Ricci curvature} 

Assume that $M$ is  non-parabolic with   volume growth satisfying    (\ref{cond}). 
Let $q$ be a positive integer.  Consider  the following condition
\begin{equation}\label{q}
   \begin{cases}
q-2\mu-\frac{2k_0\nu}{k_0+1}-\underset{j=1}{\overset{q}\sum}\frac{1}{k_j+1}>0  \\
q-2\nu-\frac{2k_0\mu}{k_0+1}-\underset{j=1}{\overset{q}\sum}\frac{1}{k_j+1}>0
\end{cases}
\end{equation}
with $k_0=\max\{k_1,\cdots,k_q\}.$ 
 \begin{thd1} Let $u, v$ be nonconstant $\mu$-valued,  $\nu$-valued algebroid  functions on $M$ respectively.  
  Let $a_1,\cdots,a_q$ be distinct values in $\mathbb P^1(\mathbb C).$    
  Assume that ${\rm{Supp}}_{k_j}(u^*a_j)={\rm{Supp}}_{k_j}(v^*a_j)\not=\emptyset$ for all $j.$ 
   If $q$ satisfies  
 $(\ref{q}),$  then $u\equiv v.$ 
 \end{thd1}

\begin{cord2} [Five-Value Type Theorem]
Let $u, v$ be nonconstant $\mu$-valued, $\nu$-valued algebroid  functions on $M$ respectively.  
  If $u, v$  share $2\mu+2\nu+1$ distinct values ignoring multiplicities  in $\mathbb P^1(\mathbb C),$  then $u\equiv v.$   
\end{cord2}

\noindent\textbf{(d2)}  \emph{$M$ has non-positive sectional  curvature} 

Assume that $M$ has curvature  satisfying (\ref{curvature}).

 \begin{thd3}\label{} Let $u, v$ be nonconstant $\mu$-valued,  $\nu$-valued algebroid  functions on $M$  satisfying  $(\ref{xjing})$ respectively. 
  Let $a_1,\cdots,a_q$ be distinct values in $\mathbb P^1(\mathbb C).$    
  Assume that ${\rm{Supp}}_{k_j}(u^*a_j)={\rm{Supp}}_{k_j}(v^*a_j)\not=\emptyset$ for all $j.$ 
   If $q$ satisfies   $(\ref{q}),$
     then $u\equiv v.$
 \end{thd3}
 
\begin{cord4} [Five-Value Type Theorem]
Let $u, v$ be nonconstant $\mu$-valued, $\nu$-valued algebroid  functions on $M$ satisfying  $(\ref{xjing})$ respectively. 
  If $u, v$  share $2\mu+2\nu+1$ distinct values ignoring multiplicities  in $\mathbb P^1(\mathbb C),$  then $u\equiv v.$   
\end{cord4}

\section{Algebroid Functions on  Hermitian Manifolds}

\vskip\baselineskip

\subsection{Notion of Algebroid Functions}~\label{sec21}

First of all, we would like to  extend     the notion of  algebroid functions   to a    non-compact Hermitian manifold. 
Let $M$ be an $m$-dimensional non-compact Hermitian manifold.   Consider an irreducible algebraic equation in  $w$ on $M$:  
\begin{equation}\label{eq}
A_\nu(x) w^\nu+A_{\nu-1}(x) w^{\nu-1}+\cdots+A_0(x)=0, \ \ \  \ A_\nu\not\equiv0, 
\end{equation}
where
$A_0, \cdots, A_\nu$ are holomorphic functions locally defined on $M$ such that they  well define a meromorphic mapping 
$$\mathscr A=[A_0: \cdots: A_\nu]: \ M\rightarrow\mathbb P^{\nu}(\mathbb C).$$
Let $I_{\mathscr A}:=A^{-1}_0(0)\cap\cdots\cap A^{-1}_\nu(0)$ be the indeterminacy locus of $\mathscr A,$ which is either an analytic set  of complex codimension not smaller  than $2$
for $m\geq2$ or an empty set for $m=1$ in $M.$  
Note that $\mathscr A$ is  holomorphic  on $M\setminus I_{\mathscr A}.$ 

\begin{remark}\label{rek}  The local definition   for $A_0, \cdots, A_\nu$ means  that $A_k=\{A_{kj}\}_{j\in I},$ i.e.,  a collection of holomorphic functions 
$A_{kj}: U_j\rightarrow \mathbb C$  for all  $k, j$ such that 
$$\frac{A_{0i}(x)}{A_{0j}(x)}=\frac{A_{1i}(x)}{A_{1j}(x)}=\cdots=\frac{A_{\nu i}(x)}{A_{\nu j}(x)}, \ \ \ \ ^\forall x\in U_i\cap U_j\not=\emptyset,$$
where  $\{U_j\}_{j\in I}$ is an open covering of  $M.$ Whence,  we see that $A_i/A_j$ is well defined on $M\setminus I_{\mathscr A},$ and it hence defines a  meromorphic function on $M$ with $i, j=0,\cdots,\nu.$
 \end{remark}
 
 \subsubsection{Meromorphic function elements}~
 
 Set
\begin{equation}
\psi(x, w)=A_\nu(x)w^\nu+A_{\nu-1}(x)w^{\nu-1}+\cdots+A_0(x). 
\end{equation}
It gives that   
  \begin{eqnarray*}
\psi_w(x,w)&:=&\frac{\partial\psi(x,w)}{\partial w} \\
&=&\nu A_\nu(x) w^{\nu-1}+(\nu-1)A_{\nu-1}(x)w^{\nu-2}+\cdots+A_1(x).
  \end{eqnarray*}

\begin{defi}
Let $x$ be any point in $M.$  A pair $(x, u)$ is called a regular function element,  if $u$ is a holomorphic function  in a neighborhood of $x;$ and  called a meromorphic function element,  if $u$ is a meromorphic function  in a neighborhood of $x.$  
A pair $(x, u)$ is said to be subordinate to equation $(\ref{eq}),$ if $u$ is  a local   solution of equation $(\ref{eq})$ in a   neighborhood of $x.$  
\end{defi}

\begin{theorem}\label{thm1} Let $(x_0, w_0)\in M\times\mathbb C$ be  a pair such  that 
$$\psi(x_0, w_0)=0, \ \ \  \    \psi_w(x_0, w_0)\not=0.$$
Then, there exists a unique regular function  element $(x_0, u)$ subordinate  to  equation $(\ref{eq})$  satisfying $w_0=u(x_0).$
\end{theorem}
\begin{proof}  The theorem can be confirmed    by applying    implicit function theorem. Here, we give   an alternative proof in   complex analysis. 
Take a $U_\lambda$ such that   $x_0\in U_\lambda$   given by Remark \ref{rek}.   Write $\psi(x, w)$ as  the  form
  \begin{eqnarray*}
\psi(x,w) &=& B_\nu(x)(w-w_0)^\nu+B_{\nu-1}(x)(w-w_0)^{\nu-1}+\cdots+B_0(x) \\
&=& (w-w_0)P(x,w)+B_0(x),
  \end{eqnarray*}
where
$$P(x,w)=B_\nu(x)(w-w_0)^{\nu-1}+B_{\nu-1}(x)(w-w_0)^{\nu-2}+\cdots+B_1(x).$$
By the conditions, we have  
  \begin{eqnarray*}
B_0(x_0) &=& \psi(x_0, w_0)=0, \\
P(x_0, w_0)&=& \psi_w(x_0, w_0):=b\not=0.
  \end{eqnarray*}
The continuity implies that here exist $r, \rho>0$ such that 
$$|B_0(x)|<\frac{\rho|b|}{3}, \ \ \ \ |P(x, w)|>\frac{2|b|}{3}$$
for all $x, w$ with $x\in U_\lambda,$ ${\rm{dist}}(x_0, x)\leq r$ and $|w-w_0|\leq\rho.$ 
In further, restrict $w$ to $|w-w_0|=\rho,$ then we obtain 
$$\left|(w-w_0)P(x,w)\right|>\frac{2\rho|b|}{3}>\frac{\rho|b|}{3}>|B_0(x)|.$$
According to  Rouch\'e's theorem, we deduce  that  for all  $x$ with ${\rm{dist}}(x_0, x)\leq r,$ 
 $\psi(x,w)=(w-w_0)P(x,w)+B_0$ and $(w-w_0)P(x,w)$  have the same number  of  zeros in the disc
 $\{|w-w_0|<\rho\}.$ 
  In fact,  this number of zeros  is 1,  
   since $(w-w_0)P(x,w)$ in $w$ has the unique  zero $w=w_0$ 
       in the disc  $\{|w-w_0|<\rho\}$ for all $x$ with ${\rm{dist}}(x_0, x)\leq r.$ 
        It implies  that $\psi(x, w)=0$ uniquely determines   a function $w=u(x)$ in a smaller geodesic 
                 ball neighborhood $B(x_0, r_0)$  of $x_0$ with radius $r_0$  such that 
$\psi(x, u)=0,  w_0=u(x_0).$
Evidently,   $u$ is continuous on $B(x_0, r_0).$  

Next,  we show that $u$ is holomorphic on $B(x_0, r_0).$
 One just needs to show without loss of generality that $u$ is analytic at $x_0.$
Taking a local holomorphic coordinate $z=(z_1,\cdots, z_m)$ near $x_0$ such that   $z(x_0)=\textbf{0}.$ Thanks to  Hartog's theorem, it is sufficient  to examine   that $u$ is holomorphic in each holomorphic direction $z_k$ at $\textbf{0}.$ 
For $j=0,\cdots,\nu$ and $k=1,\cdots,m,$ put  
  \begin{eqnarray*}
    u_k(z_k)&=&u(0,\cdots,z_k,\cdots,0), \\
    B_{j,k}(z_k)&=&B_j(0,\cdots, 0, z_k, 0, \cdots,0).
      \end{eqnarray*}
Note that 
  \begin{eqnarray*}
  \frac{u_k-w_0}{z_k} 
&=&-\frac{B_{0, k}}{z_k}\frac{1}{B_{\nu,k}(u_k-w_0)^{\nu-1}+B_{\nu-1,k}(u_k-w_0)^{\nu-2}+\cdots+B_{1,k}} \\
&\rightarrow& -\frac{1}{B_1(\textbf{0})}\frac{\partial B_0}{\partial z_k}(\textbf{0})
  \end{eqnarray*}
as $z_k\rightarrow0.$ That is, 
$$\frac{\partial u}{\partial z_k}(\textbf{0})=-\frac{1}{B_1(\textbf{0})}\frac{\partial B_0}{\partial z_k}(\textbf{0}).$$
Thus, $u$ is holomorphic in the holomorphic direction $z_k$ at $\textbf{0}.$ This completes the proof. 
\end{proof}

  \begin{defi}
A point $x\in M$ is called a critical point,  if either   $A_\nu(x)=0$ or $\psi(x, w)$ has a multiple root  as a polynomial in $w.$
We use  $\mathscr S$ to denote  the set of all critical points,  called the critical set.  
A regular point means  a point in $\mathscr T:=M\setminus \mathscr S,$ 
where $\mathscr T$ is called  the regular set. 
  Moreover,  denote by $\mathscr M$  the set of all points $x$ in $\mathscr S$ 
   such that $\psi(x, w)$  has a multiple root  as a polynomial in $w,$ where $\mathscr M$ is called the multiple set
  and a point in $\mathscr M$ is called a multiple point. 
   \end{defi}

For any   $x\in \mathscr T,$  the fundamental theorem of algebra asserts that 
$$A_\nu(x)w^\nu+A_{\nu-1}(x)w^{\nu-1}+\cdots+A_0(x)=0$$
has exactly $\nu$ distinct complex roots. Hence, Theorem \ref{thm1} implies that 

\begin{cor}\label{cor1} For any   $x\in M\setminus\mathscr S,$  there exist  exactly $\nu$ distinct regular function elements $(x, u_1), \cdots, (x, u_\nu)$ subordinate to  equation $(\ref{eq}).$  
\end{cor}

To see the common roots of $\psi(x,w)$ and $\psi_w(x,w),$ we define the \emph{resultant}  of $\psi(x,w)$ and $\psi_w(x,w)$   by a 
$(2\nu-1)$-order  determinant:   
$$R_\psi=\left | \begin{matrix}
A_\nu &A_{\nu-1}   & \cdots & A_2 & A_1 &A_0&0&0&\cdots&0 \\
0 &A_{\nu}   & \cdots & A_3 & A_2 &A_1&A_0&0&\cdots&0 \\
\vdots&\vdots& & \vdots&\vdots &\vdots&  \vdots &\vdots & &\vdots \\
0 & 0   & \cdots & A_\nu & A_{\nu-1} &A_{\nu-2}&A_{\nu-3}&A_{\nu-4}&\cdots&A_0  \\
B_\nu &B_{\nu-1}   & \cdots & B_2 & B_1 &0&0&0&\cdots&0 \\
0 &B_{\nu}   & \cdots & B_3 & B_2 &B_1&0&0&\cdots&0 \\
\vdots&\vdots& & \vdots&\vdots &\vdots&  \vdots &\vdots & &\vdots \\
0 & 0   & \cdots & 0 & B_{\nu} &B_{\nu-1}&B_{\nu-2}&B_{\nu-3}&\cdots&B_1  \\
\end{matrix} \right |,$$
where $B_j=jA_j$ for $j=1,\cdots,\nu.$
Apply Sylvester's elimination method,   we can deduce  that $R_\psi(x)=0$ if and only 
if   $A_\nu(x)=0$ or  
 $\psi(x, w)$ has a   multiple  root  as a polynomial in $w,$ i.e., $x\in\mathscr S.$ 
 Thus,  it means   that $\mathscr S$ is an analytic set of complex codimension $1.$
 Define the  \emph{discriminant}  of $\psi(x, w)$ by 
$$J_\psi(x)=(-1)^{\frac{\nu(\nu-1)}{2}}\frac{R_\psi(x)}{A_\nu(x)}\not\equiv0.$$
 Assume that $\psi(x, w)$ splits into $\nu$ linear factors, say 
 $$\psi(x, w)=A_\nu(x)(w-w_1(x))\cdots(w-w_\nu(x)).$$
A direct computation leads to  
$$J_\psi(x)=A_\nu(x)^{2\nu-2}\prod_{i<j}\big(w_i(x)-w_j(x)\big)^2.$$
If $x_0\in \mathscr S\setminus A^{-1}_\nu(0),$  then $\psi(x_0, w)$ has a multiple root  as a polynomial in $w.$  
If $x_0\in A^{-1}_\nu(0)\setminus I_{\mathscr A},$  we  set $u=1/w$ and rewrite (\ref{eq}) as 
$\psi(x,w)=u^{-\nu}\phi(x,u),$
where 
$$\phi(x,u)=A_\nu(x)+A_{\nu-1}(x)u+\cdots+A_0(x)u^\nu.$$
Since   $\phi(x,0)=A_\nu(x),$ we  see that $x_0$ is  a zero of $u,$ i.e., a pole of $w.$  
In fact,  we    note       from Theorem \ref{pole} (see Section \ref{sec23}) that  $J_\psi(x_0)=0$ if and only if $x_0\in\mathscr M.$
To conclude,  
     Corollary \ref{cor1}  can be extended to  $M\setminus \mathscr M$ as follows. 
\begin{cor}\label{cor2} For any   $x\in M\setminus \mathscr M,$  there exist  exactly $\nu$ distinct meromorphic function elements $(x, u_1), \cdots, (x, u_\nu)$ subordinate to  equation $(\ref{eq}).$  
\end{cor}

 \subsubsection{Analytic continuation of meromorphic function elements}~

We recall that a meromorphic function element on $M$ is a pair $(x, u),$ such that  $u$ is  a meromorphic function  defined in  a  neighborhood $U(x)$ of 
$x\in M.$ 
For   convenience, we  write $\tilde x=(x, u)=(U(x), u)$   without any confusion. 
We  usually  take  $U(x)=B(x, r),$ which is  a  geodesic ball neighborhood 
 centered at $x$ with radius $r.$ For  two meromorphic function elements  
  $\tilde x_1=(U(x_1),  u_1)$ and $\tilde x_2=(U(x_2), u_2),$  
define an equivalent relation:  
  $\tilde x_1\sim \tilde x_2$   if and only if 
 $$x_1=x_2; \ \ \ \  u_1(x)= u_2(x), 
 \ \ \ \  ^\forall  x\in U(x_1)\cap U(x_2).$$ 
\ \ \ \  Set 
 $$\mathcal M^0=\big{\{}\text{all meromorphic function elements subordinate to equation (\ref{eq})}\big{\}}/\sim.$$ 
   \ \ \ \    We   equip  $\mathcal M^0$ with a topology.    
    A meromorphic function element  $\tilde y=(y, v)$ is said to be  
  a  \emph{direct analytic continuation} of  a meromorphic function element $\tilde x=(U(x), u),$ if $y\in U(x)$ and $v= u$ in a   neighborhood of $y.$ 
  By the  analytic continuation, we may  regard that  
  $$\tilde y=(y,v)=(y, u).$$
\ \ \ \    A  neighborhood $V(\tilde x)$ of $\tilde x$ is  defined by a set (containing $\tilde x$) of all  $\tilde y\in \mathcal M^0$ such that   
 $\tilde y$  is a  direct analytic continuation  of $\tilde x$ satisfying  $y\in U_1(x),$ where $U_1(x)$  is a neighborhood of $x$ in $U(x).$
That is, 
$$V(\tilde x)=\big\{(y, u): y\in U_1(x)\big\}.$$
In particular,  a $r$-\emph{neighborhood} $V(\tilde x, r)$ of $\tilde x$  (satisfying that $B(x,r)\subseteq U(x)$), 
is  defined by    
  the set (which contains $\tilde x$) of all  $\tilde y\in \mathcal M^0$ such that   $\tilde y$  is a  direct analytic continuation  of $\tilde x$ satisfying  $y\in B(x, r).$
  In other  words, it is defined  by   $V(\tilde x, r):=\{(y, u): y\in B(x, r)\}.$
    A subset $\mathcal E\subseteq \mathcal M^0$ is called an open set, if either $\mathcal E$ is an empty set or every 
          point $\tilde x\in\mathcal E$ is an inner point, i.e., there  is a $\epsilon$-neighborhood $V(\tilde x, \epsilon)$ 
                    of $\tilde x$ such that $V(\tilde x, \epsilon)\subseteq \mathcal E.$
  It is not hard  to  check that it defines  a topology of $\mathcal M^0.$

  \begin{theorem}\label{thm44} $\mathcal M^0$ is a  Hausdorff space. 
  \end{theorem}
  \begin{proof} 
 Pick    $\tilde x_1\not=\tilde x_2\in \mathcal M^0.$ If $x_1\not=x_2,$ then we can consider   $\epsilon$-neighborhoods 
 $V(\tilde x_1, \epsilon)$ of $\tilde x_1$  and  $V(\tilde x_2, \epsilon)$ of $\tilde x_2,$  where $0<\epsilon<{\rm{dist}}(x_1, x_2)/2.$
 Evidently, we have     
   $V(\tilde x_1, \epsilon)\cap V(\tilde x_2, \epsilon)=\emptyset.$
       If  $V(\tilde x_1, r_1)\cap V(\tilde x_2, r_2)\not=\emptyset$ for   $x_1=x_2,$ 
      then we  can take a function element 
           $(x_0, u_0)\in V(\tilde x_1, r_1)\cap V(\tilde x_2, r_2),$ which   
          gives   that 
 $u_1= u_0= u_2$ in a  smaller neighborhood of $x_0.$ Using the uniqueness theorem of analytic functions, we obtain   $u_1\equiv u_2,$ 
  a  contradiction  with $\tilde x_1\not=\tilde x_2.$
  \end{proof}
 
 \begin{defi}
 Let $\tilde x,$ $\tilde y$ be two  meromorphic function elements on $M.$  Let $\gamma: [0, 1]\to M$ be a curve with   $x=\gamma(0)$ and $y=\gamma(1).$ We say that $\tilde y$ is an analytic continuation of $\tilde x$ along $\gamma,$ if 
 
 $(a)$ for any $t\in [0,1],$ there is a meromorphic function element $\tilde \gamma(t)$ on $M;$
 
 $(b)$ for any $t_0\in [0,1]$ and any $\epsilon>0,$ there exists $\delta>0$ such that $|\gamma(t)-\gamma(t_0)|<\epsilon$
  for all $t\in[0,1]\cap\{|t-t_0|<\delta\},$ and that  $\tilde\gamma(t)$ is a direct analytic continuation of $\tilde\gamma(t_0).$
 \end{defi}

 \begin{remark}\label{rek2}  The analytic continuation of  elements in $\mathcal M^0$ along a curve $\gamma$ in $M$ is a continuous mapping of $\gamma$ into $\mathcal M^0,$ namely, an analytic continuation of $\tilde x$ to $\tilde y$ along $\gamma$ is a curve  connecting $\tilde x$ and $\tilde y.$ Moreover, if $\tilde y$ is an analytic continuation of $\tilde x$ along a curve,  then $\tilde x$ is also an analytic continuation
  of $\tilde y$ along the same curve, i.e.,  the  analytic continuation  is symmetric. 
\end{remark}

\begin{theorem}[Analytic continuation]\label{thm2}
Let  $\gamma: [0, 1]\to M\setminus \mathscr M$ be a curve. 
Then, a function element $\tilde \gamma(0)\in \mathcal M^0$ can  continue analytically to $\gamma(1)$ along $\gamma$ and attain  to   a unique  function element $\tilde\gamma(1)\in \mathcal M^0.$ 
\end{theorem} 
\begin{proof}
Set
$$\tau=\sup\big{\{}t\in[0,1]: \ \text{$\tilde\gamma(0)$  can continue analytically to a $(\gamma(t), u_t)\in\mathcal M^0$}\big{\}}.$$
Evidently,  $0<\tau\leq1.$
It is sufficient  to prove  that  $\tau=1.$ Otherwise, we may assume that $\tau<1.$ 
 By Corollary \ref{cor2},   there exist  exactly $\nu$  function elements  
at  $\gamma(\tau),$ denoted by   
$$\left(B(\gamma(\tau), r_\tau), w_{1}\right), \cdots, \left(B(\gamma(\tau), r_\tau), w_{\nu}\right)\in \mathcal M^0.$$
\ \ \   Take $\delta>0$ small enough such that $\{\gamma(t): t\in[\tau-\delta,\tau+\delta]\}\subseteq B(\gamma(\tau), r_\tau),$ with 
 $\delta<\min\{\tau,1-\tau\}.$
Then, there is  a unique integer $k_0\in\{1,\cdots, \nu\}$  such that $u_{\tau-\delta}\equiv w_{k_0}.$
Notice that  the domain of definition of $w_{k_0}$ is $B(\gamma(\tau),r_\tau),$   for each $t\in[\tau-\delta,\tau+\delta],$ one can redefine function element $(\gamma(t), w_{k_0})\in \mathcal M^0.$ Combined with   
the definition of $\{\gamma(t), u_t\}_{0\leq t\leq\tau-\delta},$ it yields   immediately  
 that $\tilde\gamma(0)$  can  continue analytically  to $\gamma(\tau+\delta)$  along $\gamma,$
 i.e.,  $(\gamma(\tau+\delta), w_{k_0})\in \mathcal M^0$ is an analytic continuation of $\tilde\gamma(0)$ along $\gamma.$ However, 
  it contradicts  with the definition  of $\tau.$  Hence, we obtain   $\tau=1.$ The uniqueness of   $\tilde \gamma(1)$ is immediate due to the uniqueness theorem of analytic functions.  
  \end{proof}
\begin{cor}\label{cor3}
Let  $\tilde X=\{(x, w_j)\}_{j=1}^\nu, \tilde Y=\{(y, u_j)\}_{j=1}^\nu\subseteq\mathcal M^0$ be any two groups of  distinct  function elements at $x, y\in M\setminus \mathscr M,$ respectively.   Then, $\tilde X, \tilde Y$ can continue analytically to each other along any curve 
 in $M\setminus \mathscr M.$
\end{cor} 
\begin{theorem}[Connectivity]\label{thm3} Any two function elements  in $\mathcal M^0$ can continue analytically to each other along a curve in 
$M\setminus \mathscr M.$
\end{theorem} 
\begin{proof}
Let $\{(x, w_j)\}_{j=1}^\nu\subseteq\mathcal M^0$  be $\nu$ distinct  function elements at  every  point       $x\in M\setminus \mathscr M.$ 
With the help  of Corollary \ref{cor3}, it is sufficient  to prove that all $(x, w_1), \cdots, (x, w_\nu)$ can  be continued  analytically to each other.  
Let  $n$ be  the largest    integer such that there are $n$ function elements in  $\{(x, w_j)\}_{j=1}^\nu\subseteq\mathcal M^0,$ which   continue analytically to each other along a curve. Note that   $1\leq n\leq\nu.$
Without loss of generality, one  may assume that 
  the  $n$ function elements are just $(x, w_1), \cdots, (x, w_n).$ Then, it can be  written  as
  $$A_\nu(w-w_1)(w-w_2)\cdots(w-w_n)=B_nw^n+B_{n-1}w^{n-1}+\cdots+B_0, \ \ \ $$
  where $B_0, \cdots, B_n=A_\nu$ are holomorphic functions  in a neighborhood of $x.$
Using Vieta's theorem,  we obtain  
  \begin{eqnarray*}
\sum_{1\leq j\leq n}w_j =-\frac{B_{n-1}}{B_n}&:=&C_1 \\
 \cdots\cdots\cdots &&  \\
\sum_{1\leq j_1<\cdots<j_k\leq n}w_{j_1}\cdots w_{j_k}= (-1)^k\frac{B_{n-k}}{B_n}&:=&C_k \\
 \cdots\cdots\cdots && \\
w_1\cdots w_n= (-1)^n\frac{B_{0}}{B_n}&:=&C_n 
  \end{eqnarray*}
  It suffices   to show that
  
   $(a)$ $C_1,\cdots, C_n$ can continue analytically to the whole 
$M\setminus \mathscr M;$ 

  $(b)$  $n=\nu.$ 

  Apply Theorem \ref{thm2},  $\{(x, w_j)\}_{j=1}^\nu$ can continue analytically to 
  $\{(y, u_j)\}_{j=1}^\nu$ along a curve for any  $y\in M\setminus \mathscr M,$ where $(u_1,\cdots,u_n)$ is a permutation of $(w_1,\cdots,w_n).$ 
By  Vieta's theorem, 
   $u_1,\cdots,u_n$ determine
      the same $C_1,\cdots, C_n$ in a small neighborhood of $y,$  i.e.,  $C_1,\cdots, C_n$ can  continue analytically to $y.$ Hence,  $(a)$ holds. 
If $n<\nu,$ then one can determine another set $\{B'_j\}_j$ 
using  the rest  function elements $\{(x, w_j)\}_j.$
 However, it leads to   that 
equation (\ref{eq}) is reducible, which is a contradiction. This completes the proof. 
\end{proof}

According to Remark \ref{rek2} and Theorem \ref{thm3}, we obtain:  

\begin{cor} $\mathcal M^0$ is a path-connected space. 
\end{cor}

 \subsubsection{Properties of  function elements near a multiple  point}~

 Let $(x_0, w_0)$ be 
 a pair such that $\psi(x_0, w_0)=0$ and $\psi_w(x_0, w_0)=0.$

\noindent $1^\circ$  $x_0\not\in A^{-1}_\nu(0)$

 Let us take a   small neighborhood $U(x_0)$ of $x_0$ such that $A_\nu\not=0$ on $\overline{U(x_0)}.$
  Using  the continuity  of $A_{j}$ for  $0\le j\le \nu,$   there exists a number $M_0>0$ such that 
 \begin{equation}\label{haoba1}
 0\le\min_{0\le j\le\nu}\left|\frac{A_{j}}{A_{\nu}}\right|\leq\max_{0\le j\le\nu}\left|\frac{A_{j}}{A_{\nu}}\right|\leq M_{0}, \ \ \ \   ^\forall x\in \overline{U(x_0)}.
 \end{equation}  
According to   Corollary \ref{cor2},  for any  $x, y\in U(x_0)\setminus \mathscr M$ with $x\not=y,$ 
  there exist   two groups of  distinct meromorphic  function  elements 
 $\{x, u_j\}_{j=1}^{\nu}, \{y, v_j\}_{j=1}^{\nu}$ at $x, y$  subordinate  to equation (\ref{eq}), respectively. 
     Apply     Theorem \ref{thm2},  then  for  a  curve $\gamma: [0,1]\rightarrow U(x_0)\setminus \mathscr M$ 
           with $\gamma(0)=x$ and $\gamma(1)=y,$ we are able to   
           extend    $\{x, u_j\}_{j=1}^{\nu}$   analytically to $\{y, v_j\}_{j=1}^{\nu}$ along $\gamma.$
                                          In other words,  
 there is    a permutation $(i_1,\cdots, i_{\nu})$ of $(1,\cdots,\nu)$ 
   such that $u_1,\cdots, u_{\nu}$ can   continue analytically
       to $v_{i_1},\cdots, v_{i_{\nu}}$ along $\gamma,$ respectively.  
          The  symmetry implies that 
 $\{y, v_j\}_{j=1}^{\nu}$  
  can  also  continue analytically to $\{x, u_j\}_{j=1}^{\nu}$ along $\gamma.$ 
             However,  we  should   mention that  $v_{i_1},\cdots, v_{i_\nu}$ 
                         won't  be, in  general,  continued analytically back to $u_{1},\cdots, u_{\nu}$ along $\gamma,$  respectively. 
                      In what follows,  we would    consider     the analytic continuation of a function element along a  closed Jordan curve surrounding  any multiple point 
      $x_0\in \mathscr M.$  Take    a smooth closed Jordan curve  
   $\gamma_0\subseteq  U(x_0)$   around $x_0$ such that  which is close to  $x_0$   and  doesn't surround or cross 
     any  connected components  of $\mathscr M$ that do not pass through $x_0.$
       Without loss of generality, we  look at   the  analytic continuation of  $u_1$   along $\gamma_0$ starting from $y_0\in\gamma_0.$  
When $u_1$ returns  to $y_0$  along $\gamma_0,$   it  will  turn   to certain    function 
 $u_{k_1}\in\{u_j\}_{j=1}^\nu,$ by which we mean that for a local coordinate chart $(U(x_0), \varphi)$ ($U(x_0)$ is  small enough), $u_1\circ\varphi^{-1}(z)$ jumps   to some  single-valued component $u_{k_1}\circ\varphi^{-1}(z)$ when $z$ travels  around  $\varphi(x_0)$ once  along $\varphi(\gamma_0)$ started at $\varphi(y_0),$ due to the change of 
 $m$-dimensional argument of $z=(z_1,\cdots,z_m).$ If $k_1\not= 1,$ then one   continues to extend $u_{k_1}$  analytically along $\gamma_0$ in the same direction of rotation. 
Repeat      this action,   we shall   see finally  that  $u_1$ cycles on and on periodically.
Let $\lambda_1$ be    the smallest period  such that $u_1$  can return to itself along all such curves $\gamma_0.$
    Then, we can  obtain        $u_1,  u_{k_1},  \cdots,  u_{k_{\lambda_1-1}}.$  
     We  call $x_0$     a \emph{branch point} with order $\lambda_1-1$ of $u_1$ if  $\lambda_1>1.$ 
             Repeatedly, it   can be  seen  that $\{u_j\}_{j=1}^\nu$    is  divided  into finitely many    groups. 
                                       Let $l$ be    the number of groups. 
              If the $j$-th 
group has  $\lambda_j$ members for $j=1,\cdots,l,$
 then  $$\lambda_1+\cdots+\lambda_l=\nu.$$
  We  just analyze      the first group:  $u_1,u_{k_1}, \cdots,u_{k_{\lambda_1-1}}.$
Taking      a   suitable   $U(x_0),$    
     there  exists                
   a biholomorphic mapping $\varphi: U(x_0)\to \Delta^m(2)$  such that  
   $$\varphi(x_0)=\textbf{0}, \ \ \ \  
\varphi(\gamma_0)=\big\{(0,\cdots,0,z_m)\in \Delta^m(2): \ |z_m|=1\big\},$$
where   $\Delta^m(2)$ is   the polydisc  centered at  $\textbf{0}$ with polyradius $(2,\cdots,2)$ in  $\mathbb C^m.$ 
Setting   $S_0=\varphi^{-1}(\Delta_m(1)),$ 
where  $\Delta_m(1)$ is  the unit disc centered at $\textbf{0}$  in the $z_m$-plane.  
 For  any $x\in S_0\cup\gamma_0,$ we have 
 $$u_1(z)=u_1\circ\varphi(x)=u_1(0,\cdots, 0, z_m).$$
Set $z_m=\zeta^{\lambda_1}$ and write 
$$u^*_1(\zeta)=u_1(0,\cdots,0,\zeta^{\lambda_1}).$$
We   show that $u^*_1$ is a single-valued analytic function on $\Delta_m(1).$   
Notice that $\lambda_1$ is the branch order of  $x_0,$ we deduce  that 
 $u^*_1$ is holomorphic on $\Delta_m(1)\setminus\{\textbf{0}\},$ because   
$\varphi(\gamma_0)$ is so close to $\textbf{0}$ that there exist  no any  branch points of $u^*_1$ in $\Delta_m(1)\setminus\{\textbf 0\}.$
 We  only  need to examine  that $\zeta=0$ is a removed singularity of $u^*_1.$ To do so, it suffices   to prove    that $u^*_1$ is bounded near $\textbf{0}$
 due to    Riemann's analytic continuation.  
  In fact,  it follows from   (\ref{eq}) and (\ref{haoba1}) that   
     \begin{eqnarray*}
|u_1|&=& \frac{1}{|A_\nu u_1^{\nu-1}|}\left|A_{\nu-1}u_1^{\nu-1}+\cdots+A_0\right| \\
&\leq& M_0\left(1+\frac{1}{|u_1|}+\cdots+\frac{1}{|u_1|^{\nu-1}}\right)
  \end{eqnarray*}
on $\overline{U(x_0)}.$ If $|u_1|\geq1,$ we obtain   $|u_1|\leq\nu M/m_0.$  Thus, 
$|u_1|\leq\max\{1, \nu M/m_0\}$ on $\overline{U(x_0)},$  i.e.,  $u_1$ is bounded on $\overline{U(x_0)}.$ 
   It means  that  $u^*_1$ is holomorphic on $\Delta_m(1)$
   and it  has   the  Taylor  expansion  
   \begin{equation}\label{dxj1}
   u^*_1(\zeta)=b_0+b_\tau\zeta^\tau+b_{\tau+1}\zeta^{\tau+1}+\cdots
   \end{equation}
at   $\zeta=0.$  
Let $\mathscr B_{u_1}$ be    the set of all branch points of $u_1$  in $U(x_0).$  
Saying  that a branch point $x_0$ of $u_1$  is  \emph{non-singular},  
if $x_0$ does not  lie in any crossings of components of $\mathscr B_{u_1};$ 
and  \emph{singular} otherwise. 
Assume  that     $x_0$ is  not  a  singular branch point of $u_1.$
Then,  $\mathscr B_{u_1}$ is   smooth  when   $U(x_0)$  is  small enough.
Taking a  suitable $U(x_0)$ if necessary.   
Since   $u_1$ is bounded on $\overline{U(x_0)},$ 
and  $\lambda_1$ defines  the branch order of $u_1$ at $x_0$ that  is  
 independent of the choice of $\gamma_0,$  
   we can conclude      from    $(\ref{dxj1})$ and  Riemann’s analytic continuation  that 
              there exists  
        a biholomorphic mapping $\varphi: U(x_0)\to\Delta^m(1)$ with  
  $$\varphi(x_0)=\textbf{0}; \ \ \ \    \varphi(x)=(\hat z_m(x),0), \ \ \ \
  ^\forall   x\in \mathscr B_{u_1}$$
  such that 
   $u_1$  can be expanded locally  into the following  Pusieux series  
   $$u_1(z)=B_0(\hat z_m)+B_\tau(\hat z_m)z_m^{\frac{\tau}{\lambda_1}}+B_{\tau+1}(\hat z_m)z_m^{\frac{\tau+1}{\lambda_1}}+\cdots$$
   at  $x_0,$   where  $B_0, B_\tau, \cdots$ are holomorphic functions  in $\hat z_m:=(z_1,\cdots,z_{m-1}).$  Comparing coefficients, we deduce that 
      $$B_j(\textbf{0})=b_j, \ \ \ \   j=0,\tau,\tau+1,\cdots$$
 \ \ \ \  According to the  above arguments,  
     we show that $u_1$ is a $\lambda_1$-valued analytic   function defined in a neighborhood of  $x_0.$ 
     We call $(x_0, u_1)$ a \emph{$\lambda_1$-leaf function element}.
     If one cuts $U(x_0)$ along a real hypersurface   passing    through $\mathscr B_{u_1},$ 
  then  $u_1$  can  separate  into $\lambda_1$ single-valued analytic components, 
    and all these $\lambda_1$  components can turn to each other when they continue analytically along a closed curve around $x_0$ in $U(x_0)$   that is  sufficiently close to $x_0.$  
      By  the local  expansion of $u_1,$ we note that  $x_0$ is a $\tau$-multiple, $b_0$-valued  branch point  with order $\lambda_1-1$ of $u_1.$

\noindent $2^\circ$  $x_0\in A^{-1}_\nu(0)$

Assume that $x_0$ is  not a singular branch point of $u_1.$ 
We set $u=1/w$ and define  
     \begin{equation*}
\phi(x,u):=u^\nu\psi(x,w) 
=A_\nu(x)+A_{\nu-1}(x)u+\cdots+A_0(x)u^\nu.
     \end{equation*}
 $a)$   $x_0\not\in A^{-1}_0(0)$  
 
 By the previous arguments, one  can  also expand $u_1$  locally into a  Pusieux series  of the form
   $$u_1(z)=C_0(\hat z_m)+C_\tau(\hat z_m)z_m^{\frac{\tau}{\lambda_1}}+C_{\tau+1}(\hat z_m)z_m^{\frac{\tau+1}{\lambda_1}}+\cdots$$
at $x_0.$  If  $C_0\not\equiv0,$   then  $u_1^{-1}$  can be expanded  locally at $x_0$  in  the form 
     \begin{eqnarray*}
   u_1(z)^{-1}
   &=&\tilde C_{-\tau}(\hat z_m)z_m^{-\frac{\tau}{\lambda_1}}+\tilde C_{-\tau+1}(\hat z_m)z_m^{-\frac{\tau-1}{\lambda_1}}+\cdots \\
   & & + \tilde C_{-1}(\hat z_m)z_m^{-\frac{1}{\lambda_1}} +\frac{1}{C_0(\hat z_m)}+\tilde C_{1}(\hat z_m)z_m^{\frac{1}{\lambda_1}}+\cdots
        \end{eqnarray*}
 $b)$    $x_0\in A^{-1}_0(0)$ 
 
 We  take a number  $c_0$  with  $\psi(x_0, c_0)\not=0.$  Set $\tilde w=w-c_0.$ Then   
     \begin{equation*}
\tilde\psi(x,\tilde  w):=\psi(x,\tilde w+c_0) 
=\tilde A_\nu(x)\tilde w^\nu+\tilde A_{\nu-1}(x)\tilde w^{\nu-1}+\cdots+\tilde A_0(x).
     \end{equation*}
It is clear that   $$\tilde A_0(x_0)=\tilde\psi(x_0, 0)=\psi(x_0, c_0)\not=0.$$ Hence,  the case $b)$  turns to the  case $a).$

To conclude,     for each   $x\in M,$ 
  there are     $l=l(x)$ distinct  multi-leaf   function elements subordinate to equation (\ref{eq}), 
    denoted by      $(x, u_1), \cdots, (x, u_l),$ where     $u_j$ is a $\nu_j$-valued meromorphic function 
          with   $1\leq\nu_j\leq\nu$ and $\nu_1+\cdots+\nu_l=\nu.$
       Hence,   $u_j$ has branch order $\nu_j-1$ at $x$ with $j=1,\cdots,l.$ 
        If  $l<\nu,$  then we say that  $x$  is a branch point with order $\nu-l$ of $w.$ Furthermore, if  $x$ is  not  a singular branch  point 
   of $u_{j},$   then  we deduce  that  $u_j$  can  expand  locally  into  the  convergent   Pusieux series at $x.$ That is, we have shown  that 
 
  \begin{theorem}\label{ses}
Let   $(x_0, u)$ be     
a $\lambda$-leaf  function element subordinate to  
equation $(\ref{eq}).$ If $x_0$ is  not a  singular  branch point of $u,$
then $u$  
 can be expanded locally  into the following convergent Pusieux series 
   $$u(z)=B_{0}(\hat z_m)+B_{\tau}(\hat z_m)z_m^{\frac{\tau}{\lambda}}+B_{\tau+1}(\hat z_m)z_m^{\frac{\tau+1}{\lambda}}+\cdots$$
 at $x_0$  in a local holomorphic coordinate $z=(\hat z_m,z_m)$  such that   $z(x_0)=\mathbf{0}$ and $z(x)=(\hat z_m(x), 0)$ for all $x\in\mathscr B_u,$  
 where  $\mathscr B_u$ is the branch set of $u,$ and $B_0, B_\tau, \cdots$ are  meromorphic  functions.  
\end{theorem}
  
 Note that  $B_{0}(\textbf{0})$  in  Theorem \ref{ses} is a $\lambda$-multiple root of $\psi(x_0, w),$ that is, $w$  can separate into $\lambda$ distinct single-valued components with  value    $B_{0}(\textbf{0})$ at $x_0.$ 
  If $A_\nu(x_0)=0,$  then $w$  separates into  $\lambda$ distinct  single-valued components which take   $\infty$ at $x_0.$

According to   the analytic continuation and uniqueness theorem, the above  arguments show that 
 equation (\ref{eq}) uniquely 
defines    a $\nu$-valued  meromorphic function $w$ 
on $M$ 
outside the  branch points of $w.$  Let  $\mathscr B$  stand for the set of all  branch points of $w,$ which  is an analytic set  in $M,$ called the \emph{branch set} of $w.$ 
We  extend   the  notion of algebroid functions from $\mathbb C$ to $M$ as follows. 
\begin{defi}\label{Def}
 A $\nu$-valued algebroid function $w$ on  $M$ is defined by an irreducible algebraic equation  
$$
A_\nu w^\nu+A_{\nu-1} w^{\nu-1}+\cdots+A_0=0, \ \ \  \  A_\nu\not\equiv0, 
$$
where
$A_0, \cdots, A_\nu$ are holomorphic functions locally defined on $M$ such that they   define a meromorphic mapping  $\mathscr A=[A_0:\cdots: A_\nu]: M\to \mathbb P^{\nu}(\mathbb C).$
\end{defi}

In a word,  denote by   $w_1,\cdots,w_\nu$  the   $\nu$ distinct single-valued  components of $w,$ and  given   a 
   $\lambda$-leaf  function element  $(U(x_0), u)$ subordinate  to equation (\ref{eq}), with  $\lambda$ distinct  single-valued   components 
 $\{(U(x_0), u_j)\}_{j=1}^\lambda.$ 
  Then,  
  there  exists       a permutation $(i_1,\cdots,i_\nu)$ of $(1,\cdots,\nu)$ 
  with   $u_j=w_{i_j}$ for $j=1,\cdots, \lambda$  on $U(x_0).$
  
  In the end, we  provide  an example to illustrate  how to calculate the order of  branch points of  algebroid functions.
  
\noindent\textbf{Example.}  Calculating  the branch  order 
of the  algebroid function
$$w=z_2\sqrt[3]{z_1-1}-z_3\sqrt{z_2z_3}\sqrt[4]{z_1-2i}+z_1\sqrt{z_2-3i}\sqrt[3]{z_3-4i}$$
 on $\mathbb C^3$  at   branch points 
 $$P_1(2i,0,1), \  P_2(0,3i, 4i), \ P_3(1,0, 4i),\  P_4(1,0,0).$$
Since $\sqrt{z_2},  \sqrt{z_3},  \sqrt[4]{z_1-2i}$
 have $2,2,4$   single-valued  components, respectively, 
we infer  that $\sqrt{z_2z_3}\sqrt[4]{z_1-2i}$ has  $[2,2,4]=4$  single-valued  components; 
Since  $\sqrt{z_2-3i}, \sqrt[3]{z_3-4i}$ have $2,3$  single-valued  components, respectively, 
we  infer  that $\sqrt{z_2-3i}\sqrt[3]{z_3-4i}$ has $[2,3]=6$ single-valued  components. 
Thus, $w$ is  a $3\times 4\times 6=72$-valued algebroid function. 
Let $\gamma_0$  stand for  any  closed Jordan curve in $\mathbb C^3.$
  All the branch sets of $w$ are   as follows
        \begin{eqnarray*}
        B_{11}=\{z_1=1\}; & \ &  B_{12}=\{z_1=2i\}; \\
        B_{21}=\{z_2=0\}; &\  &  B_{22}=\{z_2=3i\}; \\
        B_{31}=\{z_3=0\}; & \ &  B_{32}=\{z_3=4i\}.
          \end{eqnarray*}
  It is clear  that 
          \begin{eqnarray*}
P_1\in B_{12}\cap B_{21}; &\ & P_2\in B_{22}\cap B_{32}; \\
P_3\in B_{11}\cap B_{21}\cap B_{32}; &\ & P_4\in B_{11}\cap B_{21}\cap B_{31}.
        \end{eqnarray*}
\noindent$a)$  Order of $P_1$

Take  $\gamma_0$ such that   $P_1$ is surrounded by $\gamma_0.$   Note that $B_{12}, B_{21}$ are the only two branch sets of $w$ that pass through $P_1.$ 
 So, $\gamma_0$ is not allowed to  surround or pass through  any other branch sets of $w$ except for $B_{12}, B_{21},$ in order  to determine the  order of $P_1.$
For all such $\gamma_0,$ we  for  $x\in \gamma_0$    
   \begin{eqnarray*}
            |z_1(x)-2i|&<&|2i-1|=\sqrt{5},  \\
        |z_2(x)-0|&<& |0-3i|=3,  \\
        |z_3(x)-1|&<& \min\{|1-0|,  |1-4i|\}=1. 
\end{eqnarray*}
Hence, the order of $P_1$  is  determined by  the orders of branch point  $z_1=2i$ of $\sqrt[4]{z_1-2i}$ and the order of   branch point  $z_2=0$ of $\sqrt{z_2}.$ 
It is $4\times 2-1=7.$

\noindent$b)$  Order of $P_2$

Take  $\gamma_0$ such that   $P_2$ is surrounded by $\gamma_0.$  
 Note that $B_{22}, B_{32}$ are the only two branch sets of $w$ that pass through $P_2.$ 
 Similarly, $\gamma_0$ cannot  surround or pass through  any other branch sets of $w$ except for $B_{22}, B_{32}.$
 It implies  that    for  $x\in \gamma_0$    
   \begin{eqnarray*}
            |z_1(x)-0|&<&\min\{|0-1|, |0-2i|\}=1,  \\
        |z_2(x)-3i|&<& |3i-0|=3,  \\
        |z_3(x)-4i|&<& |4i-0|=4. 
\end{eqnarray*}
Hence, the order of $P_2$  is  determined by  the orders of branch points  $z_2=3i,$  $z_3=4i$ of 
$\sqrt{z_2-3i}\sqrt[3]{z_3-4i}.$ It is $[2, 3]-1=5.$

  \noindent$c)$  Order of $P_3$

Take  $\gamma_0$ such that   $P_4$ is surrounded by $\gamma_0.$  
 Since $B_{11}, B_{21}, B_{32}$ are the only three branch sets of $w$ that pass through $P_3,$ 
 we  have   for  $x\in \gamma_0$    
   \begin{eqnarray*}
            |z_1(x)-1|&<&|1-2i|=\sqrt{5},  \\
        |z_2(x)-0|&<& |0-3i|=3,  \\
        |z_3(x)-4i|&<& |4i-0|=4. 
\end{eqnarray*}
Hence, the order of $P_3$  is  determined by  the orders of branch points $z_1=1,$  
$z_2=0, z_3=4i$
 of  $\sqrt[3]{z_1-1}, \sqrt{z_2}, \sqrt[3]{z_3-4i},$ respectively. 
 It is $3\times 2\times 3-1=17.$
  
  \noindent$d)$  Order of $P_4$

Take  $\gamma_0$ such that   $P_4$ is surrounded by $\gamma_0.$    Since $B_{11}, B_{21}, B_{31}$
 are the only three  branch sets of $w$ that passes through $P_4,$ 
 we have for  $x\in \gamma_0$    
   \begin{eqnarray*}
            |z_1(x)-1|&<&|1-2i|=\sqrt{5},  \\
        |z_2(x)-0|&<& |0-3i|=3,  \\
        |z_3(x)-0|&<& |0-4i|=4. 
\end{eqnarray*}
Hence, the order of $P_4$  is determined by  the orders of branch points  $z_1=1,$  $z_2=0, z_3=0$
 of $\sqrt[3]{z_1-1}, \sqrt{z_2}, \sqrt{z_3}.$  
It is $3\times[2,2]-1=5.$
  
 \subsection{Multi-leaf Complex  Manifolds by  Algebroid Functions}~\label{sec22}

Let $\tilde x=(U(x), u)\in\mathcal M^0.$  A 
 neighborhood $V(\tilde x)$ of $\tilde x$  is defined to be   a set of all  $\tilde y\in \mathcal M^0$ with $y\in U(x)$ such that   
  $\tilde y$  is a  direct analytic continuation  of $\tilde x,$  i.e.,   
 $$V(\tilde x)=\big\{(y, u): y\in U(x)\big\}.$$
   Take   a suitable  $U(x),$   
  we have   a
    biholomorphic mapping $\varphi_x: U(x)\to \mathbb B^m(1)$ 
        such that $\varphi_x(x)=\textbf{0},$ where $\mathbb B^m(1)$ is the unit ball centered at the coordinate origin $\textbf{0}$  in $\mathbb C^m.$
 Dfine a natural mapping $\psi_{\tilde x}: V(\tilde x)\to \mathbb B^m(1)$ by  
$$\psi_{\tilde x}(\tilde y)=\varphi_x(y), \ \ \  \  ^\forall \tilde y\in V(\tilde x).$$
 Clearly, $\psi_{\tilde x}$ is a homeomorphism whose   inverse is 
 $$\psi^{-1}_{\tilde x}(z)=\big(\varphi^{-1}_x(z), u\big), \ \ \  \  ^\forall z\in\mathbb B^m(1).$$
\ \ \ \ Set $$\mathscr U=\left\{(V(\tilde x), \psi_{\tilde x}): \tilde x\in\mathcal M^0\right\}.$$ 
 We  prove  that  $\mathscr U$  is a complex  atlas  of $\mathcal M^0.$
Let $V(\tilde x)\cap V(\tilde y)\not=\emptyset.$  It  suffices   to examine  that  the transition mapping 
$$\psi_{\tilde y}\circ\psi^{-1}_{\tilde x}: \ \psi_{\tilde x}\left(V(\tilde x)\cap V(\tilde y)\right)\to \psi_{\tilde y}\left(V(\tilde x)\cap V(\tilde y)\right)$$
is a biholomorphic mapping. This is automatically a homeomorphism. 
Note  that for  each $\tilde a\in V(\tilde x)\cap V(\tilde y),$ we have $u=v$ in a small neighborhood of $a,$ i.e., $(a, u)=(a, v).$ 
A direct  computation gives that 
 $$\psi_{\tilde y}\circ\psi^{-1}_{\tilde x}(z)=\psi_{\tilde y}\big(\varphi^{-1}_x(z), u\big)
 =\psi_{\tilde y}\big(\varphi^{-1}_x(z), v\big)=\varphi_y\circ\varphi^{-1}_x(z)$$
 for any  $z\in \psi_{\tilde x}(V(\tilde x)\cap V(\tilde y)).$ Similarly, we have  
 $$\psi_{\tilde x}\circ\psi^{-1}_{\tilde y}(z)=\varphi_x\circ\varphi^{-1}_y(z), \ \ \ \    ^\forall z\in \psi_{\tilde y}(V(\tilde x)\cap V(\tilde y)).$$ 
 It shows that  $\psi_{\tilde y}\circ\psi^{-1}_{\tilde x}$ is a  biholomorphic mapping. 
 On the other hand,  each $x\in M\setminus \mathscr B$  corresponds to  
$\nu$ distinct  function elements
  $\{(x, u_j)\}_{j=1}^\nu\subseteq\mathcal M^0.$  That is,  
 we conclude that

\begin{theorem} $\mathcal M^0$ is  a non-compact $\nu$-leaf complex manifold of complex dimension $m.$
\end{theorem}

Set 
 $$\mathcal M^1=\big{\{}\text{all regular function elements subordinate to equation (\ref{eq})}\big{\}}/\sim.$$ 
\begin{cor} $\mathcal M^1$ is  a non-compact $\nu$-leaf complex manifold of complex dimension $m,$ which is a submanifold of $\mathcal M^0.$
\end{cor}
In what follows, we  would  construct a $\nu$-leaf  complex manifold determined  by  all  function elements subordinate to equation (\ref{eq}). 
Let  $\tilde x_1=(U(x_1), u_1),$ $\tilde x_2=(U(x_2), u_2)$ be any two   function elements subordinate to equation (\ref{eq}). We define an  equivalent relation as follows:  $\tilde x_1\sim\tilde x_2$ if and only if 
 $$x_1=x_2; \ \ \ \  u_{10}(x)= u_{20}(x), \ \  \ \    ^\forall x\in  \left(U(x_1)\setminus\mathscr B_{u_1}\right)\cap \left(U(x_2)\setminus\mathscr B_{u_2}\right)$$
 for some two single-valued components $u_{10}, u_{20}$ of $u_1, u_2,$ respectively, where $\mathscr B_{u_j}$ is the branch set of $u_j$  with   $j=1,2.$ 
 Set 
 $$\mathcal M=\big{\{}\text{all function elements subordinate to equation  (\ref{eq})}\big{\}}/\sim.$$ 

\begin{defi}
A  meromorphic  function  element $(y, v)$ is called   a direct analytic continuation of a multi-leaf   function element $(U(x), u),$ if 
both $y\in U(x)\setminus\mathscr B_{u}$ and $v=u_{0}$
in a  neighborhood of $y$ for some single-valued component $u_{0}$ of $u,$ where $\mathscr B_u$ is the branch set of $u.$    If so, one may regard that $(y, v)=(y, u_0).$
\end{defi}

 Let $\tilde x=(U(x), u)\in\mathcal M.$  A  neighborhood $V(\tilde x)$ of $\tilde x$ is  a set which contains   the following elements: 
     \begin{enumerate}
   \item[$\bullet$] all  $\tilde y\in \mathcal M$  with $y\in U_1(x)\setminus\mathscr B_{u}$ satisfying  that   
 $\tilde y$  is a  direct analytic continuation  of $\tilde x;$ 
  \item[$\bullet$]  all  $(y, u)\in \mathcal M$ with  $y\in U_1(x)\cap \mathscr B_u.$ Here and before,  $U_1(x)\subseteq U(x)$ is an arbitrary  neighborhood of $x.$
    \end{enumerate}
    
In particular,  a $r$-\emph{neighborhood} $V(\tilde x, r)$ of $\tilde x$ with $B(x, r)\subseteq U(x)$ 
is  defined by  the set that  contains the following elements: 
    \begin{enumerate}
   \item[$\bullet$] all  $\tilde y\in \mathcal M$ with $y\in B(x, r)\setminus\mathscr B_u$ satisfying  that   
 $\tilde y$  is a  direct analytic continuation  of $\tilde x;$  
   \item[$\bullet$]  all  $(y, u)\in \mathcal M$ with  $y\in B(x, r)\cap \mathscr B_u.$
    \end{enumerate}
  
     A subset $\mathcal E\subseteq \mathcal M$ is called an open set, if  $\mathcal E$ is an
      empty set or every  point $\tilde x\in\mathcal E$ is an inner point, i.e., there exists   a $\epsilon$-neighborhood $V(\tilde x, \epsilon)$ of $\tilde x$ such  
      that $V(\tilde x, \epsilon)\subseteq \mathcal E.$
    It  gives    a topology of $\mathcal M,$ which 
   contains  the topology of $\mathcal M^0.$ If $\tilde x$ is a $\lambda$-leaf function element,  then $V(\tilde x)$ has $\lambda$ leaves at $\tilde x.$
 We call $V(\tilde x)$ a \emph{$\lambda$-leaf neighborhood} of $\tilde x.$

  \begin{theorem} $\mathcal M$ is a connected  Hausdorff space. 
  \end{theorem}
  \begin{proof} 
The  Hausdorff property of $\mathcal M$  is showed  using the similar arguments as in the proof of Theorem \ref{thm44}. We assert  
   that $\mathcal M$ is connected.  Otherwise,   one  may assume  that  $\mathcal M=\mathcal A\cup\mathcal B,$ where $\mathcal A, \mathcal B$ 
   are two disjoint open subsets of $\mathcal M.$
      Set   $\mathcal S=\{\tilde x\in \mathcal M: x\in\mathscr B\}.$ It is evident  that     
          $\mathcal M^0=(\mathcal A\setminus \mathcal S)\cup (\mathcal B\setminus\mathcal S).$ 
                 Since $\mathscr B$ is an analytic set of complex codimension $1,$ 
                                  we see that 
                              $\mathcal A\setminus \mathcal S, \mathcal B\setminus\mathcal S$ are   two disjoint open subsets of $\mathcal M^0.$ 
                               This 
        implies that $\mathcal M^0$ is not connected, which is a contradiction. 
  \end{proof}

Next, we will equip   $\mathcal M$ with a complex structure such that $\mathcal M$ is a complex manifold. 
 Taking a  $\lambda$-leaf function element  $\tilde x=(U(x), u)\in \mathcal M.$  
  Assume that  $x$ is not a singular branch point of $u.$  Let    
   $U(x)$  be  chosen suitably.  
      By means of Theorem \ref{ses}, 
   there exists  a biholomorphic mapping
     $\varphi_x: U(x)\to\Delta^m(1)$ such that        
    $u$  can be expanded into  the Pusieux series 
          \begin{eqnarray*}
   u(z)=u\circ\varphi^{-1}_x(z) 
   =B_{0}(\hat z_m)+B_{\tau}(\hat z_m)z_m^{\frac{\tau}{\lambda}}+B_{\tau+1}(\hat z_m)z_m^{\frac{\tau+1}{\lambda}}+\cdots
         \end{eqnarray*}
on $U(x)$   with    $\varphi_x(x)=\textbf{0}$ and   $\varphi_x(y)=(\hat z_m(y), 0)$
 for all $y\in\mathscr B_u.$ 

Let  $V(\tilde x)$  be  a $\lambda$-leaf neighborhood  of $\tilde x$ containing 
 the following elements: 
 
   \begin{enumerate}
   \item[$\bullet$] all  $\tilde y\in \mathcal M$ with $y\in U(x)\setminus\mathscr B_u$ satisfying  that   
 $\tilde y$  is a  direct analytic continuation  of $\tilde x;$  
   \item[$\bullet$]  all  $(y, u)\in \mathcal M$ with  $y\in U(x)\cap \mathscr B_u.$
    \end{enumerate} 
 Let $u_1,\cdots, u_\lambda$ be $\lambda$ distinct single-valued  components of $u.$ 
Define a mapping 
$\psi_{\tilde x}:  V(\tilde x)\to \Delta^m(1)$
 by 
      \begin{eqnarray*}
\psi_{\tilde x}(y, u)&=& \big(\hat z_m(y), 0\big),  \quad  \quad \quad \quad  \quad \quad \quad \quad \quad \quad  \quad \ \    ^\forall y\in U(x)\cap\mathscr B_u; \\
   \psi_{\tilde x}(y, u_j)&=&\Big(\hat z_m(y), \ |z_m(y)|^{\frac{1}{\lambda}}e^{i\frac{\arg z_m(y)+2(j-1)\pi}{\lambda}}\Big), 
 \quad  ^\forall y\in U(x)\setminus\mathscr B_u
     \end{eqnarray*}
for $j=1,\cdots,\lambda$, where  $z(y)=\varphi_x(y).$  We note that $\psi_{\tilde x}$  is a homeomorphism and its inverse   is  
$$
\psi^{-1}_{\tilde x}(z) = 
\begin{cases}
\big(\varphi^{-1}_x(z), u\big), & ^\forall  z_m=0; \\
\left(\varphi_x^{-1}\big(\hat z_m, |z_m|^\lambda e^{i\lambda\arg z_m}\big), u_j\right),  &  ^\forall z_m\not= 0,
\end{cases}$$
where $j>0$ is an integer such that
 $$0\le2(j-1)\pi-\lambda\arg z_m<2\pi.$$
\ \ \ \  Let $\mathcal M_{\rm{sing}}$ stand for      the set of all  function elements $(x, u)\in\mathcal M$ such   that $x$ is a singular 
branch point  of $u.$ Moreover,  we denote by  $\mathscr B_{\rm sing}$  the set of all the  corresponding singular branch points of $w,$ which  
 is obvious  an analytic set of complex codimension not smaller  than $2.$
Set $$\mathscr U=\big\{(V(\tilde x), \psi_{\tilde x}): \tilde x\in\mathcal M^*\big\}, \ \ \ \  \mathcal M^*=\mathcal M\setminus\mathcal M_{\rm{sing}}.$$
We now show that  $\mathscr U$  is a complex atlas  of $\mathcal M^*.$
Assume that  $V(\tilde x)\cap V(\tilde y)\not=\emptyset.$  It is sufficient  to verify that    the transition mapping 
$$\psi_{\tilde y}\circ\psi^{-1}_{\tilde x}: \ \psi_{\tilde x}\big(V(\tilde x)\cap V(\tilde y)\big)\to \psi_{\tilde y}\big(V(\tilde x)\cap V(\tilde y)\big)$$
is a biholomorphic mapping.  It is automatically a homeomorphism. 
Through  a simple  analysis,  we only  need to deal with  the  situation  in which  $\tilde x=(x, u)$ is a $\lambda$-leaf  meromorphic  function  element 
and $\tilde y=(y, v)$ is  a  regular  function   element.  For  any  $\tilde a\in V(\tilde x)\cap V(\tilde y),$ 
there exists  a single-valued component $u_j$ of $u$ satisfying  that $u_j=v$ in some   neighborhood of $a,$ 
i.e., $(a, u_j)=(a,v).$ It is clear that 
$z_m\not=0,$ since  all  function elements in $V(\tilde y)$ are meromorphic. 
 A direct computation leads to 
      \begin{eqnarray*}
\psi_{\tilde y}\circ\psi^{-1}_{\tilde x}(z)&=&\psi_{\tilde y}\left(\varphi_x^{-1}\big(\hat z_m, |z_m|^\lambda e^{i\lambda\arg z_m}\big), u_j\right) \\
&=&\psi_{\tilde y}\left(\varphi_x^{-1}\big(\hat z_m, |z_m|^\lambda e^{i\lambda\arg z_m}\big), v\right) \\
&=&\varphi_y\circ\varphi^{-1}_x\Big(\hat z_m, |z_m|^\lambda e^{i\lambda\arg z_m}\Big) \\
&=& \varphi_y\circ\varphi^{-1}_x\big(\hat z_m, z_m^\lambda\big),
      \end{eqnarray*}
where $j>0$ is an integer such that $$0\le2(j-1)\pi-\lambda\arg z_m<2\pi.$$ Thus,  $\psi_{\tilde y}\circ\psi^{-1}_{\tilde x}$ is a holomorphic mapping. 
On the other hand, its inverse is
      \begin{eqnarray*}
\psi_{\tilde x}\circ\psi^{-1}_{\tilde y}(z)&=&\psi_{\tilde x}\big(\varphi_y^{-1}(z), v\big) \\
&=& \psi_{\tilde x}\big(\varphi_y^{-1}(z), u_j\big) \\
&=& \varphi_x\circ\varphi^{-1}_y\Big(\hat z_m, \ |z_m|^{\frac{1}{\lambda}}e^{i\frac{\arg z_m+2(j-1)\pi}{\lambda}}\Big) \\
&=&\varphi_x\circ\varphi^{-1}_y\Big(\hat z_m,  \big(z_m^{\frac{1}{\lambda}}\big)_j\Big),
      \end{eqnarray*}
    where $\big(z_m^{1/\lambda}\big)_j$ is the $j$-th single-valued component of $z_m^{1/\lambda}.$ Evidently,   $\psi_{\tilde x}\circ\psi^{-1}_{\tilde y}$ is also a holomorphic mapping. Hence, $\mathcal M^*$ is a  non-compact $\nu$-leaf complex manifold of complex dimension $m.$ 
    
Finally,  we    shall  treat       the singular branch  set   $\mathcal M_{\rm sing}.$  Let $\pi: \mathcal M\to M$ be the natural projection which maps   a  function  element 
    $(x, u)$ to $x.$  Note that   $\pi: \mathcal M^* \to M\setminus\mathscr B_{\rm sing}$ is a  $\nu$-sheeted  ramified analytic covering such that  
    $$M=\pi(\mathcal M^*)\cup\pi(\mathcal M_{\rm sing})=\pi(\mathcal M^*)\cup \mathscr B_{\rm{sing}}.$$
         Since  $\mathscr B_{\rm sing}$ is an analytic set of complex codimension not smaller  than $2,$  the complex structure of $M$  can be viewed  as a natural extension of the complex structure of $\pi(\mathcal M^*)=M\setminus\mathscr B_{\rm sing}.$ This implies   that  
             one  can naturally extend the complex structure of $\mathcal M^*$   to $\mathcal M_{\rm sing}$ via $\pi.$ 
                          Hence, we conclude that           
   \begin{theorem}\label{cover}
    $\mathcal M$  is a $\nu$-sheeted ramified analytic covering of $M.$  We  call $\mathcal M$ the $\nu$-leaf complex manifold by $w.$
      \end{theorem} 
      \begin{cor}[Uniformization]\label{ppll}
 An algebroid function $w$ on $M$ can  lift to a meromorphic function $f$ on $\mathcal M$ via the natural projection $\pi:  \mathcal M\to M$ by mapping  $(x, u)$ to $x$ 
such that $w=f\circ\pi^{-1}.$
\end{cor}

   \subsection{Zero  and Pole Divisors  of Algebroid  Functions}~\label{sec23}
   
   Let $w$ be a $\nu$-valued algebroid function on $M$ defined by equation $(\ref{eq}).$  Let $w_1,\cdots, w_\nu$ be $\nu$ single-valued  components of $w.$ 
   \begin{theorem}\label{pole} We have 
   $$(w=\infty)=(w_1=\infty)+\cdots+(w_\nu=\infty)=(A_\nu=0).$$
   \end{theorem}
   \begin{proof} Let $x_0$ be a zero of $A_\nu.$
   Since  
   $I_{\mathscr A}=0$ in the sense of divisors,  we may  assume that $x_0\not\in I_{\mathscr A}.$ It follows  that 
   \begin{equation}\label{ppt} 
   n\big(x_0, A_\nu=0\big)=\max_{0\leq j\leq\nu-1}n\Big(x_0, \frac{A_j}{A_\nu}=\infty\Big),
   \end{equation}
where $n(x_0, g=a)$ denotes the number of  $a$-valued points of  a function $g$ at $x_0.$ 
Without loss of generality,   we assume    that   $w_1, \cdots, w_p$ are  all   components  that  take $\infty$ at $x_0.$    According to  Vieta's theorem, we obtain     
  \begin{eqnarray*}
\sum_{1\leq j\leq \nu}w_j &=&-\frac{A_{\nu-1}}{A_\nu} \\
 \cdots\cdots\cdots && \\
\sum_{1\leq j_1<\cdots<j_p\leq \nu}w_{j_1}\cdots w_{j_p}&=& (-1)^p\frac{A_{\nu-p}}{A_\nu} \\
 \cdots\cdots\cdots &&  \\
w_1\cdots w_\nu&=& (-1)^\nu\frac{A_{0}}{A_\nu} 
  \end{eqnarray*}
     It implies that   $A_{\nu-p}(x_0)\not=0,$ since  $w_1\cdots w_p$  reaches   the maximal  pole order  at  $x_0\not\in I_{\mathscr A}.$
         It yields from (\ref{ppt}) that   
        \begin{eqnarray*}
  n(x_0, w=\infty)  
  &=& n(x_0, w_1=\infty)+\cdots+n(x_0, w_p=\infty) \\
       &=& n(x_0, w_1\cdots w_p=\infty) \\
       &=& n\Big(x_0, \frac{A_{\nu-p}}{A_\nu}=\infty\Big) \\
       &=& n(x_0, A_\nu=0).
      \end{eqnarray*}
Thus, we conclude   that 
   $(w=\infty)=(A_\nu=0).$  
\end{proof}

Similarly,   we can show  by setting  $u=1/w$ that 
   \begin{theorem}\label{zero} We have 
   $$(w=0)=(w_1=0)+\cdots+(w_\nu=0)=(A_0=0).$$
   \end{theorem}

      \section{Uniformization of Algebroid Functions without Construction} 

In this section,  we  provide   another approach  to realize   the uniformization of algebroid functions without constructing finite  ramified analytic coverings.  Our main  idea  
    originates  from   
      W. Stoll \cite{Stoll}. 
    
      \subsection{Algebroid Reduction}~

Let $M$ be a non-compact complex manifold of complex dimension $m.$
Let $\pi: N\to M$ be a  $\nu$-sheeted ramified analytic covering,   which  means  
a  proper,  surjective   holomorphic mapping  
with a ramification set $\mathscr B_\pi,$ i.e.,  
$$\pi: \  N\setminus\mathscr B_\pi\to M\setminus\pi(\mathscr B_\pi)$$
is a  proper,  surjective, local  biholomorphic mapping of  sheet number $\nu,$ and thus a $\nu$-sheeted  analytic covering. 
Here,  $\mathscr B_\pi$ is the set of those 
points $x\in N$ such that $\pi$ is not locally biholomorphic at $x.$
Note that $\mathscr B_\pi$ is   an empty set or 
an analytic set  of   pure complex codimension $1.$  
By R. Remmert \cite{Rem} (see  Andreotti-Stoll \cite{A-S} also),  the \emph{exceptional set} 
$$\mathscr E_\pi=\left\{x\in N:  \dim_{\mathbb C, x}\pi^{-1}(\pi(x))>0\right\}$$
is an analytic set satisfying that $\mathscr E_\pi\subseteq\mathscr B_\pi$ and $\dim_{\mathbb C}\pi(\mathscr E_\pi)\leq m-2.$ Moreover, 
$\dim_{\mathbb C}\pi(\mathscr B_\pi)\leq m-1$ and $\dim_{\mathbb C, x}\pi(\mathscr B_\pi)= m-1$ for $x\in \pi(\mathscr B_\pi)\setminus\pi(\mathscr E_\pi).$ 
 Let $x\in N\setminus \mathscr E_\pi.$ The \emph{mapping degree} $\lambda_\pi(x)$ of $\pi$ at $x$ is defined in such manner: if $U$ is a   neighborhood of $x$ such that $ \overline U\cap\pi^{-1}(\pi(x))=\{x\},$ then 
  $$\lambda_\pi(x)=\limsup_{y\rightarrow x}\#\Big(\overline U\cap\pi^{-1}(\pi(y))\Big),$$
 which is independent of the choice of $U.$ Note that $1\leq \lambda_\pi<\infty$ and $\lambda_\pi(x)=1$ if and only if $x\in N\setminus \mathscr B_\pi.$
 By W. Stoll \cite{Stoll},  for any $y\in M\setminus \mathscr \pi(\mathscr E_\pi),$  we have
 $$\sum_{x\in \pi^{-1}(y)}\lambda_\pi(x)=\nu.$$
\ \ \   Let $V$ be a complex vector space of complex dimension $d+1,$ with a basis $e_0, \cdots, e_d.$
 Let  $\mathbb P(V)=V^*/{\mathbb C^*}$ be the projection of $V,$    where  
  $\mathbb P:V^*\rightarrow \mathbb P(V)$ is the natural projection.
 Let 
 $\odot^\nu V$
 be the $\nu$-fold symmetric tensor product of $V,$  i.e., 
 $$\odot^\nu V={\rm{Span}}_{\mathbb C}\Big\{e_{j_0}\odot\cdots\odot e_{j_d}: \  0\leq j_0\leq\cdots\leq j_d\leq d\Big\},$$
 where 
  $$v_1\odot\cdots\odot v_\nu:=\frac{1}{\nu!}\sum_{\sigma\in S_\nu}v_{\sigma(1)}\otimes\cdots\otimes v_{\sigma(\nu)}$$
  is called  the symmetric tensor product of $v_1,\cdots,v_\nu\in V,$ and   
   $S_\nu$ stands for   the $\nu$-order symmetric  group. 
    Note  that $\odot^\nu V$ is a subspace  of $\nu$-fold tensor product space $\otimes^\nu V$ 
    with $\dim_{\mathbb C}\odot^\nu V=(d+\nu)!/d!\nu!.$ Again, put  
  $$\mathbb P(V)^{(\nu)}=\Big\{[v_1]\odot\cdots\odot [v_\nu]: \  v_1,\cdots, v_\nu\in V\Big\},$$
where $[v]=\mathbb P(v)$ for any $v\in V.$ Naturally,    $\mathbb P(V)^{(\nu)}$ is  embedded into $\mathbb P(\odot^\nu V)$ by the  mapping
$\imath: [v_1]\odot\cdots\odot[v_\nu]\mapsto [v_1\odot\cdots\odot v_\nu].$ Thus, we obtain 
\begin{equation}\label{po2}
\mathbb P(V)^{(\nu)}\subseteq \mathbb P(\odot^\nu V).
\end{equation}

 \begin{defi}\label{de1} Let $U$ be a nonempty  open subset of $M.$ Let $\vartheta:  \pi^{-1}(U)\rightarrow V$ be a continuous mapping.  Then, the algebroid reduction of $\vartheta$  via $\pi$ is a mapping $\vartheta_\pi: U\setminus \pi(\mathscr E_\pi)\rightarrow \odot^\nu V$ defined by 
 $$\vartheta_\pi(x)=\bigodot_{y\in\tilde x}\vartheta(y)^{\lambda_\pi(y)}, \ \ \ \  ^\forall x\in U\setminus \pi(\mathscr E_\pi).$$
 \end{defi} 

\begin{theorem}[Stoll, \cite{Stoll}]  $\vartheta_\pi$ is a continuous mapping. If $\vartheta$ is  holomorphic, then $\vartheta_\pi$ is also  holomorphic and extends to a holomorphic mapping $\vartheta_\pi: U\rightarrow \odot^\nu V.$
\end{theorem}
 
  \begin{defi}    Let $f: N\rightarrow \mathbb P(V)$ be a meromorphic mapping with indeterminacy locus  $I_f.$  Then, the  algebroid reduction of $f$ via $\pi$ is a meromorphic mapping $f_\pi:  M\rightarrow \mathbb P(\odot^\nu V)$ with 
 indeterminacy locus  $I_{f_\pi}\subseteq \pi(I_f)\cup \pi(\mathscr E_\pi),$ defined by 
 $$f_\pi(x)=\bigodot_{y\in\pi^{-1}(x)}[f(y)]^{\lambda_\pi(y)}, \ \ \ \ ^\forall x\in M\setminus I_{f_\pi}.$$
 \end{defi} 

   \begin{theorem}[Stoll, \cite{Stoll}]\label{S3}    Let $f: N\rightarrow \mathbb P(V)$ be a meromorphic mapping with indeterminacy locus $I_f.$  Let $f_\pi:  M\rightarrow \mathbb P(\odot^\nu V)$ be the algebroid reduction of $f$ via $\pi.$  
     Fix any $x_0\in M\setminus \pi(\mathscr E_\pi).$ Assume that  there exist  a nonempty  neighborhood $U$ of $x_0$ such that $U\subseteq M\setminus \pi(\mathscr E_\pi)$ and  a reduced representation 
   $\vartheta:\pi^{-1}(U)\rightarrow V.$ If $\vartheta$ is a representation of $f$ on $\pi^{-1}(U),$ then the algebroid reduction  $\vartheta_\pi: U\rightarrow \odot^\nu V$ of $\vartheta$ via $\pi$  is a representation of $f_\pi.$ If $\vartheta$ is reduced, then $\vartheta_\pi$ is also reduced. 
 \end{theorem} 
According to Theorem \ref{S3}, we obtain  the  following 
   commutative diagram:
    $$ \xymatrix{
V\ar[d]_{\mathbb P} &\pi^{-1}(U)\ar[l]_{\vartheta} \ar@{_(->}[d]_{\rotatebox{90}{$\supseteq$}} \ar[r]^{\pi}  & U\ar@{^(->}[d]^{\rotatebox{270}{$\subseteq$}} \ar[r]^{\vartheta_\pi} & \odot^\nu V 
\ar[d]^{\mathbb P}  \\
\mathbb P(V) & N\ar[l]^{f} \ar[r]_{\pi} & M\ar[r]_{f_\pi}  &  \mathbb P(\odot^\nu V)
}
$$ 
   
  \begin{theorem}[Stoll, \cite{Stoll}]\label{S2}  Let 
 $g: M\rightarrow \mathbb P(V)^{(\nu)}$ $(\nu\geq2)$ be a  differentiably non-degenerate  and irreducible meromorphic mapping. Then,  there exist a non-compact  complex manifold $N$ of 
 complex dimension $m,$ a proper surjective holomorphic mapping $\pi: N\rightarrow M$ of  generic sheet number $\nu,$ and a holomorphic  mapping $f: N\rightarrow\mathbb P(V)$ which separates the fibers of $\pi$ such that $g$ is the algebroid reduction of $f$ via $\pi.$ 
  \end{theorem}  
  
  \subsection{Uniformization}~

In what follows, we consider the case when $\dim_{\mathbb C}V=2.$ For  convenience,  we write $v^{\odot k}=v^{k}$  and $[v]^{\odot k}=[v]^{k}$ for short.
Take a basis $\{e_0, e_1\}$ of $V.$ 
For  any $h\in \odot^\nu V,$  there is  a unique  representation 
\begin{equation}\label{eee}
h=\sum_{j=0}^\nu a_je_0^{\nu-j}\odot e^j_1
\end{equation}
for  certain    $a_0,\cdots,a_\nu\in\mathbb C.$
It induces  a natural  isomorphism $\imath:  \odot^\nu V\rightarrow \mathbb C^{\nu+1}$ defined  by 
$$\imath: \ \sum_{j=0}^\nu a_je_0^{\nu-j}\odot e^j_1\mapsto (a_0,\cdots, a_n),$$
which follows   that  
$\mathbb P(\odot^\nu V)\cong \mathbb P^{\nu}(\mathbb C).
$ If  $h\not=0,$ then  (\ref{eee}) defines an algebraic equation
\begin{equation}\label{equa1}
a_\nu w^\nu+a_{\nu-1}w^{\nu-1}+\cdots+a_0=0.
\end{equation}
Let $k$ be the maximal integer in $\{0, \cdots, \nu\}$ such that $A_k\not=0.$ Then, equation (\ref{equa1}) reduces to 
$a_k w^k+a_{k-1}w^{k-1}+\cdots+a_0=0.$
 Let $w_1,\cdots, w_k$ be  the  roots of  the  equation.  Applying Vieta's theorem, it is immediate that  
$$h=a_ke_0^{\nu-k}\odot(e_1-w_1e_0)\odot\cdots\odot(e_1-w_ke_0).$$
So,  $[h]$ is decomposable, i.e., 
\begin{equation*}
[h]=[e_0]^{\nu-k}\odot[e_1-w_1e_0]\odot\cdots\odot[e_1-w_ke_0]\in \mathbb P(V)^{(\nu)}.
\end{equation*}
By this with (\ref{po2}),  we have  
$\mathbb P(\odot^\nu V)\cong\mathbb P(V)^{(\nu)}.$  Hence,  we obtain 
\begin{equation}\label{proj}
\mathbb P(\odot^\nu V)\cong \mathbb P^{\nu}(\mathbb C)\cong \mathbb P(V)^{(\nu)}.
\end{equation}

  \begin{lemma}\label{SS3}  Let 
 $g: M\rightarrow \mathbb P^\nu(\mathbb C)$ be a  differentiably non-degenerate  and irreducible meromorphic mapping. Then,  there exist a non-compact  complex manifold $N$ of 
 complex dimension $m,$ a proper surjective holomorphic mapping $\pi: N\rightarrow M$ of  generic sheet number $\nu,$ and a meromorphic function $f: N\rightarrow\mathbb P^1(\mathbb C)$ which separates the fibers of $\pi$ such that $g$ is the algebroid reduction of $f$ via $\pi.$ 
  \end{lemma}  
\begin{proof}
Setting  $V=\mathbb C^2$ in Theorem \ref{S3},  then we have  the lemma  proved   due to  (\ref{proj}). 
\end{proof}

 \begin{theorem}[Uniformization]\label{D1} Let $w: M\to\mathbb P^1(\mathbb C)$ be an algebroid function.
 Then, there exist a non-compact  complex manifold $N$ of  complex dimension  $m,$  a proper surjective holomorphic mapping $\pi: N\rightarrow M$ of  generic sheet number $\nu,$ and a meromorphic function $f: N\to \mathbb P^1(\mathbb C)$  such that $w=f\circ\pi^{-1}.$
    \end{theorem}
\begin{proof} Let $w=\{w_j\}_{j=1}^\nu$ be defined by  equation (\ref{eq}). It  gives   a differentiably non-degenerate,  irreducible  meromorphic mapping 
$$\mathscr A=[A_0: \cdots : A_\nu]: \  M\to\mathbb P^{\nu}(\mathbb C).$$
 Set $V=\mathbb C^2$ with a basis  $e_0, e_1.$ By (\ref{proj}),   we  can write
$$\mathscr A= \bigg[\sum_{j=0}^\nu A_j e_0^{\nu-j}\odot e_1^j\bigg]$$
 in the sense of isomorphisms. 
In further,  the Vieta's theorem  leads to 
  \begin{eqnarray*}
\mathscr A&=& [ A_\nu(e_1-w_1e_0)\odot\cdots\odot [(e_1-w_\nu e_0)] \\
&=& [(e_1-w_1e_0)]\odot\cdots\odot [(e_1-w_\nu e_0)] \\
&=& [w_1]\odot\cdots\odot[w_\nu]
  \end{eqnarray*}
 in the sense of isomorphisms. 
On the other hand,   Lemma \ref{SS3}  says that  there exist a non-compact complex manifold $N$  of complex dimension $m,$  a proper surjective holomorphic mapping $\pi: N\rightarrow M$ of generic sheet number $\nu,$ and a holomorphic mapping $f: N\rightarrow\mathbb P^1(\mathbb C)$ which separates the fibers of $\pi$ such that $\mathscr A$ is the algebroid reduction of $f$ via $\pi.$ 
Hence, we have
 $$\mathscr A(x)=\bigodot_{y\in\pi^{-1}(x)}[f(y)]^{\lambda_\pi(y)}, \ \ \ \  ^\forall x\in M\setminus \pi(\mathscr E_\pi)$$
with
 $$\sum_{y\in\pi^{-1}(x)}\lambda_\pi(y)=\nu.$$
Let $y_1,\cdots,y_\nu$ be  the $\nu$  $x$-valued points of $\pi.$ Then 
  \begin{eqnarray*}
\mathscr A(x)=[f(y_1)]\odot\cdots\odot [f(y_\nu)] 
= [w_1]\odot\cdots\odot[w_\nu].
  \end{eqnarray*}
Combining  $w(x)=\{w_j(x)\}_{j=1}^\nu$
with 
  $f\circ\pi^{-1}(x)=\{f(y_j)\}_{j=1}^\nu,$  
then we prove  the theorem. 
\end{proof}

\section{Nevanlinna Theory of Meromorphic Functions}\label{sec40}

  Let $(M, g)$ be a complete non-compact  K\"ahler manifold of $\dim_{\mathbb C}M=m,$ with Laplace-Beltrami operator $\Delta.$ 
 The  K\"ahler  form of  $M$ is defined by  
$$\alpha=\frac{\sqrt{-1}}{\pi}\sum_{i,j=1}^mg_{i\bar{j}}dz_i\wedge d\bar{z}_j$$ 
in  a local holomorphic coordinate $z=(z_1,\cdots,z_m).$ Moreover,  the Ricci form of $M$  is defined by  $$\mathscr R=-dd^c\log\det(g_{i\bar j}).$$ 

\subsection{First Main Theorem}~\label{sec41}

   Let   $\Omega$ be any   precompact domain with smooth boundary $\partial\Omega$ in $M.$ Fix a  point $o\in \Omega$ and denote by $g_{\Omega}(o, x)$ the  Green function of $\Delta/2$ for $\Omega$ with a pole at $o$ satisfying Dirichlet boundary condition, i.e., 
  $$-\frac{1}{2}\Delta g_{\Omega}(o, x)=\delta_o(x), \ \ \ \  ^\forall x\in\Omega; \ \ \ \  \  g_{\Omega}(o, x)=0, \ \ \ \ ^\forall x\in\partial\Omega,$$
  where $\delta_o$ is the Dirac's delta function  on $M$ with a pole at $o.$  The harmonic measure $\pi_{\partial\Omega}$ on $\partial\Omega$ with respect to $o$ is defined by 
$$d\pi_{\partial\Omega}(x)=\frac{1}{2}\frac{\partial  g_{\Omega}(o,x)}{\partial \vec \nu}d\sigma_{\partial\Omega}(x),$$
where $d\sigma_{\partial\Omega}$ is the Riemannian volume element of $\partial\Omega$ and $\partial/\partial \vec{\nu}$ is the inward normal derivative on $\partial\Omega.$  
  
Next, we introduce   Nevanlinna's functions.   Let $f=[f_0: f_1]: M\to\mathbb P^1(\mathbb C)$ be a meromorphic function. 
Denote by $\omega_{FS}=dd^c\log\|w\|^2$ the Fubini-Study form on $\mathbb P^1(\mathbb C),$ where  $w=[w_0: w_1]$ is the homogeneous coordinate of $\mathbb P^1(\mathbb C)$ with $\|w\|^2=|w_0|^2+|w_1|^2$ and   $dd^c=(\sqrt{-1}/2\pi)\partial\overline{\partial}$ with 
$$d=\partial+\overline{\partial}, \ \  \  \ d^c=\frac{\sqrt{-1}}{4\pi}(\overline{\partial}-\partial).$$ 
The \emph{characteristic function} of $f$ on $\Omega$ with respect to $\omega_{FS}$ is defined by
  \begin{eqnarray*}
  T_f(\Omega, \omega_{FS})&=& \frac{\pi^m}{(m-1)!}\int_{\Omega}g_{\Omega}(o,x)f^*\omega_{FS}\wedge\alpha^{m-1} 
    \end{eqnarray*}
  Let $a=[a_0:a_1]\in\mathbb P^1(\mathbb C).$  The \emph{proximity  function} of $f$ on $\Omega$ with respect to $a$ is defined by
$$  m_f(\Omega, a)=\int_{\partial\Omega}\log\frac{1}{\|f, a\|}d\pi_{\partial\Omega},
$$
  where  $\|\cdot  , \cdot \|$ denotes  the spherical distance on $\mathbb P^1(\mathbb C)$ defined by
 $$\|f, a\|=\frac{|\langle f; a \rangle|}{\|a\|\|f\|}=\frac{|\langle f; a \rangle|}{\sqrt{|a_0|^2+|a_1|^2}\sqrt{|f_0|^2+|f_1|^2}}$$
 with
 $$\langle f; a \rangle=a_0f_1-a_1f_0.$$
   The \emph{counting  function} of $f$ on $\Omega$ with respect to $a$ is defined by
 \begin{eqnarray*}
  N_f(\Omega, a)&=&\frac{\pi^m}{(m-1)!}\int_{f^*a\cap\Omega}g_{\Omega}(o,x)\alpha^{m-1}. 
    \end{eqnarray*}
  \ \ \ \   We need the following Dynkin formula (see, e.g., \cite{at3, Dong, DY}). 
    \begin{lemma}[Dynkin Formula]\label{dynkin} Let $u$ be a $\mathscr C^2$-class function on  $M$ outside a polar set of singularities at most. Assume that $u(o)\not=\infty.$  Then
$$\int_{\partial \Omega}u(x)d\pi_{\partial \Omega}(x)-u(o)=\frac{1}{2}\int_{\Omega}g_\Omega(o,x)\Delta u(x)dv(x).$$
\end{lemma}
\begin{theorem}[First Main Theorem]\label{first}  Assume that $f(o)\not=a.$ Then  
$$T_f(\Omega, \omega_{FS})+\log\frac{1}{\|f(o), a\|}=m_f(\Omega, a)+N_f(\Omega, a).$$
\end{theorem}
\begin{proof} Applying Dynkin formula to $\log\|f, a\|,$  we get  
$$\int_{\partial\Omega}\log\|f, a\|d\pi_{\partial\Omega}+\log\frac{1}{\|f(o),a\|}=\frac{1}{2}\int_{\Omega}g_{\Omega}(o,x)\Delta\log\|f, a\|dv.$$
Note that the first term on the left hand side equals $-m_f(\Omega, a)$ and the right hand side equals 
 \begin{eqnarray*}
 \frac{1}{4}\int_{\Omega}g_{\Omega}(o,x)\Delta\log\|f, a\|^2dv 
 &=& \frac{\pi^m}{(m-1)!}\int_{\Omega}g_{\Omega}(o,x)dd^c\left[\log\|f, a\|^2\right]\wedge\alpha^{m-1}  \\
  &=&\frac{\pi^m}{(m-1)!}\int_{\Omega}g_{\Omega}(o,x)dd^c\left[\log|\langle f; a \rangle|^2\right]\wedge\alpha^{m-1} \\
 &&  -\frac{\pi^m}{(m-1)!}\int_{\Omega}g_{\Omega}(o,x)f^*\omega_{FS}\wedge\alpha^{m-1} \\
 &=& \frac{\pi^m}{(m-1)!}\int_{f^*a\cap\Omega}g_{\Omega}(o,x)\alpha^{m-1}-T_f(\Omega, \omega_{FS}) \\
 &=& N_f(\Omega, a)-T_f(\Omega, \omega_{FS}).
    \end{eqnarray*}
    This proves the theorem. 
\end{proof}

\subsection{Second Main Theorem for Non-negative Ricci Curvature}~

 Assume   that  $M$ is  non-parabolic. 
      Fix a  reference point $o\in M$ and  denote by $B(r)$    the geodesic ball  centered at $o$ with radius  $r$ in $M.$
   Let    $V(r)$ denote  the Riemannian volume of  $B(r),$   The non-parabolicity of $M$ means  that
 $$\int_1^\infty\frac{t}{V(t)}dt<\infty.$$
Thus, we have the unique minimal positive global Green function $G(o,x)$ of $\Delta/2$  for $M.$  
We can write $G(o,x)$ as 
$$G(o, x)=2\int_0^\infty p(t, o,x)dt,$$
in which  $p(t,o,x)$ is the heat kernel of $\Delta$ on $M.$
  Let $\rho(x)$ be the Riemannian distance function of $x$ from $o.$  By  Li-Yau's estimate \cite{Li-Yau}, there exist constants $A, B>0$ such that  
 \begin{equation}\label{Gr}
 A\int_{\rho(x)}^\infty\frac{t}{V(t)}dt\leq G(o,x)\leq B\int_{\rho(x)}^\infty\frac{t}{V(t)}dt
 \end{equation}
holds for all $x\in M.$ 
\subsubsection{Construction of $\Delta(r)$ and Nevanlinna's functions}~\label{sec421}

 For $r>0,$ define   
 $$\Delta(r)=\left\{x\in M: \    G(o,x)>A\int_r^\infty\frac{t}{V(t)}dt\right\}.$$
Since    
$$\lim_{x\to o}G(o,x)=\infty, \ \ \ \     \lim_{x\to\infty}G(o,x)=0,$$  we see   that 
    $\Delta(r)$ is a precompact  domain containing $o,$  with   
 $ \lim_{r\to0}\Delta(r)\to \emptyset$ and  $\lim_{r\to\infty}\Delta(r)=M.$
Moreover,  the family 
     $\{\Delta(r)\}_{r>0}$ can exhaust $M,$ i.e.,  for any   sequence $\{r_n\}_{n=1}^\infty$ with    $0<r_1<r_2<\cdots\to \infty,$ 
     we have       
 $$\bigcup_{n=1}^\infty\Delta(r_n)=M, \ \ \ \ 
\emptyset\not=\Delta(r_1)\subseteq\overline{\Delta(r_1)}\subseteq\Delta(r_2)\subseteq\overline{\Delta(r_2)}\subseteq\cdots$$    
So,  the boundary $\partial\Delta(r)$ of $\Delta(r)$ can be formulated as
 $$\partial\Delta(r)=\left\{x\in M: \    G(o,x)=A\int_r^\infty\frac{t}{V(t)}dt\right\}.$$
  By   Sard's theorem,   $\partial\Delta(r)$  is a submanifold of $M$ for almost all $r>0.$  
    
     Set
 $$g_r(o,x)=G(o,x)-A\int_r^\infty\frac{t}{V(t)}dt.$$
Note that   $g_r(o,x)$ defines   the  Green function of $\Delta/2$ for $\Delta(r)$ with a pole at $o$ satisfying Dirichelet boundary condition, i.e., 
  $$-\frac{1}{2}\Delta g_r(o,x)=\delta_o(x), \ \ \ \  ^\forall x\in\Delta(r);  \ \ \  \   \  g_r(o,x)=0, \ \ \  \ ^\forall x\in\partial\Delta(r).$$
 Let  $\pi_r$ denote the harmonic measure  on $\partial\Delta(r)$ with respect to $o,$ defined by
  $$d\pi_r(x)=\frac{1}{2}\frac{\partial g_r(o,x)}{\partial{\vec{\nu}}}d\sigma_r(x), \ \ \ \   ^\forall x\in\partial\Delta(r),$$
  where  $\partial/\partial \vec\nu$ is the inward  normal derivative on $\partial \Delta(r),$ $d\sigma_{r}$ is the Riemannian area element of 
$\partial \Delta(r).$

Let $f: M\to \mathbb P^1(\mathbb C)$ be a meromorphic function.  Take  $\Omega=\Delta(r)$ in Section \ref{sec41},   we  have  
the  Nevanlinna's functions of $f$: 
  \begin{eqnarray*}
  T_f(r, \omega_{FS})&=& \frac{\pi^m}{(m-1)!}\int_{\Delta(r)}g_{r}(o,x)f^*\omega_{FS}\wedge\alpha^{m-1}, \\
  m_f(r, a)&=&\int_{\partial \Delta(r)}\log\frac{1}{\|f, a\|}d\pi_{r}, \\
    N_f(r, a)&=&\frac{\pi^m}{(m-1)!}\int_{f^*a\cap \Delta(r)}g_{r}(o,x)\alpha^{m-1}, \\
        \overline{N}_f(r, a)&=&\frac{\pi^m}{(m-1)!}\int_{f^{-1}(a)\cap \Delta(r)}g_{r}(o,x)\alpha^{m-1}.
    \end{eqnarray*} 

\subsubsection{Gradient estimate of Green functions}~

   \begin{lemma}\label{thm4} For $x\in\partial\Delta(t),$ we have
 $$g_r(o,x)=A\int_t^r \frac{s}{V(s)}ds$$
holds  for all $0<t\leq r,$ where $A$ is given by $(\ref{Gr}).$
  \end{lemma}
 \begin{proof}    Invoking  the definition of Green function for $\Delta(r),$  it is immediate to deduce  that  for $0<t\leq r$
  \begin{eqnarray*} 
g_r(o,x)&=& G(o,x)-A\int_{r}^\infty\frac{t}{V(t)}dt   \\
&=& G(o,x)-A\int_{t}^\infty\frac{s}{V(s)}ds +A\int_{t}^r\frac{s}{V(s)}ds  \\
 &=& g_t(o,x)+A\int_{t}^r\frac{s}{V(s)}ds.   
   \end{eqnarray*}
 Since
 $g_t(o,x)=0$ for  $x\in\partial\Delta(t),$
  we obtain 
 $$g_r(o,x)\big|_{\partial\Delta(t)}=A\int_t^r \frac{s}{V(s)}ds.$$
 \end{proof}

 Let $\nabla$ stand for the gradient operator on a Riemannian manifold.  Cheng-Yau \cite{C-Y} proved the following theorem.

 \begin{lemma}\label{CY}
 Let $N$ be a complete Riemannian manifold.  Let   $B(x_0, r)$ be a geodesic ball centered at $x_0\in N$ with radius $r.$ 
 Then, there exists  a constant $c_N>0$ depending only on the dimension of $N$ such that 
 $$\frac{\|\nabla u(x)\|}{u(x)}\leq \frac{c_N r^2}{r^2-d(x_0, x)^2}\left(|\kappa(r)|+\frac{1}{d(x_0, x)}\right)$$
 holds for any  non-negative  harmonic function $u$ on $B(x_0, r),$ 
 where $d(x_0, x)$ is the Riemannian distance between $x_0$ and $x,$ and $\kappa(r)$ is the lower bound of  Ricci curvature of  $B(x_0, r).$
 \end{lemma}

 \begin{theorem}\label{hh} There exists a constant  $c_1>0$   such that 
 $$\|\nabla g_r(o,x)\|\leq \frac{c_1}{r}\int_{r}^\infty\frac{tdt}{V(t)}, \ \ \ \    ^\forall x\in\partial\Delta(r).$$
\end{theorem}
\begin{proof} From the curvature assumption,   ${\rm{Ric}}_M\geq0.$  It yields   from Lemma \ref{CY} and (\ref{Gr})  that   (see Remark 5 in \cite{Sa0} also)
  \begin{eqnarray*}
\|\nabla G(o, x)\|\leq\frac{c_0}{\rho(x)}G(o,x) 
\leq \frac{c_0B}{\rho(x)}\int_{\rho(x)}^\infty\frac{tdt}{V(t)}
  \end{eqnarray*}
 for some large  constant $c_0>0$ which depends  only on  the dimension  $m.$  By  (\ref{Gr})   again 
 $$\int_{\rho(x)}^\infty\frac{tdt}{V(t)}\leq \int_r^\infty \frac{tdt}{V(t)}, \ \ \  \ ^\forall x\in \partial\Delta(r),$$
 which  gives    
 $$\rho(x)\geq r, \ \ \ \     ^\forall x\in\partial\Delta(r).$$
Hence,  we  conclude  that   
$$\|\nabla g_r(o, x)\|\big|_{\partial\Delta(r)}=\|\nabla G(o, x)\|\big|_{\partial\Delta(r)}\leq \frac{c_1}{r}\int_{r}^\infty\frac{tdt}{V(t)},
$$
where   $c_1=c_0B.$
  \end{proof}
   
   Applying 
   $$\frac{\partial g_r(o,x)}{\partial{\vec{\nu}}}=\|\nabla g_r(o, x)\|,$$
we  obtain   an   upper bound estimate of  $d\pi_r$  as follows. 
  \begin{cor}\label{bbbq}  There exists a constant  $c_2>0$   such that 
   $$d\pi_r\leq \frac{c_2}{r}\int_{r}^\infty\frac{tdt}{V(t)}d\sigma_r,$$
 where  $d\sigma_{r}$ is the Riemannian area element of 
$\partial \Delta(r).$
\end{cor}

\subsubsection{Calculus Lemma and Logarithmic Derivative Lemma}~

 We need the following Borel's growth lemma (see  \cite{No, ru}). 
 \begin{lemma}[Borel's Growth  Lemma]\label{} Let $u\geq0$ be a non-decreasing  function  on $(r_0, \infty)$ with $r_0\geq0.$ Then for any $\delta>0,$ there exists a subset $E_\delta\subseteq(r_0,\infty)$
 of finite Lebesgue measure such that  
 $$u'(r)\leq u(r)^{1+\delta}$$
 holds for all $r>r_0$ outside $E_{\delta}.$  
 \end{lemma}
 \begin{proof} The conclusion is clearly true  for $u\equiv0.$ Next, we assume that $u\not\equiv0.$
 Since $u\geq0$ is a non-decreasing  function, then there exists a number $r_1>r_0$ such that $u(r_1)>0.$  The non-decreasing property of $u$ implies that  the  limit
 $\eta:=\lim_{r\to\infty}u(r)$
exists or equals $\infty.$  If $\eta=\infty,$ then $\eta^{-1}=0.$ 
 Set  
 $$E_\delta=\left\{r\in(r_0,\infty): \  u'(r)>u(r)^{1+\delta}\right\}.$$
Note that  $u'(r)$ exists for  almost all  $r\in(r_0, \infty).$  Then, we have  
   \begin{eqnarray*}
\int_{E_\delta}dr 
 &\leq& \int_{r_0}^{r_1}dr+\int_{r_1}^\infty\frac{u'(r)}{u(r)^{1+\delta}}dr \\
 &=&\frac{1}{\delta u(r_1)^\delta}-\frac{1}{\delta \eta^\delta}+r_1-r_0 \\
&<&\infty.
    \end{eqnarray*}
 This completes the proof. 
 \end{proof}
 
 Set
 \begin{equation}\label{Hr1}
H(r,\delta)=\frac{1}{r}\left(\frac{V(r)}{r}\right)^{1+\delta}\int_{r}^\infty\frac{tdt}{V(t)}. 
 \end{equation} 

 \begin{theorem}[Calculus Lemma]\label{calculus}
Let $k\geq0$ be a locally integrable function on $M.$ Assume that $k$ is locally bounded at $o.$ Then there exists a  constant $C>0$ such that
 for any $\delta>0,$ there exists  a subset $E_{\delta}\subseteq(0,\infty)$ of finite Lebesgue measure such that
$$\int_{\partial\Delta(r)}kd\pi_r\leq CH(r,\delta)\bigg(\int_{\Delta(r)}g_r(o,x)kdv\bigg)^{(1+\delta)^2}$$
holds for all $r>0$ outside $E_{\delta},$ where $H(r,\delta)$ is given by $(\ref{Hr1}).$  
\end{theorem}
 
  \begin{proof} 
 Invoking  Lemma  \ref{thm4}, we have   
   \begin{eqnarray*}
 \int_{\Delta(r)}g_r(o,x)kdv 
 &=&\int_0^rdt \int_{\partial\Delta(t)}g_r(o,x)kd\sigma_t \\
    &=&  A\int_0^r\left(\int_{t}^r\frac{s}{V(s)}ds\right)dt \int_{\partial\Delta(t)}kd\sigma_t.
   \end{eqnarray*}
 Set 
 $$\Lambda(r)=A\int_0^{r}\left(\int_{t}^{r}\frac{s}{V(s)}ds\right)dt \int_{\partial\Delta(t)}kd\sigma_t.$$
 A simple computation leads to  
    \begin{eqnarray*}
 \Lambda'(r)&=&\frac{d\Lambda(r)}{dr} 
 =\frac{Ar}{V(r)}\int_0^rdt\int_{\partial\Delta(t)}kd\sigma_t.
    \end{eqnarray*}
 In further, we have
 $$\frac{d}{dr}\left(\frac{V(r)\Lambda'(r)}{r}\right)=A\int_{\partial\Delta(r)}kd\sigma_r.$$
Applying  Borel's growth lemma  to the left hand side of this  equality  twice:  one is to $V(r)\Lambda'(r)/r$ and  another   is to 
 $\Lambda'(r),$   we deduce  that for any  $\delta>0,$   there exists a subset $E_\delta\subseteq(0,\infty)$ of finite Lebesgue measure such that 
$$ \int_{\partial\Delta(r)}kd\sigma_r\leq \frac{1}{A}\left(\frac{V(r)}{r}\right)^{1+\delta}\Lambda(r)^{(1+\delta)^2}
$$ holds for all $r>0$ outside $E_\delta.$  
  On the other hand,  Corollary \ref{bbbq} implies that there exists a constant $c>0$ such that
  $$d\pi_r\leq \frac{c}{r}\int_{r}^\infty\frac{tdt}{V(t)}d\sigma_r.$$
Set $C=c/A.$ Combining the above, we have 
   \begin{eqnarray*}
 \int_{\partial\Delta(r)}kd\pi_r&\leq&\frac{c}{Ar}\left(\frac{V(r)}{r}\right)^{1+\delta}\int_{r}^\infty\frac{tdt}{V(t)}\Lambda(r)^{(1+\delta)^2} \\
 &=&CH(r,\delta)\Lambda(r)^{(1+\delta)^2}.
   \end{eqnarray*}
 holds for all $r>0$ outside $E_\delta.$  
 \end{proof}

In the following,  we establish  a logarithmic derivative lemma. 
Let $\psi$ be a meromorphic function on $M.$
The norm of the gradient of $\psi$ is defined by
$$\|\nabla\psi\|^2=2\sum_{i,j=1}^m g^{i\overline j}\frac{\partial\psi}{\partial z_i}\overline{\frac{\partial \psi}{\partial  z_j}}$$
in  a local holomorphic   coordinate $z=(z_1,\cdots,z_m),$ where $(g^{i\overline{j}})$ is the inverse of $(g_{i\overline{j}}).$
Define the Nevanlinna's characteristic function of $\psi$ by  
$$T(r,\psi)=m(r,\psi)+N(r,\psi),$$
where    
\begin{eqnarray*}
m(r,\psi)&=&\int_{\partial\Delta(r)}\log^+|\psi|d\pi_r, \\
N(r,\psi)&=& \frac{\pi^m}{(m-1)!}\int_{\psi^*\infty\cap \Delta(r)}g_r(o,x)\alpha^{m-1}.
 \end{eqnarray*}
It is not difficult  to show   that 
 \begin{equation*}\label{goed}
T\Big(r,\frac{1}{\psi-\zeta}\Big)= T(r,\psi)+O(1).
 \end{equation*}
   On $\mathbb P^1(\mathbb C),$ one puts a singular metric
$$\Psi=\frac{1}{|\zeta|^2(1+\log^2|\zeta|)}\frac{\sqrt{-1}}{4\pi^2}d\zeta\wedge d\bar \zeta$$
with   
$\int_{\mathbb P^1(\mathbb C)}\Psi=1.$
\begin{lemma}\label{oo12} We have
$$\frac{1}{4\pi}\int_{\Delta(r)}g_r(o,x)\frac{\|\nabla\psi\|^2}{|\psi|^2(1+\log^2|\psi|)}dv\leq T(r,\psi)+O(1).$$
\end{lemma}
\begin{proof}  Note that  
$$\frac{\|\nabla\psi\|^2}{|\psi|^2(1+\log^2|\psi|)}=4m\pi\frac{\psi^*\Psi\wedge\alpha^{m-1}}{\alpha^m}.$$
It follows   from 
 Fubini's theorem that 
\begin{eqnarray*}
&& \frac{1}{4\pi}\int_{\Delta(r)}g_r(o,x)\frac{\|\nabla\psi\|^2}{|\psi|^2(1+\log^2|\psi|)}dv \\ 
&=&m\int_{\Delta(r)}g_r(o,x)\frac{\psi^*\Psi\wedge\alpha^{m-1}}{\alpha^m}dv  \\
&=&\frac{\pi^m}{(m-1)!}\int_{\mathbb P^1(\mathbb C)}\Psi(\zeta)\int_{\psi^*\zeta\cap \Delta(r)}g_r(o,x)\alpha^{m-1} \\
&=&\int_{\mathbb P^1(\mathbb C)}N\Big(r, \frac{1}{\psi-\zeta}\Big)\Psi(\zeta). 
\end{eqnarray*}
By this with the First Main Theorem, we conclude that 
\begin{eqnarray*}
\frac{1}{4\pi}\int_{\Delta(r)}g_r(o,x)\frac{\|\nabla\psi\|^2}{|\psi|^2(1+\log^2|\psi|)}dv 
&\leq&\int_{\mathbb P^1(\mathbb C)}\big{(}T(r,\psi)+O(1)\big{)}\Psi \\
&=& T(r,\psi)+O(1). 
\end{eqnarray*}
\end{proof}

\begin{lemma}\label{999a}  Let
$\psi\not\equiv0$ be a  meromorphic function on  $M.$  Then for any $\delta>0,$ there exists a subset 
 $E_\delta\subseteq(0,\infty)$ of finite Lebesgue measure such that
  \begin{eqnarray*}
&&  \int_{\partial\Delta(r)}\log^+\frac{\|\nabla\psi\|^2}{|\psi|^2(1+\log^2|\psi|)}d\pi_r \\
  &\leq&(1+\delta)^2\log^+ T(r,\psi)+\log H(r,\delta)+O(1)
\end{eqnarray*}
 holds for all $r>0$ outside  $E_\delta,$   where $H(r,\delta)$ is given by $(\ref{Hr1}).$   
\end{lemma}
\begin{proof} The concavity of $\log$ implies that    
\begin{eqnarray*}
&& \int_{\partial\Delta(r)}\log^+\frac{\|\nabla\psi\|^2}{|\psi|^2(1+\log^2|\psi|)}d\pi_r \\
   &\leq&   \log\int_{\partial\Delta(r)}\bigg(1+\frac{\|\nabla\psi\|^2}{|\psi|^2(1+\log^2|\psi|)}\bigg)d\pi_r \\
    &\leq&  \log^+\int_{\partial\Delta(r)}\frac{\|\nabla\psi\|^2}{|\psi|^2(1+\log^2|\psi|)}d\pi_r+O(1). \nonumber
\end{eqnarray*}
By  this with  Theorem \ref{calculus} and Lemma \ref{oo12}, we conclude  that  for any $\delta>0,$ there exists a subset 
 $E_\delta\subseteq(0,\infty)$ of finite Lebesgue measure such that 
\begin{eqnarray*}
   && \log^+\int_{\partial\Delta(r)}\frac{\|\nabla\psi\|^2}{|\psi|^2(1+\log^2|\psi|)}d\pi_r\\
   &\leq& (1+\delta)^2 \log^+\int_{\Delta(r)}g_r(o,x)\frac{\|\nabla\psi\|^2}{|\psi|^2(1+\log^2|\psi|)}dv 
    +\log H(r,\delta)+O(1) \\
   &\leq& (1+\delta)^2 \log^+T(r,\psi)+\log H(r,\delta)+O(1)
\end{eqnarray*}
holds  for all $r>0$ outside  $E_\delta.$  
\end{proof}

Define
$$m\left(r,\frac{\|\nabla\psi\|}{|\psi|}\right)=\int_{\partial\Delta(r)}\log^+\frac{\|\nabla\psi\|}{|\psi|}d\pi_r.$$

\begin{theorem}[Logarithmic Derivative  Lemma]\label{log1} Let
$\psi\not\equiv0$ be a  meromorphic function on  $M.$   Then for any $\delta>0,$ there exists a  subset  $E_\delta\subseteq(0,\infty)$ of  finite Lebesgue measure such that 
\begin{eqnarray*}
   m\Big(r,\frac{\|\nabla\psi\|}{|\psi|}\Big)&\leq& \frac{2+(1+\delta)^2}{2}\log^+ T(r,\psi)+\frac{1}{2}\log H(r,\delta)+O(1)
\end{eqnarray*}
 holds for all $r>0$ outside  $E_\delta,$  where $H(r,\delta)$ is given by $(\ref{Hr1}).$  
\end{theorem}
\begin{proof} Note that
\begin{eqnarray*}
    m\left(r,\frac{\|\nabla\psi\|}{|\psi|}\right)  
   &\leq& \frac{1}{2}\int_{\partial\Delta(r)}\log^+\frac{\|\nabla\psi\|^2}{|\psi|^2(1+\log^2|\psi|)}d\pi_r \ \ \  \  \    \   \  \  \    \    \   \   \\ 
 &&   +\frac{1}{2}\int_{\partial\Delta(r)}\log\left(1+\log^2|\psi|\right)d\pi_r \\
  &=& \frac{1}{2}\int_{\partial\Delta(r)}\log^+\frac{\|\nabla\psi\|^2}{|\psi|^2(1+\log^2|\psi|)}d\pi_r \\
   && +\frac{1}{2}\int_{\partial\Delta(r)}\log\bigg(1+\Big{(}\log^+|\psi|+\log^+\frac{1}{|\psi|}\Big{)}^2\bigg)d\pi_r  \\
    &\leq&  \frac{1}{2}\int_{\partial\Delta(r)}\log^+\frac{\|\nabla\psi\|^2}{|\psi|^2(1+\log^2|\psi|)}d\pi_r \\
   && +\log\int_{\partial\Delta(r)}\Big{(}\log^+|\psi|+\log^+\frac{1}{|\psi|}\Big{)}d\pi_r +O(1)  \\
   &\leq& \frac{1}{2}\int_{\partial\Delta(r)}\log^+\frac{\|\nabla\psi\|^2}{|\psi|^2(1+\log^2|\psi|)}d\pi_r+\log^+T(r,\psi)+O(1).
\end{eqnarray*}
By Lemma \ref{999a},  we have for any $\delta>0,$ there exists a subset 
 $E_\delta\subseteq(0,\infty)$ of finite Lebesgue measure such that 
 \begin{equation*}
 m\left(r,\frac{\|\nabla\psi\|}{|\psi|}\right)  \leq \frac{2+(1+\delta)^2}{2}\log^+ T(r,\psi)+\frac{1}{2}\log H(r,\delta)+O(1)
\end{equation*}
 holds for all $r>0$ outside  $E_\delta.$  
  \end{proof}

\subsubsection{Second Main  Theorem}~\label{sec424}

  The characteristic function of $\mathscr R$ is defined  by 
 $$  T(r, \mathscr R)= \frac{\pi^m}{(m-1)!}\int_{\Delta(r)}g_{r}(o,x)\mathscr R\wedge\alpha^{m-1}. 
$$
\begin{theorem}[Second Main Theorem]\label{sec1}
 Let $f: M\rightarrow \mathbb P^1(\mathbb C)$ be a nonconstant   meromorphic function. Let $a_1, \cdots, a_q$ be distinct values in 
  $\mathbb P^1(\mathbb C).$ Then  for any $\delta>0,$ there exists a subset $E_\delta\subseteq(0, \infty)$ of finite Lebesgue measure such that 
     \begin{eqnarray*}
&& (q-2)T_f(r,\omega_{FS})+T(r, \mathscr R) \\
&\leq& \sum_{j=1}^q\overline N_f(r,a_j)+O\left(\log^+T_f(r,\omega_{FS})+\log H(r)+\delta\log r\right)
     \end{eqnarray*}
holds for all $r>0$ outside $E_\delta,$ where $H(r)$ is given by $(\ref{Hr}).$ 
\end{theorem}
\begin{proof}
  Define a non-negative function $\xi$ on $M$ by   
  $$f^*\Phi\wedge\alpha^{m-1}=\xi\alpha^m,$$
 where
    \begin{equation}\label{form}
\Phi=\frac{\omega_{FS}}{\prod_{j=1}^q\|w, a_j\|^2}
 \end{equation}
is a  singular volume  form  on $\mathbb P^1(\mathbb C).$  
 By ${\rm{Ric}}(\omega_{FS})=2\omega_{FS},$ we obtain 
   \begin{eqnarray}\label{be}
dd^c[\log\xi] &=& (q-2)f^*\omega_{FS}-\sum_{j=1}^q(f=a_j)+D_{f, {\rm{ram}}}  +\mathscr R \nonumber
  \end{eqnarray}
 in the sense of currents,  in which $(f=a_j)$ is the zero divisor of $f-a_j,$ and  $D_{f, {\rm{ram}}}$ is the ramification divisor of $f.$ 
Using  Dynkin formula, we have  
  \begin{eqnarray}\label{lower}
 && \frac{1}{2}\int_{\partial\Delta(r)}\log\xi d\pi_r \\
&=& (q-2)T_f(r, \omega_{FS})-\sum_{j=1}^qN_f(r, a_j)+N(r, D_{f, {\rm{ram}}})  
 +T(r,\mathscr R)+O(1) \nonumber \\
&\geq&  (q-2)T_f(r, \omega_{FS})-\sum_{j=1}^q\overline{N}_f(r, a_j)+
T(r, \mathscr R)+O(1). \nonumber
  \end{eqnarray}
 \ \ \  On the other hand,   there exist a finite  open  overing   $\{U_1, \cdots, U_q\}$ of $\mathbb P^1(\mathbb C)$ and 
  rational functions
$w_1,\cdots,w_q$ on $\mathbb P^1(\mathbb C)$   such that  $w_{\lambda}$ is holomorphic on $U_\lambda$  and  $dw_{\lambda}\neq0$ on $U_\lambda$ as well as  
$$\{a_1,\cdots,a_q\}\cap U_\lambda=\{a_\lambda\}=w_\lambda^{-1}(0)$$ 
for   $\lambda=1,\cdots,q.$   
 On each $U_\lambda,$   write 
$$\Phi=\frac{e_\lambda}{|w_{\lambda}|^2}\frac{\sqrt{-1}}{\pi}dw_{\lambda}\wedge d\bar w_{\lambda},$$
where  $e_\lambda$ is a  positive smooth function on $U_\lambda.$  
Given a partition $\{\phi_\lambda\}$ of the unity   subordinate to $\{U_\lambda\}.$ 
 Put 
  $$\Phi_\lambda=\frac{e_\lambda \phi_\lambda}{|w_{\lambda}|^2}\frac{\sqrt{-1}}{\pi}dw_{\lambda}\wedge d\bar w_{\lambda}.$$
Again, set $f_{\lambda}=w_{\lambda}\circ f.$  Then on  $f^{-1}(U_\lambda),$ we have  
\begin{eqnarray*}
 f^*\Phi_\lambda&=&
   \frac{\phi_{\lambda}\circ f\cdot e_\lambda\circ f}{|f_{\lambda}|^2}\frac{\sqrt{-1}}{\pi}df_{\lambda}\wedge d\bar f_{\lambda} \\
   &=& \frac{\phi_{\lambda}\circ f\cdot e_\lambda\circ f}{|f_{\lambda}|^2}\frac{\sqrt{-1}}{\pi}\sum_{i,j=1}^m\frac{\partial f_\lambda}{\partial z_i}
   \overline{\frac{\partial f_\lambda}{\partial z_j}}dz_i\wedge d\bar z_j.
 \end{eqnarray*}
 For any  $x_0\in M,$  we can take a  local holomorphic coordinate $z=(z_1,\cdots,z_m)$ near $x_0$ and  a local holomorphic coordinate
  $\zeta$ 
near $f(x_0)$ such that
$$\alpha|_{x_0}=\frac{\sqrt{-1}}{\pi}\sum_{j=1}^m dz_j\wedge d\bar{z}_j, \ \ \  \  
\omega_{FS}|_{f(x_0)}=\frac{\sqrt{-1}}{\pi}d\zeta\wedge d\bar{\zeta}.
$$
It yields  that 
\begin{eqnarray*}
&& f^*\Phi_\lambda\wedge\alpha^{m-1}\big|_{x_0} \\
&=&\frac{(m-1)!\phi_{\lambda}\circ f\cdot e_\lambda\circ f}{|f_{\lambda}|^2}\left(\frac{\sqrt{-1}}{\pi}\right)^m\sum_{j=1}^m\left|\frac{\partial f_\lambda}{\partial z_j}
\right|^2dz_1\wedge d\bar z_1\wedge\cdots \wedge dz_m\wedge d\bar z_m.
\end{eqnarray*}
 Put  
$f^*\Phi_\lambda\wedge\alpha^{m-1}=\xi_\lambda\alpha^m.$
Then, we obtain 
$$ \xi_\lambda
= \frac{\phi_{\lambda}\circ f\cdot e_\lambda\circ f}{m|f_{\lambda}|^2}\sum_{j=1}^m\left|\frac{\partial f_\lambda}{\partial z_j}
\right|^2= \frac{\phi_{\lambda}\circ f\cdot e_\lambda\circ f}{2m}\frac{\|\nabla f_\lambda\|^2}{|f_\lambda|^2}
$$ 
on $f^{-1}(U_\lambda).$
It is therefore  
$$\xi=\sum_{\lambda=1}^q \xi_\lambda=\sum_{\lambda=1}^q\frac{\phi_{\lambda}\circ f\cdot e_\lambda\circ f}{2m}\frac{\|\nabla f_\lambda\|^2}{|f_\lambda|^2}$$
on $M.$
Since $\phi_\lambda\circ f\cdot e_\lambda\circ f$ is bounded on $M,$ then it yields  that
\begin{equation*}
   \log^+\xi\leq  2\sum_{\lambda=1}^q\log^+\frac{\|\nabla f_{\lambda}\|}{|f_{\lambda}|}+O(1).
 \end{equation*}  
By this with Dynkin formula  and Theorem \ref{log1}, for any $\delta>0,$ there exists  a subset $E_\delta\subseteq(0,\infty)$ of finite Lebesgue measure such that 
\begin{eqnarray}\label{upper}
 \frac{1}{2}\int_{\partial\Delta(r)}\log\xi d\pi_r 
&\leq& \sum_{\lambda=1}^q\int_{\partial\Delta(r)}\log^+\frac{\|\nabla f_{\lambda}\|}{|f_{\lambda}|}d\pi_r+O(1) \nonumber \\
&\leq&  \frac{2+(1+\delta)^2}{2}\sum_{\lambda=1}^q\log^+T(r, f_\lambda)+O\big(\log H(r,\delta)\big) \nonumber \\
&=& O\big(\log^+T_f(r,\omega_{FS})+\log H(r,\delta)\big)  \nonumber \\
&\leq & O\big(\log^+T_f(r,\omega_{FS})+\log H(r)+\delta\log r\big) 
\end{eqnarray}
   holds for all $r>0$ outside $E_\delta,$ in which  we used the fact that  $V(r)\leq O(r^{2m}).$ 
Combining  (\ref{lower}) with (\ref{upper}),  we have the theorem proved. 
  \end{proof}

\subsection{Second Main Theorem for Non-positive Sectional  Curvature}~

 Assume   that  $M$ has   curvature satisfying (\ref{curvature}).
According to  the definition of curvature, it is not difficult  to see   that        
 $\tau\geq\sigma.$ In particular, we have $\sigma=\tau$ when $M$ has constant sectional curvature $\sigma.$
 Fix a reference point $o\in M.$ 
 We may   assume  without loss of generality  that  $o$ is a pole of $M$
   for a  technical reason   if necessary, 
   since, 
 otherwise, one  can   consider the universal  covering $\pi: \tilde M\to M.$
 If    $\tilde M$  is equipped  with  pullback metric $\tilde g=\pi^*g$  induced from $g$ via $\pi,$  
    then  $(\tilde M, \tilde g)$ preserves   the  same   curvature    as  $(M, g).$ 
  Note that     $(\tilde M, \tilde g)$ is a Cartan-Hadamard manifold. 
 In the following,  we use   $\rho(x)$  to denote  the Riemannian distance function of $x$ from $o$ on $M.$

\subsubsection{Construction of $\Delta(r)$ and Nevanlinna's functions}~

 \noindent\textbf{A. Setups  for $\tau\geq\sigma>0$}~

  By  the curvature conditions,   there exists  a unique minimal positive global Green function $G(o,x)$ of $\Delta/2$  
 with a pole at $o$ for $M.$ 
 Consider  the following  initial value  problems on $[0,\infty)$:
 $$G''-\tau^2G=0; \ \ \ \    G(0)=0, \ \ \ \  G'(0)=1$$
 and 
  $$H''-\sigma^2H=0; \ \ \ \    H(0)=0, \ \ \ \  H'(0)=1.$$
Evidently,  the two  ODEs  have  the unique solutions   
    \begin{equation*}
G(t)=\frac{\sinh\tau t}{\tau}, \ \ \ \     H(t)=\frac{\sinh\sigma t}{\sigma}. 
    \end{equation*}
Using the estimates of   $\Delta\rho,$  we have  (see \cite{Ka0, Sa0})
    \begin{equation}\label{tiancai}
\phi(\rho(x))\leq G(o,x)\leq\psi(\rho(x)),
    \end{equation}
where   
$$\phi(t)=\frac{2}{\omega_{2m-1}}\int_t^\infty G(s)^{1-2m}ds, \ \ \ \    \psi(t)=\frac{2}{\omega_{2m-1}}\int_t^\infty H(s)^{1-2m}ds.$$ 
 Here,  $\omega_{2m-1}:=2\pi^m/(m-1)!$ is  the standard  Euclidean  area  of    $\mathbb S^{2m-1}.$
For $r>0,$  define  
   \begin{eqnarray*}
  \Delta(r)&=& \big\{x\in M: \  G(o,x)>\phi(r)\big\} \\
  &=& \left\{x\in M: \  G(o,x)>\frac{2}{\omega_{2m-1}}\int_r^\infty \left(\frac{\sinh\tau t}{\tau}\right)^{1-2m}dt\right\}. 
     \end{eqnarray*}
  Since  
  $\phi(r)\searrow0$ as $r\to\infty,$ 
$\lim_{x\to o}G(o,x)=\infty$  and   $\lim_{x\to\infty}G(o,x)=0,$
  one infers     that  $\Delta(r)$ is a precompact  domain containing $o$ in $M,$ which     
exhausts $M.$ 
 Thus, the boundary of $\Delta(r)$ can be formulated as 
   \begin{eqnarray*}
  \partial\Delta(r)&=&  \left\{x\in M: \ G(o,x)=\frac{2}{\omega_{2m-1}}\int_r^\infty \left(\frac{\sinh\tau t}{\tau}\right)^{1-2m}dt\right\}. 
     \end{eqnarray*}
    Set  
        \begin{eqnarray*}
 g_r(o,x) 
 &=&G(o,x)-\frac{2}{\omega_{2m-1}}\int_r^\infty \left(\frac{\sinh\tau t}{\tau}\right)^{1-2m}dt,
     \end{eqnarray*}
 which defines   the Green function of $\Delta/2$ for $\Delta(r)$ with a pole at $o$ satisfying Dirichlet boundary condition.  
 Denote by  $\pi_r$  the harmonic measure  on $\partial\Delta(r)$ with respect to $o,$ i.e., 
  $$d\pi_r=\frac{1}{2}\frac{\partial g_r(o,x)}{\partial{\vec{\nu}}}d\sigma_r,$$
  where  $\partial/\partial \vec\nu$ is the inward  normal derivative on $\partial \Delta(r),$ $d\sigma_{r}$ is the Riemannian area element  of 
$\partial \Delta(r).$
 In particular,     $G(o,x)=G(\rho(x))=H(\rho(x))$ if  $\sigma=\tau.$ In this case,    we have 
     \begin{equation}\label{ext}
\Delta(r)= \big\{x\in M: \  \rho(x)<r\big\},
    \end{equation}
which is just  the geodesic ball centered at $o$ with radius $r$ in $M.$ Then  
        \begin{eqnarray}\label{gtt}
 g_r(o,x)&=& \frac{2}{\omega_{2m-1}}\int_{\rho(x)}^\infty \left(\frac{\sinh\tau t}{\tau}\right)^{1-2m}dt-\frac{2}{\omega_{2m-1}}\int_r^\infty \left(\frac{\sinh\tau t}{\tau}\right)^{1-2m}dt  \nonumber \\
 &=& \frac{2}{\omega_{2m-1}}\int_{\rho(x)}^r \left(\frac{\sinh\tau t}{\tau}\right)^{1-2m}dt. 
     \end{eqnarray}
\ \ \ \   Let $B(r)$ be the geodesic ball  centered at $o$ with radius $r$ in $M.$ We  give the  relation between  $\Delta(r)$ and $B(r)$ as follows. 
 
 Set 
 $$\rho_{r, \min}=\min_{x\in\partial\Delta(r)}\rho(x), \ \ \ \   \rho_{r, \max}=\max_{x\in\partial\Delta(r)}\rho(x).$$
  \begin{theorem}\label{lem22}   For $\tau\geq\sigma>0,$ 
  we have $\rho_{r,\min}\geq r$ and 
$$\limsup_{r\to\infty}\frac{\rho_{r,\max}}{r}\leq\frac{\tau}{\sigma}.$$
Hence, there exists a constant $\gamma\geq\tau/\sigma$ independent of $r$ such that 
$$B(r)\subseteq\Delta(r)\subseteq B(\gamma r).$$ 
\end{theorem} 
 \begin{proof}    
   It yields    from  (\ref{tiancai}) that  
 \begin{equation*}
\int_{r}^\infty \left(\frac{\sinh\tau t}{\tau}\right)^{1-2m}dt\geq 
\int_{\rho_{r,\min}}^\infty \left(\frac{\sinh\tau t}{\tau}\right)^{1-2m}dt, 
  \end{equation*}
 which leads to   $\rho_{r,\min}\geq r.$ 
 Again,  it yields from    (\ref{tiancai})   that 
 \begin{equation*}
\int_{r}^\infty\left(\frac{\sinh\tau t}{\tau}\right)^{1-2m}dt\leq
 \int_{\rho_{r,\max}}^\infty \left(\frac{\sinh\sigma t}{\sigma}\right)^{1-2m}dt. 
  \end{equation*}
Thus, we obtain 
 $$ \limsup_{r\to\infty}\frac{\displaystyle\int_{r}^\infty \left(\frac{\sinh\tau t}{\tau}\right)^{1-2m}dt}{\displaystyle\int_{\rho_{r,\max}}^\infty 
\left(\frac{\sinh\tau t}{\tau}\right)^{1-2m}dt} \leq1.$$
Moreover,  
       \begin{eqnarray*}
\int_{r}^\infty\left(\frac{\sinh\tau t}{\tau}\right)^{1-2m}dt&\geq& \int_{r}^\infty\left(\frac{e^{\tau t}}{2\tau}\right)^{1-2m}dt \\
&=&  \frac{2^{2m-1}\tau^{2m-2}}{2m-1}e^{(1-2m)\tau r}
       \end{eqnarray*}
for $r>0$  and
      \begin{eqnarray*}
\int_{\rho_{r,\max}}^\infty\left(\frac{\sinh\sigma t}{\sigma}\right)^{1-2m}dt&\leq&
 \int_{\rho_{r,\max}}^\infty\left(\frac{e^{\sigma t}}{4\sigma}\right)^{1-2m}dt \\
&=&  \frac{4^{2m-1}\sigma^{2m-2}}{2m-1}e^{(1-2m)\sigma \rho_{r,\max}}
       \end{eqnarray*}
for  $r>(2\sigma)^{-1}\log 2.$ Combining the above,   we get
      \begin{eqnarray*}
 2^{1-2m}\left(\frac{\tau}{\sigma}\right)^{2m-2}\limsup_{r\to\infty}e^{(2m-1)(\sigma\rho_{r,\max}-\tau r)}\leq 1, 
       \end{eqnarray*}
which  gives  
$$\limsup_{r\to\infty}\frac{\rho_{r,\max}}{r}\leq\frac{\tau}{\sigma}.$$
This completes the proof.
 \end{proof}

\noindent\textbf{B. Setups  for $\sigma=\tau=0$}~

We shall  extend   the settings for $\Delta(r), g_r(o,x)$ to the  case that     $\sigma=\tau=0.$ 
For $\tau=0,$   the function $\tau^{-1}\sinh \tau t$ does  not make sense. 
However, since  
$$\lim_{s\to0}\frac{\sinh s t}{s}=t,$$
  we  can extend $s^{-1}\sinh s t$ continuously to  $\chi(s,t)$   defined by (\ref{chi}).
             
        It is noted   that    $M$ is parabolic if $m=1,$ i.e.,  
         there exist no  positive global Green functions for $M.$ In other words, $G(o,x)\equiv\infty$ for $m=1.$ 
         Hence, one cannot use $G(o,x)$ to define the domain $\Delta(r)$ in this case.  
        However,  we note  (\ref{ext}) for $\sigma=\tau>0,$ it is thus natural to define 
             \begin{equation*}\label{ext1}
\Delta(r)= \big\{x\in M: \  \rho(x)<r\big\}
    \end{equation*}
 for $\sigma=\tau=0.$ Let  $\tau\to0,$  then it  yields  from      (\ref{gtt}) that 
    \begin{eqnarray}\label{bbb}
       g_r(o, x)&=& \frac{2}{\omega_{2m-1}}\int_{\rho(x)}^r \chi(0,t)^{1-2m}dt  \\
       &=& \frac{2}{\omega_{2m-1}}\int_{\rho(x)}^r \frac{dt}{t^{2m-1}}. \nonumber
         \end{eqnarray}
    To conclude, $g_r(o, x)$ has the unified form 
  \begin{equation*}\label{ddd1}
       g_r(o, x)=\frac{2}{\omega_{2m-1}}\int_{\rho(x)}^r \chi(\tau,t)^{1-2m}dt
         \end{equation*}
for  $\sigma=\tau\geq0.$ 
In further,  we have 
 \begin{eqnarray}\label{para}
 d\pi_r&=&\frac{1}{\omega_{2m-1}}\frac{\partial}{\partial{\vec{\nu}}}\int_{\rho(x)}^r \chi(\tau,t)^{1-2m}d\sigma_r \\
 &=&  \frac{1}{\omega_{2m-1}}\chi(\tau,r)^{1-2m}d\sigma_r   \nonumber
   \end{eqnarray}
for  $\sigma=\tau\geq0.$

Let $f: M\to \mathbb P^1(\mathbb C)$ be a meromorphic function. Similarly,  one can defined the  Nevanlinna's functions $T_f(r, \omega_{FS}), m_f(r, a),     N_f(r, a)$ and 
    $\overline{N}_f(r, a)$ of $f$ as well as $T(r, \mathscr R)$  on $\Delta(r)$  as in Section \ref{sec421} and Section \ref{sec424}.

\subsubsection{Calculus Lemma and Logarithmic Derivative Lemma}~

  \begin{lemma}\label{grest}   For  $0<t\leq r,$ we have    for   $x\in\partial\Delta(t)$
 $$g_r(o,x)=\frac{2}{\omega_{2m-1}}\int_t^r\chi(\tau, s)^{1-2m}ds.$$
 \end{lemma}
 \begin{proof}
When $\sigma=\tau=0,$ $\Delta(r)$ is the geodesic ball centered at  $o$ with radius $r.$   The conclusion holds immediately  due to (\ref{bbb}).    In the following, we  consider the other case. Write 
    \begin{eqnarray*}
g_r(o,x)&=& G(o,x)-\frac{2}{\omega_{2m-1}}\int_r^\infty \left(\frac{\sinh\tau t}{\tau}\right)^{1-2m}dt \\
&=& G(o,x)-\frac{2}{\omega_{2m-1}}\int_t^\infty \left(\frac{\sinh\tau s}{\tau}\right)^{1-2m}ds \\
&& +\frac{2}{\omega_{2m-1}}\int_t^r \left(\frac{\sinh\tau s}{\tau}\right)^{1-2m}ds.
  \end{eqnarray*}
    Since $ g_t(o, x)\equiv0$ on $\partial\Delta(t),$  we deduce that 
    $$G(o,x)-\frac{2}{\omega_{2m-1}}\int_t^\infty \left(\frac{\sinh\tau s}{\tau}\right)^{1-2m}ds=0, \ \ \  \ ^\forall x\in \partial\Delta(r).$$
   Hence,  for  $x\in\partial\Delta(t)$
      \begin{eqnarray*}
 g_r(o, x)&=& \frac{2}{\omega_{2m-1}}\int_t^r \left(\frac{\sinh\tau s}{\tau}\right)^{1-2m}ds \\
&=& \frac{2}{\omega_{2m-1}}\int_t^r\chi(\tau, s)^{1-2m}ds.  
  \end{eqnarray*}
  This completes the proof. 
 \end{proof}

  \begin{theorem}\label{hh} 
 There exists a constant $c_1>0$ such that 
   $$\|\nabla g_r(o, x)\|\leq \frac{c_1}{\omega_{2m-1}} \left((2m-1)\tau^2+r^{-1}\right)\int_{r}^\infty\chi(\sigma,t)^{1-2m}dt$$
holds for all  $x\in\partial\Delta(r).$
\end{theorem}
\begin{proof}  When $\sigma=\tau=0,$ $\Delta(r)$ is a geodesic ball centered at $o$ with radius $r.$
Hence,  it yields from (\ref{chi}) and  (\ref{bbb})   that 
   \begin{eqnarray*}
 \|\nabla g_r(o,x)\|\big|_{\partial\Delta(r)}
 &=&\frac{2}{\omega_{2m-1}}r^{1-2m} \\
 &=& \frac{4m-4}{\omega_{2m-1}}r^{-1}\int_r^\infty t^{1-2m}dt \\
 &=& \frac{4m-4}{\omega_{2m-1}}r^{-1}\int_r^\infty \chi(0,t)^{1-2m}dt. 
    \end{eqnarray*}
    This implies  the conclusion holds. 
For $\sigma,\tau\not=0,$  by  Lemma \ref{CY} with  (\ref{curvature}), (\ref{chi}) and  (\ref{tiancai})      (see Remark 5 in \cite{Sa0} also)
  \begin{eqnarray*}
\|\nabla G(o, x)\|&\leq& c_0\left((2m-1)\tau^2+\rho(x)^{-1}\right)G(o,x) \\
&\leq& \frac{2c_0}{\omega_{2m-1}} \left((2m-1)\tau^2+\rho(x)^{-1}\right)\int_{\rho(x)}^\infty\chi(\sigma, t)^{1-2m}dt
  \end{eqnarray*}
 for some large  constant $c_0>0$ which depends  only on  the dimension  $m.$  By  (\ref{tiancai})   again 
 $$\int_{\rho(x)}^\infty \left(\frac{\sinh\tau t}{\tau}\right)^{1-2m}dt \leq  \int_r^\infty \left(\frac{\sinh\tau t}{\tau}\right)^{1-2m}dt, \ \ \  \ ^\forall x\in \partial\Delta(r),$$
 which  implies that     
 $\rho(x)\geq r$ for  all  $x\in\partial\Delta(r).$
Thus, we conclude that 
  \begin{eqnarray*}\label{pqp}
\|\nabla g_r(o, x)\|\big|_{\partial\Delta(r)}&=&\|\nabla G(o, x)\|\big|_{\partial\Delta(r)}  \\
&\leq& \frac{c_1}{\omega_{2m-1}} \left((2m-1)\tau^2+r^{-1}\right)\int_{r}^\infty\chi(\sigma,t)^{1-2m}dt, \nonumber
  \end{eqnarray*}
  where $c_1=2c_0.$
    \end{proof}
    
  \begin{cor}\label{bbba}  There exists a constant $c_2>0$ such that 
   $$d\pi_r\leq \frac{c_2}{\omega_{2m-1}} \left((2m-1)\tau^2+r^{-1}\right)\int_{r}^\infty\chi(\sigma,t)^{1-2m}dt\cdot d\sigma_r$$
holds for all  $x\in\partial\Delta(r),$
where $d\sigma_r$ is the Riemannian area element  of $\partial\Delta(r).$
\end{cor}
 
 Set
  \begin{equation}\label{0p0}
E(r,\delta)=\frac{(2m-1)\tau^2+r^{-1}}{\chi(\tau,r)^{(1-2m)(1+\delta)}}\int_{r}^\infty\chi(\sigma,t)^{1-2m}dt, 
 \end{equation}

\begin{theorem}[Calculus Lemma]\label{logo1} Let $k\geq0$ be a locally integrable function on $M.$ Assume that $k$ is locally bounded at $o.$ Then  
 for any   $\delta>0,$  there exists   a set $E_{\delta}\subseteq(0,\infty)$ of finite Lebesgue measure such that
$$\int_{\partial\Delta(r)}kd\pi_r\leq CE(r,\delta)\left(\int_{\Delta(r)}g_r(o,x)kdv\right)^{(1+\delta)^2}$$
holds for  all $r>0$ outside $E_{\delta},$ where $E(r,\delta)$ is defined by $(\ref{0p0}).$ 
 \end{theorem}
\begin{proof} 
  Lemma \ref{grest} yields  that
      \begin{eqnarray*}
\Gamma(r)&:=&   \int_{\Delta(r)}g_r(o,x)kdv \\
&=& \int_0^rdt\int_{\partial\Delta(t)}g_r(o,x)kd\sigma_t \\
      &=& \frac{2}{\omega_{2m-1}}\int_{0}^rdt\int_t^r\chi(\tau, s)^{1-2m}ds\int_{\partial\Delta(t)}kd\sigma_t. 
      \end{eqnarray*}
     Differentiating $\Gamma$  to get    
      $$\Gamma'(r)=\frac{2}{\omega_{2m-1}}\chi(\tau, r)^{1-2m}\int_{0}^rdt\int_{\partial\Delta(t)}kd\sigma_t.$$
      In further, we have 
      $$\frac{d}{dr}\frac{\Gamma'(r)}{\chi(\tau, r)^{1-2m}}=\frac{2}{\omega_{2m-1}}\int_{\partial\Delta(r)}kd\sigma_r.$$
   Applying  Borel's growth  lemma  to the left hand side on this equality twice,  we  can deduce  that there exists a set $E_\delta\subseteq(0,\infty)$ of finite Lebesgue measure such that   
            \begin{eqnarray}\label{w1}
      \int_{\partial\Delta(r)}kd\sigma_r   
    &\leq& 
    \frac{\omega_{2m-1}}{2}\chi(\tau, r)^{(2m-1)(1+\delta)}\Lambda(r)^{(1+\delta)^2}   \\
      &=&\frac{\omega_{2m-1}}{2}\chi(\tau,r)^{(2m-1)(1+\delta)}\left(\int_{\Delta(r)}g_r(o,x)kdv\right)^{(1+\delta)^2} \nonumber
           \end{eqnarray}
  holds for  all $r>0$ outside $E_\delta.$
         Meanwhile,    Corollary  \ref{bbba}  means   that  there exists a constant    $c>0$ such that  
 \begin{equation}\label{w2}
d\pi_r\leq\frac{c}{\omega_{2m-1}} \left((2m-1)\tau^2+r^{-1}\right)\int_{r}^\infty\chi(\sigma,t)^{1-2m}dt\cdot d\sigma_r.
 \end{equation}
Substituting (\ref{w2}) into (\ref{w1}),  we conclude that   
  \begin{eqnarray*}\label{formula}
       \int_{\partial\Delta(r)}kd\pi_r  &\leq&   
   CE(r, \delta)\left(\int_{\Delta(r)}g_r(o,x)kdv\right)^{(1+\delta)^2}
            \end{eqnarray*}
             for all $r>0$ outside $E_\delta$  with $C=c/2.$
\end{proof}

\begin{lemma}\label{999a}  Let
$\psi\not\equiv0$ be a  meromorphic function on  $M.$  Then for any   $\delta>0,$ there exists a subset 
 $E_\delta\subseteq(0,\infty)$ of finite Lebesgue measure such that
\begin{eqnarray*}
    \int_{\partial\Delta(r)}\log^+\frac{\|\nabla\psi\|^2}{|\psi|^2(1+\log^2|\psi|)}d\pi_r 
   &\leq& (1+\delta)^2 \log^+T(r,\psi)+\log E(r,\delta)+O(1)
   \end{eqnarray*}
 holds for  all $r>0$ outside  $E_\delta,$ where $E(r,\delta)$ is defined by $(\ref{0p0}).$  
\end{lemma}
\begin{proof} The concavity of $``\log"$ implies that    
\begin{eqnarray*}
  \int_{\partial\Delta(r)}\log^+\frac{\|\nabla\psi\|^2}{|\psi|^2(1+\log^2|\psi|)}d\pi_r  
    &\leq&  \log^+\int_{\partial\Delta(r)}\frac{\|\nabla\psi\|^2}{|\psi|^2(1+\log^2|\psi|)}d\pi_r+O(1).
\end{eqnarray*}
 Apply  Theorem \ref{logo1} and Lemma \ref{oo12} to the first term on the right hand side of the above inequality, 
then  for any $\delta>0,$ there exists a subset 
 $E_\delta\subseteq(0,\infty)$ of finite Lebesgue measure such that  this term is bounded from above by 
    $$(1+\delta)^2 \log^+T(r,\psi)+\log E(r,\delta)+O(1)$$
  for all  $r>0$ outside  $E_\delta.$ This completes the proof. 
\end{proof}

\begin{theorem}[Logarithmic Derivative Lemma]\label{logo2} Let
$\psi\not\equiv0$ be a  meromorphic function on  $M.$   Then for any   $\delta>0,$ there exists a  subset  $E_\delta\subseteq(0,\infty)$ of  finite Lebesgue measure such that 
\begin{eqnarray*}
   m\Big(r,\frac{\|\nabla\psi\|}{|\psi|}\Big)&\leq& \frac{2+(1+\delta)^2}{2}\log^+ T(r,\psi) 
    +\frac{1}{2}\log E(r, \delta)+O(1)
\end{eqnarray*}
 holds for  all $r>0$ outside  $E_\delta,$ where $E(r,\delta)$ is defined by $(\ref{0p0}).$ 
\end{theorem}
\begin{proof}  We have 
\begin{eqnarray*}
    m\left(r,\frac{\|\nabla\psi\|}{|\psi|}\right)  
   &\leq& \frac{1}{2}\int_{\partial\Delta(r)}\log^+\frac{\|\nabla\psi\|^2}{|\psi|^2(1+\log^2|\psi|)}d\pi_r \ \ \  \  \    \   \  \  \    \    \   \   \\ 
 &&   +\frac{1}{2}\int_{\partial\Delta(r)}\log\left(1+\log^2|\psi|\right)d\pi_r \\
        &\leq&  \frac{1}{2}\int_{\partial\Delta(r)}\log^+\frac{\|\nabla\psi\|^2}{|\psi|^2(1+\log^2|\psi|)}d\pi_r  \\ 
   && +\log\int_{\partial\Delta(r)}\Big{(}\log^+|\psi|+\log^+\frac{1}{|\psi|}\Big{)}d\pi_r +O(1)  \\
   &\leq& \frac{1}{2}\int_{\partial\Delta(r)}\log^+\frac{\|\nabla\psi\|^2}{|\psi|^2(1+\log^2|\psi|)}d\pi_r+\log^+T(r,\psi)+O(1).
\end{eqnarray*}
Apply  Lemma \ref{999a} again,  we have the theorem proved.
   \end{proof}
\vskip\baselineskip

\subsubsection{Second Main  Theorem}~

\begin{theorem}[Second Main Theorem]\label{sec2}
 Let $f: M\rightarrow \mathbb P^1(\mathbb C)$ be a nonconstant   meromorphic function. Let $a_1, \cdots, a_q$ be distinct values in 
  $\mathbb P^1(\mathbb C).$ Then  for any $\delta>0,$ there exists a subset $E_\delta\subseteq(0, \infty)$ of finite Lebesgue measure such that 
          \begin{eqnarray*}
&& (q-2)T_f(r,\omega_{FS})+T(r, \mathscr R) \\
&\leq& \sum_{j=1}^q\overline N_f(r,a_j)+O\left(\log^+T_f(r,\omega_{FS})+(\tau-\sigma)r+\delta\log\chi(\tau,r)\right)
        \end{eqnarray*}
holds for all $r>0$ outside $E_\delta.$ where $\chi(\tau, r)$ is defined by $(\ref{chi}).$ 
\end{theorem}

\begin{proof}
Applying  the similar  arguments as in the proof of Theorem \ref{sec1} and combining Theorem \ref{logo1} with Theorem \ref{logo2},  we  conclude that  
  for any $\delta>0,$ there exists a subset $E_\delta\subseteq(0, \infty)$ of finite Lebesgue measure such that 
          \begin{eqnarray*}
&& (q-2)T_f(r,\omega_{FS})+T(r, \mathscr R) \\
&\leq& \sum_{j=1}^q\overline N_f(r,a_j)+O\left(\log^+T_f(r,\omega_{FS})+\log E(r,\delta)\right)
        \end{eqnarray*}
holds for all $r>0$ outside $E_\delta.$  

 In the following, we estimate 
$\log E(r,\delta).$ When $\sigma>0,$ we have
$$E(r,\delta)=\frac{r^{-1}}{r^{(1-2m)(1+\delta)}}\int_r^\infty t^{1-2m}=\frac{1}{2m-2}r^{(2m-1)\delta}.$$
This implies that the theorem holds. We consider   the other case that $\sigma>0.$  
Since $e^{-\tau r}\leq e^{-\sigma r}\to0$ as $r\to\infty,$ we deduce that 
$$\chi(\tau, r)=O\left(e^{\tau r}\right)$$
and   $$\int_r^\infty\chi(\sigma, r)^{1-2m}=O\left(e^{(1-2m)\sigma r}\right).$$
Thus, we conclude that 
$$\log E(r,\delta)\leq (2m-1)(\tau-\sigma)r+(2m-1)\delta\log\chi(\tau, r)+O(1).$$
This completes the proof.
\end{proof}

\section{Nevanlinna Theory of Algebroid Functions}

 Let $M$ be a  complete non-compact  K\"ahler manifold of $\dim_{\mathbb C}M=m,$   with Laplace-Beltrami operator $\Delta$ associated to  the K\"ahler metric $g.$  
  We  utilize   the same notations such as  $\alpha, \mathscr R,$  etc. defined  in Section \ref{sec40}. 
    
\subsection{First Main Theorem}~\label{sec51}

   Let   $\Omega$ denote a   precompact domain with smooth boundary $\partial\Omega$ in $M.$ Fix a  point $o\in \Omega.$ As  defined   in Section \ref{sec41}, we have the Green function $g_{\Omega}(o,x)$ for $\Omega$ and the harmonic measure $\pi_{\partial \Omega}$ on $\partial\Omega.$

In what follows, we  introduce   Nevanlinna's functions of  algbroid functions.   Let $w=\{w_j\}_{j=1}^\nu$  be a $\nu$-valued algebroid function   with branch set $\mathscr B$ on $M,$  which is defined by an irreducible algebraic equation 
$$\psi(x, w):=A_\nu(x) w^\nu+A_{\nu-1}(x)w^{\nu-1}+\cdots+A_0(x)=0, \ \ \ \     A_\nu\not\equiv0$$
such that    $\mathscr A:=[A_0:\cdots:A_\nu]: M\to \mathbb P^{\nu}(\mathbb C)$  defines  a meromorphic mapping
with  indeterminacy locus  $I_{\mathscr A}.$ 
 In the following, we  assume  that $o\not\in I_{\mathscr A}.$

The \emph{characteristic function} of $w$  on $\Omega$ is defined  by 
$$T(\Omega, w)=m(\Omega, w)+N(\Omega, w),$$
where 
  \begin{eqnarray*}
  m(\Omega, w)&=&\frac{1}{\nu}\sum_{j=1}^\nu\int_{\partial \Omega}\log^+|w_j|d\pi_{\partial\Omega}, \\
    N(\Omega, w)&=&\frac{\pi^m}{(m-1)!\nu}\int_{w^*\infty\cap \Omega}g_{\Omega}(o,x)\alpha^{m-1} \\
    \end{eqnarray*} 
are  called the  \emph{proximity function, counting function} of $w$ on $\Omega,$ respectively.
For $a\in\mathbb C,$ define 
$$T\Big(\Omega, \frac{1}{w-a}\Big)=m\Big(\Omega, \frac{1}{w-a}\Big)+N\Big(\Omega, \frac{1}{w-a}\Big),$$ 
where 
 \begin{eqnarray*}
m\Big(\Omega, \frac{1}{w-a}\Big)&=&\frac{1}{\nu}\sum_{j=1}^\nu\int_{\partial \Omega}\log^+\frac{1}{|w_j-a|}d\pi_{\partial\Omega}, \\
 N\Big(\Omega, \frac{1}{w-a}\Big)&=&\frac{\pi^m}{(m-1)!\nu}\int_{w^*a\cap \Omega}g_{\Omega}(o,x)\alpha^{m-1}.
    \end{eqnarray*} 
In order to unify the form, we write
$$m(\Omega, w)=m\Big(\Omega, \frac{1}{w-\infty}\Big),  \ \ \ \ 
    N(\Omega, w)=N\Big(\Omega, \frac{1}{w-\infty}\Big).$$

\begin{theorem}[First Main Theorem]\label{fr3}  Assume that $w(o)\not=\infty.$ Then  
 $$\left|T\Big(\Omega, \frac{1}{w-a}\Big)-T(\Omega, w)+\frac{1}{\nu}\log\left|\frac{\psi(o,a)}{A_\nu(o)}\right|\right|
\leq \log^+|a|+\log2.$$
 \end{theorem}
 \begin{proof} According to   Vieta's theorem, we have    
$$w_1\cdots w_\nu=(-1)^{\nu}\frac{A_{0}}{A_\nu}.$$ 
Since $\mathscr B$ is polar,   it yields from  Dynkin formula that 
   \begin{eqnarray*}
 && \int_{\partial\Omega}\log\prod_{k=1}^\nu|w_k|d\pi_{\partial\Omega}-\log\prod_{k=1}^\nu|w_k(o)| \\
&=&\frac{1}{2}\int_{\Omega}g_{\Omega}(o,x)\Delta\log\prod_{k=1}^\nu|w_k|dv.
   \end{eqnarray*}
By the definition,   the left hand side equals 
 $$\nu m(\Omega, w)-\nu m\Big(\Omega, \frac{1}{w}\Big)-\log\left|\frac{A_{0}(o)}{A_\nu(o)}\right|.$$
\ \ \ On the other hand,  Theorem \ref{pole} and Theorem \ref{zero} imply that  the right hand side equals 
   \begin{eqnarray*}
&& \frac{1}{2}\int_{\Omega}g_{\Omega}(o,x)\Delta\log\left|\frac{A_0}{A_\nu}\right|dv \\
&=& \frac{\pi^m}{(m-1)!}\int_{A_0^*0\cap \Omega}g_{\Omega}(o,x)\alpha^{m-1} 
-\frac{\pi^m}{(m-1)!}\int_{A_\nu^*0\cap \Omega}g_{\Omega}(o,x)\alpha^{m-1} \\
&=&  N\Big(\Omega, \frac{1}{A_0}\Big)- N\Big(\Omega, \frac{1}{A_\nu}\Big) \\
&=&\nu N\Big(\Omega, \frac{1}{w}\Big)-\nu N(\Omega, w). 
   \end{eqnarray*}
Then, we arrive at  
 $$T(\Omega, w)=T\Big(\Omega, \frac{1}{w}\Big)+\frac{1}{\nu}\log\left|\frac{A_{0}(o)}{A_\nu(o)}\right|.$$
 Setting $u=w-a$ to get 
   \begin{eqnarray*}
  \psi(x,w)&=&A_\nu(x)(u+a)^{\nu}+A_{\nu-1}(x)(u+a)^{\nu-1}+\cdots+A_0(x) \\
  &=& B_\nu(x)u^\nu+B_{\nu-1}(x)u^{\nu-1}+\cdots+B_{0}(x),
   \end{eqnarray*}
   where   
   $$B_k=\sum_{j=0}^{\nu-k}\binom{\nu-j}{k}a^{\nu-k-j}A_{\nu-j}, \ \ \ \  k=0,\cdots,\nu.$$
 In particular,  we have 
 $$B_\nu=A_\nu, \ \ \ \   B_0(o)=\psi(o,a).$$
 Similarly, we obtain  
  $$T(\Omega,u)=T\Big(\Omega, \frac{1}{u}\Big)+\frac{1}{\nu}\log\left|\frac{B_{0}(o)}{B_\nu(o)}\right|,$$
which yields  that 
 $$T(\Omega,w-a)=T\Big(\Omega, \frac{1}{w-a}\Big)+\frac{1}{\nu}\left|\frac{\psi(o,a)}{A_\nu(o)}\right|.$$
Since 
   \begin{eqnarray*}
T(\Omega, w-a)&\leq& T(\Omega,w)+\log^+|a|+\log2, \\ 
T(\Omega,w)&\leq& T(\Omega,w-a)+\log^+|a|+\log 2,
   \end{eqnarray*}
 we  deduce that      
$$\left|T(\Omega,w-a)-T(\Omega,w)\right|\leq \log^+|a|+\log2.$$
Put together the above,  we conclude that 
$$\left|T\Big(\Omega, \frac{1}{w-a}\Big)-T(\Omega,w)+\frac{1}{\nu}\log\left|\frac{\psi(o, a)}{A_\nu(o)}\right|\right|
\leq \log^+|a|+\log2.$$
 \end{proof}

\subsection{Estimate of  Branch Divisors}~

Let $w=\{w_j\}_{j=1}^\nu$ be a $\nu$-valued algebroid function on $M.$ We use the same notations as  
defined  in Section \ref{sec51}. Our purpose  is to  obtain  an estimate of  the branch divisor $\mathscr D$ of $w.$
For each $x\in\mathscr B,$ we use 
   $u_1,\cdots, u_{l(x)}$ to stand for   the all  $\lambda_1(x)$-valued, $\cdots,$ $\lambda_{l(x)}(x)$-valued  components of $w$  at  $x,$ 
respectively. 
Then, we have   $\lambda_1(x)+\cdots+\lambda_{l(x)}(x)=\nu,$ 
and  $u_j$ has branch order $\lambda_{j}(x)-1$  at $x$  with   $j=1,\cdots, l(x),$ and $w$ has  branch order   $\nu-l(x)$ at $x.$ 
The \emph{branch divisor} of $w$ associated to  $\mathscr B$ is  defined by a divisor $\mathscr D$  on $M$   
such that $${\rm{Supp}}\mathscr D=\mathscr B, \ \ \ \   {\rm{Ord}}_{x}\mathscr D=\nu-l(x), \ \ \  \  ^\forall x\in\mathscr B.$$
The counting function of  branch divisor $\mathscr D$ of $w$ for $\Omega$ is defined by
$$N_{{\rm{bran}}}(\Omega, w)=\frac{1}{\nu}N(\Omega,\mathscr D),$$ where 
$$N(\Omega,\mathscr D)=\frac{\pi^m}{(m-1)!}\int_{\mathscr D\cap\Omega}g_{\Omega}(o,x)\alpha^{m-1}.$$
  Recall  that the discriminant of $\psi(x,w)$ is given  by 
$$J_\psi=A_\nu^{2\nu-2}\prod_{1\leq i<j\leq\nu}\big(w_i-w_j\big)^2.$$

 \begin{lemma}\label{esti} We have 
$$\mathscr D\leq \left(J_\psi=0\right).$$
 \end{lemma}
\begin{proof}   For  any $x_0\in\mathscr B,$ we use    $(x_0, u_1),\cdots, (x_0,  u_{l})$ to denote   the all  distinct  function  elements  of $w$ at $x_0,$ with 
$\lambda_1, \cdots,  \lambda_{l}$ leaves 
 respectively. 
It is noted that   $\lambda_1+\cdots+\lambda_{l}=\nu.$ 
From   the assumption, we deduce that  ${\rm{Ord}}_{x_0}\mathscr D=\nu-l.$
Since the set of all singular branch points of $w$ has complex codimension  not smaller  than $2,$ 
  it is sufficient  to deal with     the case  where   $x_0$ is a  non-singular branch point of $w.$  
By  Theorem \ref{ses}, 
we  may write    for $j=1,\cdots, l$
\begin{equation}\label{terms}
   u_{j}(z)=B_{j0}(\hat z_m)+B_{j\tau_j}(\hat z_m)z_m^{\frac{\tau_j}{\lambda_j}}+B_{j(\tau_j+1)}(\hat z_m)z_m^{\frac{\tau_j+1}{\lambda_j}}+\cdots
   \end{equation}
 in a local holomorphic coordinate $z=(z_1,\cdots,z_m)$ near $x_0,$  such that $x_0$ has a local holomorphic coordinate $z=\textbf{0}.$  
If $B_{j0}(\textbf{0})\not=\infty,$
then  $x_0$ is not a pole of $u_j.$ 
Through  a simple observation of $J_\psi,$   
we  deduce     from (\ref{terms}) immediately  that    there exist   $\lambda_j(\lambda_j-1)/2$ terms that  contain  the factor $z_m^{2/\lambda_j},$  i.e.,  $\lambda_j-1$ terms that contain the factor $z_m$ in 
$J_\psi.$ For  the worst case where $B_{j0}(\textbf{0})\not=\infty$ for all $j,$
we  deduce  that   the factor $z_m$ appears a total of at least $\nu-l$ times in $J_\psi.$
  Hence,   we get 
   $${\rm{Ord}}_{x_0}\mathscr D\leq {\rm{Ord}}_{x_0}\bigg(\prod_{1\leq i<j\leq\nu}\big(w_i-w_j\big)^2=0\bigg)\leq {\rm{Ord}}_{x_0} (J_\psi=0).$$
   If $B_{j0}(\textbf0)=\infty,$ then $x_0$ is a pole of $u_j.$
   For  the worst case where $B_{j0}(\textbf{0})=\infty$ for all $j,$
    it yields from  Theorem \ref{pole} that 
$${\rm{Ord}}_{x_0}\mathscr D\leq {\rm{Ord}}_{x_0}(A_\nu=0)\leq {\rm{Ord}}_{x_0} (J_\psi=0).$$
Combining the above two cases, we prove   the lemma. 
\end{proof}
Set 
$$N\big(\Omega, (J_\psi=0)\big) = \frac{\pi^m}{(m-1)!}\int_{(J_\psi=0)\cap\Omega}g_\Omega(o,x)\alpha^{m-1}.$$
 \begin{theorem}[Estimate of branch divisors]\label{bran} We have 
$$N_{\rm{bran}}(\Omega,w)\leq (2\nu-2)T(\Omega,w)+O(1).$$
 \end{theorem}
\begin{proof}   Apply  Dynkin formula, we are led to  
   \begin{eqnarray*} 
 &&  N\big(\Omega, (J_\psi=0)\big) \\ 
 &=&  \frac{1}{4}\int_{\Omega}g_\Omega(o,x)\Delta\log|J_\psi|^2dv \\
   &=&  \frac{\nu-1}{2}\int_{\Omega}g_\Omega(o,x)\Delta\log|A_\nu|^2dv+ \frac{1}{4}\int_{\Omega}g_\Omega(o,x)\Delta\log\prod_{i<j}\big|w_i-w_j\big|^4dv \\
   &=&(2\nu-2)N\Big(\Omega, \frac{1}{A_\nu}\Big)+ \sum_{i<j}\int_{\partial\Omega}\log|w_i-w_j|^2d\pi_{\partial\Omega}+O(1).
     \end{eqnarray*}
  In further, it yields from    Theorem \ref{pole} that 
        \begin{eqnarray*}
         &&  N\big(\Omega, (J_\psi=0)\big) \\  
      &\leq& \nu(2\nu-2)N(\Omega, w) +(2\nu-2)\sum_{j=1}^\nu\int_{\partial\Omega}\log^+|w_j|d\pi_{\partial\Omega}+O(1) \\
   &=& \nu(2\nu-2)N(\Omega, w)+(2\nu-2)\sum_{j=1}^\nu m(\Omega, w_j)+O(1) \\
  &=& \nu(2\nu-2)N(\Omega, w)+\nu(2\nu-2)m(\Omega, w)+O(1) \\
   &=& \nu(2\nu-2)T(\Omega,w)+O(1).
        \end{eqnarray*}
By this, we conclude  from Lemma \ref{esti} that 
$$N_{\rm{bran}}(\Omega,w)\leq \frac{1}{\nu}N(\Omega,(J_\psi=0))\leq (2\nu-2)T(\Omega,w)+O(1).$$
\end{proof}

 \subsection{Second Main Theorem}~\label{sec53}

 Let $\mathcal M$ be the $\nu$-leaf complex manifold by  a $\nu$-valued algebroid function $w$ on $M,$ which is of  complex dimension $m$
 by Theorem \ref{cover}.  
Let $\pi: \mathcal M\to M$ be the  natural  projection.  We equip  $\mathcal M$ with the  pullback metric $\pi^*g$ induced from $g$ via $\pi.$
 Note that  $\pi^*g$ is a positive  semi-definite K\"ahler metric  on $\mathcal M,$ i.e.,  
 $\pi^*g$  is a  K\"ahler  metric    on  $\mathcal M\setminus\pi^{-1}(\mathscr B)$ and  degenerate on $\pi^{-1}(\mathscr B).$
      Since  $\mathscr B$ is   a polar  set with ${\rm codim}_{\mathbb C}\mathscr B=1,$    Dynkin formula still  works on $(\mathcal M, \pi^*g)$  for the following reasons. 
 Set $K=\pi^{-1}(\mathscr B).$   Picking    a sequence $\{U_n(K)\}_{n=1}^\infty$ of  neighborhoods $U_n(K)$ of $K$ such that 
$$\cdots\subseteq \overline{U_{k+1}(K)}\subseteq U_k(K)\subseteq\cdots\subseteq \overline{U_2(K)}\subseteq U_1(K), \ \ \ \   \bigcap_{n\geq1}U_n(K)=K.$$
\ \ \ \  Let $\{\tau_{n}, \theta_{n}\}$ be a partition  of the unity   subordinate to the open covering $\{\mathcal M\setminus U_{n+1}(K), U_n(K)\}$ of $\mathcal M.$
Then,  $0\leq\tau_n\leq1, 0\leq\theta_n\leq1$ and $\tau_n+\theta_n=1.$ 
Fix  a  K\"ahler metric $g_0$ on $\mathcal M$ with K\"ahler form $\omega_0.$  
  Set $g_n=\pi^*g+n^{-2}\theta_{n}g_0.$  Since $g_0$ is  positive  definite on $\mathcal M,$ we see that $g_n$ is also  positive  definite on $\mathcal M.$  Moreover, since 
   $d(\pi^*\omega+n^{-2}\theta_{n}\omega_0)=\pi^*d\omega+n^{-2}\theta_{n}d\omega_0=0,$
   we conclude    that $g_n$ is a  K\"ahler metric on $\mathcal M$ with   $g_n=\pi^*g$ on $\mathcal M\setminus U_n(K).$
   Let   $\tilde m, \tilde m_{n}$ denote    the Riemannian  measures  
 of $\pi^*g, g_n,$  respectively. 
  By   ${\rm codim}_{\mathbb C}\mathscr B=1$ and  $g_n\to \pi^*g$    as $n\to+\infty,$  
          we  can deduce that       $\tilde m_{n}\to \tilde m$ weakly as $n\to+\infty.$
                      Let $\tilde\Delta, \tilde{\Delta}_n$ 
                         stand for      the Laplace-Beltrami operators
         of $\pi^*g, g_n$ on $\mathcal M\setminus K, \mathcal M,$ 
                  respectively. According to  $g_n\to \pi^*g$  and  $\tilde m_{n}\to \tilde m$ weakly as $n\to+\infty,$ 
         one   has   $\tilde\Delta_{n}\to\tilde\Delta$ weakly as $n\to+\infty.$        
Taking a   bounded domain $\tilde\Omega$ with smooth boundary $\partial\tilde\Omega$ in $\mathcal M.$
Fixing  $\tilde o\in\tilde\Omega$ with  $\tilde o\not\in K.$
Denote by    $\tilde g_{\tilde\Omega}(\tilde o, x), \tilde g_{\tilde\Omega, n}(\tilde o, x)$    
   the  Green functions  of  $\tilde\Delta/2, \tilde{\Delta}_n/2$ for
             $\tilde\Omega$ with a pole at $\tilde o$ satisfying Dirichlet boundary condition respectively,  here   
                                              $\tilde g_{\tilde\Omega}(\tilde o, x)$ is well-defined  on $\overline{\tilde\Omega}\setminus K.$
                          Also,  we denote by      $\tilde\pi_{\partial\tilde\Omega}, \tilde\pi_{\partial \tilde\Omega, n}$       
                                      the harmonic measures of $\pi^*g, g_n$ on 
                                     $\partial\tilde\Omega$ with respect to $\tilde o$  respectively,  
                                here     $\tilde\pi_{\partial\tilde\Omega}$  is well-defined  on 
                                    $\partial\Omega\setminus K.$ Since $\mathscr B$ is a  polar set, 
                                                                          the previous   arguments can  show that 
 $\tilde\Delta_{n}\tilde g_{\tilde\Omega, n}(\tilde o, x)\to \tilde\Delta\tilde g_{\tilde\Omega}(\tilde o, x)$   in the sense of 
  distributions and  
 $\tilde\pi_{\partial \tilde\Omega, n}\to \tilde\pi_{\partial \tilde\Omega}$ weakly  as $n\to+\infty.$  
By   Dynkin formula 
$$\int_{\partial \tilde\Omega}u(x)d\tilde\pi_{\partial \tilde\Omega, n}(x)-u(\tilde o)=\frac{1}{2}\int_{\tilde\Omega}\tilde g_{\tilde\Omega, n}(\tilde o,x)\tilde{\Delta}_n u(x)d\tilde v_n(x), \ \ \ \  ^\forall n\geq1$$
 for a $\mathscr C^2$-class function $u$ with $u(\tilde o)\not=\infty$ on  $\mathcal M$ outside a polar set of singularities at most.  
 Since $\mathscr B$ is polar,   we conclude by letting $n\to+\infty$ that   
 $$\int_{\partial \tilde\Omega}u(x)d\tilde\pi_{\partial \tilde\Omega}(x)-u(\tilde o)=\frac{1}{2}\int_{\tilde\Omega}\tilde g_{\tilde\Omega}(\tilde o,x)\tilde{\Delta} u(x)d\tilde v(x).$$
 This means   that Dynkin formula works on $\mathcal M$  under the pullback metric $\pi^*g.$
    By   Corollary \ref{ppll},  one can lift $w$  to a meromorphic function $f: \mathcal M\to \mathbb P^1(\mathbb C)$ 
   such that $w=f\circ\pi^{-1}.$  In further, it is not difficult  to examine  that 
 Theorem \ref{first},  Theorem \ref{sec1} and Theorem \ref{sec2}
 can apply  to $f$  on each leaf of $\mathcal M$ under the corresponding conditions.

  \subsubsection{$M$ has non-negative Ricci curvature}~

Assume   that $M$ is non-parabolic.  With the pullback metric $\pi^*g,$ 
one  notes  that  $\mathcal M$ is  non-parabolic,  with non-negative Ricci curvature outside $\pi^{-1}(\mathscr B).$
     Let $\tilde\alpha$ be the K\"ahler form of $\pi^*g.$
   Set $\Omega=\Delta(r),$ where $\Delta(r)$ is the bounded   domain defined in Section \ref{sec421}.    
  We recall   the  definitions of Green function $g_r(o,x)$   for $\Delta(r)$ and  harmonic measure $\pi_r$  on $\partial\Delta(r)$  in Section \ref{sec421}.  
 Then, the characteristic function of $w$  is defined  by 
$$T(r, w)=m(r, w)+N(r, w),$$
where 
  \begin{eqnarray*}
  m(r, w)&=&\frac{1}{\nu}\sum_{j=1}^\nu\int_{\partial \Delta(r)}\log^+|w_j|d\pi_{r}, \\
    N(r, w)&=&\frac{\pi^m}{(m-1)!\nu}\int_{w^*\infty\cap \Delta(r)}g_{r}(o,x)\alpha^{m-1}.
    \end{eqnarray*} 
Indeed, define  for $a\in\mathbb C$
$$T\Big(r, \frac{1}{w-a}\Big)=m\Big(r, \frac{1}{w-a}\Big)+N\Big(r, \frac{1}{w-a}\Big), \ \ \ \ \ $$ 
where 
 \begin{eqnarray*}
m\Big(r, \frac{1}{w-a}\Big)&=&\frac{1}{\nu}\sum_{j=1}^\nu\int_{\partial \Delta(r)}\log^+\frac{1}{|w_j-a|}d\pi_{r}, \\
   N\Big(r, \frac{1}{w-a}\Big)&=&\frac{\pi^m}{(m-1)!\nu}\int_{w^*a\cap \Delta(r)}g_{r}(o,x)\alpha^{m-1}.
    \end{eqnarray*} 
 In order to unify the form, we write
$$m(r, w)=m\Big(r, \frac{1}{w-\infty}\Big), \ \ \ \ 
N(r, w)=N\Big(r, \frac{1}{w-\infty}\Big).$$
The simple counting function of $w$ with respect to $a\in\mathbb P^1(\mathbb C)$ is defined by
     $$\overline{N}\Big(r, \frac{1}{w-a}\Big)=\frac{\pi^m}{(m-1)!\nu}\int_{w^{-1}(a)\cap \Delta(r)}g_{r}(o,x)\alpha^{m-1}.$$

    Assume that $o\not\in\mathscr B.$  
  Let  $\tilde\Delta_1(r),\cdots,\tilde\Delta_\nu(r)$ be the    
  $\nu$ connected components of $\pi^{-1}(\Delta(r))$   centered   at $\tilde o_1,\cdots,\tilde o_\nu$ respectively, 
with that $\pi(\tilde o_k)=o$ for each $k.$  Let $\tilde\Delta$ denote the Laplace-Beltrami operator of $\pi^*g,$ which is well-defined on $\mathcal M\setminus \pi^{-1}(\mathscr B).$  Define 
$$\tilde g_r(\tilde o_k, x)=g_r(o,\pi(x)), \ \ \ \  d\tilde \pi_{k, r}(x)=d\pi_r(\pi(x)).$$
In a sense,  we  can  
regard        outside  $\pi^{-1}(\mathscr B)$     that  $\tilde g_r(\tilde o_k,x)$ is  the  Green function of $\tilde\Delta/2$ for $\tilde\Delta_k(r)$ with a pole at $\tilde o_k$ satisfying  Dirichlet boundary condition,  
and that   $\tilde\pi_{k, r}$ is  the harmonic measure on  $\partial{\tilde\Delta}_k(r)$ with respect to $\tilde o_k.$
Set  
 $$T_k(r,f)=m_k(r,f)+N_k(r,f),$$
 where
   \begin{eqnarray*}
m_k(r,f)&=& \int_{\partial{\tilde\Delta}_k(r)}\log^+|f|d\tilde\pi_{k, r}, \\
N_k(r,f)&=& \frac{\pi^m}{(m-1)!}\int_{f^*\infty\cap\tilde\Delta_k(r)}\tilde g_r(\tilde o_k,x)\tilde\alpha^{m-1}. 
    \end{eqnarray*}
Define for $a\in\mathbb P^1(\mathbb C)$
$$\overline{N}_k\Big(r,\frac{1}{f-a}\Big)=\frac{\pi^m}{(m-1)!}\int_{f^{-1}(a)\cap\tilde \Delta_k(r)}\tilde g_r(\tilde o_k,x)\tilde\alpha^{m-1} \ \ \ \ \   $$
and
$$T_k(r, \pi^*\mathscr R)= \frac{\pi^m}{(m-1)!}\int_{\tilde \Delta_k(r)}\tilde g_{r}(\tilde o_k,x)\pi^*\mathscr R\wedge\tilde\alpha^{m-1}.$$
\ \ \ \    If $w$ is not constant, then it follows from  Theorem \ref{sec1} that for $q$ distinct values $a_1,\cdots, a_q\in\mathbb P^1(\mathbb C),$ 
there is  a subset $E_\delta\subseteq(0, \infty)$ of  finite Lebesgue measure such that 
 \begin{eqnarray}\label{bcg}
  &&  (q-2)T_{k}(r, f)+T_k(r, \pi^*\mathscr R) \\
  &\leq& \sum_{j=1}^q\overline{N}_k\Big(r,\frac{1}{f-a_j}\Big)+O\big(\log^+T_k(r, f)+\log H(r)+\delta\log r\big) \nonumber
    \end{eqnarray}
holds for $k=1,\cdots,\nu$ and all $r>0$ outside   $E_\delta.$

\begin{theorem}[Second Main Theorem]\label{sec3}
 Let $w$ be a nonconstant   $\nu$-valued algebroid function on $M.$  Let $a_1, \cdots, a_q$ be distinct values in 
  $\mathbb P^1(\mathbb C).$ Then  for any $\delta>0,$ there exists a subset $E_\delta\subseteq(0, \infty)$ of finite Lebesgue measure such that 
     \begin{eqnarray*}
&&  (q-2\nu)T(r,w)+T(r, \mathscr R)  \\
  &\leq& \sum_{j=1}^q\overline{N}\Big(r,\frac{1}{w-a_j}\Big)+O\big(\log^+T(r,w)+\log H(r)+\delta\log r\big)
    \end{eqnarray*}
holds for all $r>0$ outside $E_\delta.$ 
\end{theorem}
\begin{proof}  Note   that $\pi: \mathcal M\setminus{\pi^{-1}(\mathscr B)}\rightarrow M\setminus{\mathscr B}$  is a $\nu$-sheeted  analytic covering,  
and   two distinct connected components 
  $\tilde\Delta_i(r), \tilde\Delta_j(r)$  of $\pi^{-1}(\Delta(r))$  meet   only at  the   points  in   
  $\pi^{-1}(\mathscr B).$  Moreover, 
                                              any  point  in  $w^{-1}\{a_1,\cdots, a_q\}$ is possibly  a branch point of $w,$
                                               or  equivalently, any  point  in  $f^{-1}\{a_1,\cdots, a_q\}$ is possibly  a ramification  point of $\pi.$ 
                                     Hence, it follows from    $w=f\circ\pi^{-1}$ that   
$$\frac{1}{\nu}\sum_{k=1}^\nu \sum_{j=1}^q\overline{N}_k\Big(r,\frac{1}{f-a_j}\Big)\leq \sum_{j=1}^q\overline{N}\Big(r,\frac{1}{w-a_j}\Big)+ N_{\rm{bran}}(r,w).$$
On the other hand,  we  have $T(r, \mathscr R)=T_{k}(r, \pi^*\mathscr R)$ as well as   
   \begin{eqnarray*}
 T(r, w) 
&=& m(r,w)+N(r,w) \\
&\leq& \frac{1}{\nu}\sum_{k=1}^\nu m_k(r,f)+ \frac{1}{\nu}\sum_{k=1}^\nu N_k(r,f) \\
&=&  \frac{1}{\nu}\sum_{k=1}^\nu T_k(r,f) \\
&\leq& T(r,w)+N_{\rm{bran}}(r,w).
   \end{eqnarray*}
Combining the above with (\ref{bcg}) and Theorem \ref{bran},  it yields   that for any $\delta>0,$ there exists a subset $E_\delta\subseteq(0, \infty)$ of finite Lebesgue measure such that 
   \begin{eqnarray*}
&& (q-2)T(r, w)+T(r, \mathscr R)  \\
&\leq& \frac{q-2}{\nu}\sum_{k=1}^\nu T_k(r,f)+\frac{1}{\nu}\sum_{k=1}^{\nu} T_{k}(r, \pi^*\mathscr R) \\
&\leq& \frac{1}{\nu}\sum_{k=1}^\nu\sum_{j=1}^q\overline{N}_k\Big(r,\frac{1}{f-a_j}\Big)
+O\big(\log^+T(r,w)+\log H(r)+\delta\log r\big) \\
&\leq& \sum_{j=1}^q\overline{N}\Big(r,\frac{1}{w-a_j}\Big)+ N_{\rm{bran}}(r,w) 
 +O\big(\log^+T(r,w)+\log H(r)+\delta\log r\big)
   \end{eqnarray*}
holds for all $r>0$ outside  $E_\delta.$ 
By Theorem \ref{bran}, we prove the theorem.
\end{proof}

When $M=\mathbb C^m,$ we have $\mathscr R=0.$ By $V(r)=O(r^{2m}),$  it is trivial   to deduce  that  $\log H(r)=O(1).$
Hence, we obtain: 

\begin{cor}\label{haoba}
 Let $w$ be a nonconstant   $\nu$-valued algebroid function on $\mathbb C^m$ with $m\geq2.$  Let $a_1, \cdots, a_q$ be distinct values in 
  $\mathbb P^1(\mathbb C).$ Then  for any $\delta>0,$ there exists a subset $E_\delta\subseteq(0, \infty)$ of finite Lebesgue measure such that 
     \begin{eqnarray*}
  (q-2\nu)T(r,w)
  &\leq& \sum_{j=1}^q\overline{N}\Big(r,\frac{1}{w-a_j}\Big)+O\big(\log^+T(r,w)+\delta\log r\big)
    \end{eqnarray*}
holds for all $r>0$ outside $E_\delta.$ 
\end{cor}

  \subsubsection{$M$ has non-positive sectional curvature}~

Assume that     $M$ has  curvature  satisfying  (\ref{curvature}).   
Note that the curvature of $\mathcal M$ also  satisfies   (\ref{curvature})  outside $\pi^{-1}(\mathscr B)$
 under   the pullback metric $\pi^*g.$ Similarly, 
 we obtain     notations  $T(r, w), m(r, w), N(r, w), \overline{N}(r,w)$ and $T(r, \mathscr R)$
   by setting $\Omega=\Delta(r).$
  Employing     the similar  arguments as in the proof of Theorem \ref{sec3}, we can conclude  from  Theorem \ref{sec2} that 
\begin{theorem}[Second Main Theorem]\label{sec4}
 Let $w$ be a nonconstant   $\nu$-valued algebroid function on $M.$  Let $a_1, \cdots, a_q$ be distinct values in 
  $\mathbb P^1(\mathbb C).$ Then  for any $\delta>0,$ there exists a subset $E_\delta\subseteq(0, \infty)$ of finite Lebesgue measure such that 
          \begin{eqnarray*}
         &&   (q-2\nu)T(r,w)+T(r, \mathscr R) \\
  &\leq& \sum_{j=1}^q\overline{N}\Big(r,\frac{1}{w-a_j}\Big)+O\left(\log^+T(r,w)+(\tau-\sigma)r+\delta\log\chi(\tau, r)\right)
                  \end{eqnarray*}
        holds for all $r>0$ outside $E_\delta.$
\end{theorem}

When  $M=\mathbb C^m,$ we deduce that      $\mathscr R=0,$            
 $\sigma=\tau=0$ and $\chi(\tau, r)=r.$ Thus, we conclude   that

 \begin{cor}\label{dong1}
 Let $w$ be a nonconstant   $\nu$-valued algebroid function on $\mathbb C^m.$  Let $a_1, \cdots, a_q$ be distinct values in 
  $\mathbb P^1(\mathbb C).$ Then  for any $\delta>0,$ there exists a subset $E_\delta\subseteq(0, \infty)$ of finite Lebesgue measure such that 
          \begin{eqnarray*}
           (q-2\nu)T(r,w) 
  &\leq& \sum_{j=1}^q\overline{N}\Big(r,\frac{1}{w-a_j}\Big)+O\left(\log^+T(r,w)+\delta\log r\right)
                  \end{eqnarray*}
        holds for all $r>0$ outside $E_\delta.$
\end{cor}

Let $\mathbb B^m$ be the unit ball centered at the origin $\textbf 0$ in $\mathbb C^m.$  Consider   $M=\mathbb B^m$  equipped with the  Poincar\'e  metric 
\begin{equation}\label{kkk}
ds^2=\frac{4\|dz\|^2}{(1-\|z\|^2)^2}
\end{equation}
 in a global holomorphic  coordinate $z=(z_1,\cdots,z_m),$  where  
         \begin{eqnarray*}
\|z\|^2=|z_1|^2+\cdots+|z_m|^2, \ \ \ \  
  \|dz\|^2=|dz_1|^2+\cdots+|dz_m|^2.
        \end{eqnarray*}
It is known  that  $\mathbb B^m$ is a  complete non-compact  K\"ahler manifold of  constant  sectional curvature $-1$ under this metric.  Using   (\ref{kkk}), the associated    Laplace-Beltrami operator  is   
$$\Delta=\left(1-\|z\|^2\right)^2\sum_{j=1}^m\frac{\partial^2}{\partial z_j\partial{\bar z}_j}.$$
 Take   $\textbf{0}\in\mathbb B^m$ as a reference point. 
   The associated  Poincar\'e distance function $\rho(z)$ of $z$ from  $\textbf{0}$ on $\mathbb B^m$ can be easily written   as   
\begin{equation*}
\rho(z)=2\int_0^{\|z\|}\frac{dt}{1-t^2}=\log\frac{1+\|z\|}{1-\|z\|}=2{\rm{artanh}}\|z\|.
\end{equation*}
Hence, we obtain   the minimal  positive global  Green function $G(\textbf{0}, z)$ of $\Delta/2$ for $\mathbb B^m$ with a pole at  $\textbf{0}$:    
$$G(\textbf{0}, z)=\left\{
                \begin{array}{ll}
                                 \frac{1}{\pi}\log\coth\frac{\rho(z)}{2}, \ \ & m=1;  \\
 \frac{1}{(m-1)\omega_{2m-1}}\Big[\Big(\tanh\frac{\rho(z)}{2}\Big)^{2-2m}-1\Big], \  \ & m\geq2,
                                               \end{array}
              \right.$$
where $\omega_{2m-1}:=2\pi^m/(m-1)!$ stands for   the  standard Euclidean area of the  unit sphere $\mathbb S^{2m-1}$ in $\mathbb C^{m}.$   
By  the  definition (see  Section \ref{sec21}), we have 
$$\Delta(r)=\left\{
                \begin{array}{ll}
                                \Big{\{}z\in\mathbb B^m: \ G(\textbf{0}, z)>\frac{1}{\pi}\log\coth\frac{r}{2}\Big{\}}, \ \   &   m=1; \\
\Big{\{}z\in\mathbb B^m: \ G(\textbf{0}, z)>  \frac{1}{(m-1)\omega_{2m-1}}\Big[\Big(\tanh\frac{r}{2}\Big)^{2-2m}-1\Big], \  \ & m\geq2.
                                               \end{array}
              \right.$$       
This gives   
$$\Delta(r)=\big\{z\in\mathbb B^m: \rho(z)<r\big\},$$
 which  is  just  the geodesic ball centered at $\textbf{0}$ with radius $r$ in $\mathbb B^m.$  Hence, the boundary $\partial\Delta(r)$ of $\Delta(r)$ is the geodesic sphere centered 
 at $\textbf{0}$ with radius $r.$
 Whence, we  obtain  the  Green function $g_r(\textbf{0}, z)$ of $\Delta/2$ for $\Delta(r)$ with a pole at  $\textbf{0}$ satisfying Dirichlet boundary condition:    
$$g_r(\textbf{0}, z)=\left\{
                \begin{array}{ll}
                                 \frac{1}{\pi}\log\left(\tanh\frac{r}{2}\coth\frac{\rho(z)}{2}\right), \ \ & m=1;  \\
 \frac{1}{(m-1)\omega_{2m-1}}\Big[\Big(\tanh\frac{\rho(z)}{2}\Big)^{2-2m}-\Big(\tanh\frac{r}{2}\Big)^{2-2m}\Big], \  \ & m\geq2.
                                               \end{array}
              \right.$$          
Thus, for  $z\in\partial\Delta(t)$ with $0<t\leq r$         
\begin{equation}\label{aaa}
g_r(\textbf{0}, z)=\left\{
                \begin{array}{ll}
                                 \frac{1}{\pi}\log\left(\tanh\frac{r}{2}\coth\frac{t}{2}\right), \ \ & m=1;  \\
 \frac{1}{(m-1)\omega_{2m-1}}\Big[\Big(\tanh\frac{t}{2}\Big)^{2-2m}-\Big(\tanh\frac{r}{2}\Big)^{2-2m}\Big], \  \ & m\geq2.
                                               \end{array}
              \right.
             \end{equation}
                    We write 
            \begin{eqnarray*}
           ds^2=dr^2+d\xi^2_r, \ \ \ \ 
\|dz\|^2=d\|z\|^2+\|z\|^2d\theta^2
     \end{eqnarray*}
in the (geodesic) polar coordinate form, where $d\xi_r^2$ is the  Riemannian metric on  $\partial\Delta(r)$ induced by $ds^2,$ and 
$$d\theta^2=\sum_{i,j=1}^{2m-1}\phi_{ij}d\theta_i d\theta_j$$
is the standard Euclidean metric on $\mathbb S^{2m-1}.$ It yields from  (\ref{kkk}) that  
            \begin{equation}\label{ppp}
                     dr= \frac{2d\|z\|}{1-\|z\|^2}, \ \ \ \ 
d\xi_r =\frac{2\|z\|d\theta}{1-\|z\|^2}
     \end{equation}
with  
$\|z\|=\tanh(r/2).$          
Let $\sigma_r$ be the Riemannian area element of $\partial\Delta(r).$ By  the second formula in (\ref{ppp}),    it gives immediately  
            \begin{equation}\label{bbbb}
       d\sigma_r= \sqrt{\det\left(\frac{4\|z\|^2\phi_{ij}}{(1-\|z\|^2)^2}\right)}
       =\left(\sinh r\right)^{2m-1}dA,
 \end{equation}
        where $dA$ is  the  standard Euclidean  area element of $\mathbb S^{2m-1}.$  
       
       With these  preparations, we shall  give an upper  estimate of  $T(r, \mathscr R).$
Let   $s_c$  be   the   scalar curvature of   $\mathbb B^m,$ which  is  a negative  constant. It yields   from    (\ref{bbbb}) that   
     \begin{eqnarray*}
 |T(r,\mathscr R)|
 &=& \frac{|s_c|}{2}\int_{\Delta(r)}g_r(\textbf{0}, z)dv \\
 &=&  \frac{|s_c|}{2}\int_{0}^rdt\int_{\partial\Delta(t)}g_r(\textbf{0}, z)d\sigma_t \\
 &=& \frac{|s_c|}{2}\int_{0}^rg_r(\textbf{0}, z)|_{\partial\Delta(t)}\left(\sinh t\right)^{2m-1}dt\int_{\mathbb S^{2m-1}}dA  \\
 &=& \frac{\omega_{2m-1}|s_c|}{2}\int_{0}^rg_r(\textbf{0}, z)|_{\partial\Delta(t)}\left(\sinh t\right)^{2m-1}dt.
     \end{eqnarray*}
Substituting  (\ref{aaa}) into the  expression of $|T(r,\mathscr R)|$ and computing  the integral, it   is not  hard   to deduce    that 
$|T(r, \mathscr R)|\leq O(r).$ 
We proceed to estimate $E(r)$ with $\sigma,\tau\not=0.$
By $$\chi(\tau, r)\leq \frac{1}{2\tau}e^{\tau r},$$
we have $\log\chi(\tau, r)\leq \tau r+O(1).$  Hence, we obtain: 
 
  \begin{cor}\label{dong2}
  Let $w$ be a nonconstant   $\nu$-valued algebroid function on $\mathbb B^m.$  Equip $\mathbb B^m$ with the Poincar\'e metric.  Then        \begin{eqnarray*}
 (q-2\nu)T(r, w)
&\leq&  \sum_{j=1}^q\overline{N}\Big(r,\frac{1}{w-a_j}\Big)+O\left(\log^+ T(r, w)+r\right) 
         \end{eqnarray*}
holds for  all  $r>0$ outside a subset of finite Lebesgue measure. 
\end{cor}

\section{Unicity Problems  of Algebroid Functions}

Assume that  $M$ is a complete non-compact  K\"ahler manifold.

\subsection{Propagation of Algebraic Dependence}~

 Let $\omega_{FS}$ be the Fubini-Study form on $\mathbb P^1(\mathbb C).$  Given  a $\nu$-valued algebroid function  $w=\{w_j\}_{j=1}^\nu$ on $M.$ For   $a\in\mathbb P^1(\mathbb C),$
define  
    \begin{eqnarray*}
T_w(r,\omega_{FS})&=& \frac{1}{\nu}\sum_{j=1}^\nu T_{w_j}(r,\omega_{FS}), \\
m_w(r,a)&=&\frac{1}{\nu}\sum_{j=1}^\nu m_{w_j}(r,a), \\
  N_w(r, a)&=&\frac{1}{\nu}\sum_{j=1}^\nu N_{w_j}(r,a), \\
        \overline{N}_w(r, a)&=&\frac{1}{\nu}\sum_{j=1}^\nu \overline{N}_{w_j}(r,a).
            \end{eqnarray*}
It is not hard to deduce  that  
$T_w(r,\omega_{FS})+O(1)=m_w(r,a)+N_w(r,a).$ 
Since     
$m_w(r, \infty)=m(r, w)+O(1)$ and $N_w(r, \infty)=N(r, w),$
        we obtain  
$$T_w(r,\omega_{FS})=T(r, w)+O(1).$$ 
\ \ \ \   Fix any integer $l\geq2.$   
A proper algebraic subset $\Sigma$ of $\mathbb P^1(\mathbb C)^{l}$ is said to be  decomposible, if there exist  $s$ positive integers $l_1,\cdots,l_s$ 
with $l=l_1+\cdots+l_s$ for some  integer $s\leq l,$ 
and $s$ algebraic subsets $\Sigma_j\subseteq \mathbb P^1(\mathbb C)^{l_j}$ with $1\leq j\leq s,$ such that 
$\Sigma=\Sigma_1\times\cdots\times\Sigma_s.$ If $\Sigma$ is not decomposable,  then we say that $\Sigma$ is  indecomposable. 
Given  $l$ algebroid functions $w^1,\cdots, w^l: M\rightarrow \mathbb P^1(\mathbb C).$ Set 
 $$\tilde w=w^1\times\cdots\times w^l: \  M\rightarrow \mathbb P^1(\mathbb C)^{l}.$$ 

\begin{defi}\label{def11} Let $S$ be an analytic subset of  $M.$  We say that  $l$ nonconstant algebroid functions $w^1,\cdots,w^l$ on $M$ are  algebraically dependent on $S,$ if there exists a proper indecomposable algebraic subset $\Sigma$ of $\mathbb P^1(\mathbb C)^l$ such that
 $\tilde w(S)\subseteq\Sigma.$ If so,  $w^1,\cdots, w^l$ are said to be $\Sigma$-related on $S.$
\end{defi}

We let $Z$ be an effective divisor on $M$ and  fix a positive integer $k$ (allowed to be $+\infty$). Write $Z=\sum_j\nu_j Z_j$ for distinct irreducible hypersurfaces $Z_j$ of $M$ and  positive  integers $\nu_j.$  Define the support of $Z$ with order at most $k$ by
\begin{equation*}\label{def12}
{\rm{Supp}}_kZ=\bigcup_{1\leq\nu_j\leq k}Z_j.
\end{equation*}
  Note that  ${\rm{Supp}}Z={\rm{Supp}}_{+\infty}Z.$  
    Set 
\begin{equation*}
\mathscr G=\big\{\text{all  algebroid functions on $M$}\big\}. 
\end{equation*}
 Let $S_1,\cdots, S_q$ be  hypersurfaces of $M$ satisfying that $\dim_{\mathbb C}S_i\cap S_j\leq m-2$ if $m\geq2$ or $S_i\cap S_j=\emptyset$ if $m=1$ for  $i\not=j.$ 
 Let $a_1,\cdots, a_q$ be  distinct values in $\mathbb P^1(\mathbb C)$ and  let $k_1,\cdots, k_q$ be positive integers  (allowed to be $+\infty$).

We introduce the notation $\mathscr W$ as follows: 
  \begin{enumerate}
   \item[$\bullet$] 
 If $M$ is non-parabolic with non-negative Ricci curvature and volume growth satisfying (\ref{cond}),  then  we denote by
\begin{equation*}
\mathscr W=\mathscr W\big(w\in \mathscr G; \ \{k_j\}; \ (M, \{S_j\}); \ (\mathbb P^1(\mathbb C),  \{a_j\})\big)
\end{equation*}
the set of all $w\in\mathscr G$ such that 
$$S_j={\rm{Supp}}_{k_j}(w^*a_j), \ \ \ \  j=1,\cdots q.$$
   \item[$\bullet$] 
  If $M$ has   curvature satisfying (\ref{curvature}), then  we denote by 
\begin{equation*}
\mathscr W=\mathscr W\big(w\in \mathscr G; \ \{k_j\}; \ (M, \{S_j\}); \ (\mathbb P^1(\mathbb C),  \{a_j\})\big)
\end{equation*}
the set of all $w\in\mathscr G$ such that 
$$S_j={\rm{Supp}}_{k_j}(w^*a_j), \ \ \  \ j=1,\cdots, q$$
and  (\ref{xjing}) is satisfied.
  \end{enumerate}

 Let $\mathscr O(1)$   be    the point  line  bundle over $\mathbb P^1(\mathbb C).$ 
It  gives    a holomorphic  line bundle over $\mathbb P^1(\mathbb C)^l$:  
$$\tilde{\mathscr O}(1)=\pi^*_1\mathscr O(1)\otimes\cdots\otimes\pi^*_l\mathscr O(1),$$
where $\pi_k: \mathbb P^1(\mathbb C)^l\rightarrow \mathbb P^1(\mathbb C)$ is the natural projection on the $k$-th factor with  $1\leq k\leq l.$
Let $w^k=\{w^k_{j}\}_{j=1}^{\nu_k}\in\mathscr W$ with $k=1,\cdots,l.$
 For $\tilde D\in|\tilde{\mathscr O}(1)|$, define 
  \begin{eqnarray*}
T_{\tilde w}\big(r, \tilde{\mathscr O}(1)\big)&=&\frac{1}{\nu_1\cdots\nu_l}\sum_{1\leq j_1\leq\nu_1,\cdots, 1\leq j_l\leq\nu_l} T_{w_{j_1}^1\times\cdots\times w_{j_l}^l}(r, \tilde{\mathscr O}(1)), \\
 N_{\tilde w}\big(r, \tilde{D}\big)&=&\frac{1}{\nu_1\cdots\nu_l}\sum_{1\leq j_1\leq\nu_1,\cdots, 1\leq j_l\leq\nu_l} N_{w_{j_1}^1\times\cdots\times w_{j_l}^l}(r, \tilde D)
  \end{eqnarray*}
  with 
    \begin{eqnarray*}
T_{w_{j_1}^1\times\cdots\times w_{j_l}^l}(r, \tilde{\mathscr O}(1)) &=&\frac{\pi^m}{(m-1)!}\int_{\Delta(r)}\big(w_{j_1}^1\times\cdots\times w_{j_l}^l\big)^*\omega\wedge\alpha^{m-1}, \\
N_{w_{j_1}^1\times\cdots\times w_{j_l}^l}(r, \tilde D) &=&\frac{\pi^m}{(m-1)!}\int_{\big(w_{j_1}^1\times\cdots\times w_{j_l}^l\big)^*\tilde{D}\cap\Delta(r)}\alpha^{m-1}, 
  \end{eqnarray*}
  where $\omega\in c_1(\tilde{\mathscr O}(1))$  is a  positive $(1,1)$-form. 
\begin{lemma}\label{ghgh} We have 
$$N_{\tilde w}\big(r, \tilde{D}\big)\leq T_{\tilde w}\big(r, \tilde{\mathscr O}(1)\big)+O(1)= \sum_{k=1}^lT_{w^{k}}(r, \omega_{FS}).$$
\end{lemma}
\begin{proof}  A direct computation leads to  
$$T_{w_{j_1}^1\times\cdots\times w_{j_l}^l}\big(r, \tilde{\mathscr O}(1)\big)=
\sum_{k=1}^lT_{w^{k}_{j_k}}(r, \mathscr O(1))+O(1).$$
Thus, we have  
  \begin{eqnarray*}  
T_{\tilde w}\big(r, \tilde{\mathscr O}(1)\big) &=&  \frac{1}{\nu_1\cdots\nu_l}\sum_{1\leq j_1\leq\nu_1,\cdots, 1\leq j_l\leq\nu_l} \sum_{k=1}^lT_{w^{k}_{j_k}}(r, \mathscr O(1))+O(1) \\
&=&  \frac{1}{\nu_1\cdots\nu_l}\sum_{1\leq j_1\leq\nu_1,\cdots, 1\leq j_l\leq\nu_l} \sum_{k=1}^lT_{w^{k}_{j_k}}(r, \omega_{FS})+O(1) \\
&=&\sum_{k=1}^l \frac{1}{\nu_k}\sum_{j=1}^{\nu_k}T_{w^{k}_{j}}(r, \omega_{FS})+O(1) \\
&=& \sum_{k=1}^lT_{w^{k}}(r, \omega_{FS})+O(1). 
  \end{eqnarray*}
  Applying  Dynkin formula, we conclude  that      
  $N_{\tilde w}(r, \tilde{D})\leq T_{\tilde w}(r, \tilde{\mathscr O}(1))+O(1).$
  This completes the proof. 
\end{proof}
We denote by $\mathscr H$  the set of all indecomposable  hypersurfaces $\Sigma$ of $\mathbb P^1(\mathbb C)^l$ such that $\Sigma={\rm{Supp}}\tilde D$ for some $\tilde D\in|\tilde {\mathscr O}(1)|.$  
Set  
$$S=S_1\cup\cdots\cup S_q, \ \ \ \ 
 k_0=\max\{k_1,\cdots, k_q\}.
$$
\begin{lemma}\label{lem1}  Assume that $\tilde w(S)\subseteq \Sigma$ and $\tilde w(M)\not\subseteq \Sigma$ for some $\Sigma\in\mathscr H.$
Then  
$$N(r, S)\leq \nu_1\cdots\nu_l\sum_{k=1}^lT_{w^{k}}(r, \omega_{FS})+O(1).$$
\end{lemma}
\begin{proof}  Take $\tilde D\in|\tilde L|$ such that $\Sigma={\rm{Supp}}\tilde D.$     Then 
  \begin{eqnarray*}
N(r, S)\leq N\big(r, \tilde{w}^{-1}(\Sigma)\big) 
\leq N\big(r, \tilde{w}^*{\rm{Supp}}\tilde D\big) 
\leq \nu_1\cdots\nu_lN_{\tilde w}\big(r, \tilde D\big). 
  \end{eqnarray*}
 In further, it yields from  Lemma \ref{ghgh} that 
  \begin{eqnarray*}
N(r, S)&\leq&  \nu_1\cdots\nu_l T_{\tilde w}\big(r, \tilde{\mathscr O}(1)\big)+O(1) \\
&\leq& \nu_1\cdots\nu_l\sum_{k=1}^lT_{w^{k}}(r, \omega_{FS})+O(1). 
  \end{eqnarray*}
\end{proof}
\begin{lemma}\label{lem}
Let $w$ be a $\nu$-valued algebroid function in $\mathscr W.$ Then 
  \begin{eqnarray*}
&&(q-2\nu)T_w(r, w_{FS})\\
&\leq& \frac{k_0}{(k_0+1)\nu}N(r, S)+\sum_{j=1}^q\frac{1}{k_j+1}N_w(r, a_j)+o\big{(}T_w(r,w_{FS})\big{)}.
  \end{eqnarray*}
\end{lemma}

\begin{proof} Since $S_j={\rm{Supp}}_{k_j}w^*a_j,$ we obtain    
  \begin{eqnarray*}
{\rm{Supp}}(w^*a_j)&\leq&\frac{k_j}{k_j+1}{\rm{Supp}}_{k_j}(w^*a_j)+\frac{1}{k_j+1}w^*a_j \\
&=&\frac{k_j}{k_j+1}S_j+\frac{1}{k_j+1}w^*a_j
  \end{eqnarray*}
in the sense of currents, where ${\rm{Supp}}(w^*a_j)$ and ${\rm{Supp}}_{k_j}(w^*a_j)$ are understood as  divisors. 
Thus, we are led to   
  \begin{eqnarray*}
\nu\overline{N}_w(r, a_j) 
&=& \overline{N}(r, w^*a_j) \\
&\leq& \frac{k_j}{k_j+1}N(r, S_j)+\frac{1}{k_j+1}N(r, w^*a_j) \\
&\leq&  \frac{k_0}{k_0+1}N(r, S_j)+\frac{\nu}{k_j+1}N_w(r, a_j). 
  \end{eqnarray*}
  By this with Theorem \ref{sec3}, we obtain 
    \begin{eqnarray*}
&&(q-2\nu)T_w(r, w_{FS})\\
&\leq& \frac{k_0}{(k_0+1)\nu}\sum_{j=1}^qN(r, S_j)+\sum_{j=1}^q\frac{1}{k_j+1}N_w(r, a_j)+o\big{(}T_w(r,w_{FS})\big{)} \\
&=& \frac{k_0}{(k_0+1)\nu}N(r, S)+\sum_{j=1}^q\frac{1}{k_j+1}N_w(r, a_j)+o\big{(}T_w(r,w_{FS})\big{)}.
  \end{eqnarray*}
  This proves the lemma. 
\end{proof}
Set     
$$\gamma=\sum_{j=1}^q\frac{k_j}{k_j+1}-\frac{k_0\nu_1\cdots\nu_l}{k_0+1}
\sum_{k=1}^l\frac{1}{\nu_k}-2\nu_{0},
$$
where  $\nu_{0}=\max\{\nu_1,\cdots,\nu_l\}.$

\begin{theorem}\label{uni1} Let $w^k$ be a $\nu_k$-valued algebrod function in $\mathscr W$ for $1\leq k\leq l.$  Assume that $w^1,\cdots,w^l$  are $\Sigma$-related on $S$ for some $\Sigma\in \mathscr H.$   If $\gamma>0,$
   then $w^1,\cdots,w^l$  are $\Sigma$-related on $M.$
\end{theorem}
\begin{proof}
It suffices to show $\tilde w(M)\subseteq\Sigma.$ Otherwise,  we assume that $\tilde w(M)\not\subseteq\Sigma.$ According to   Lemma \ref{lem}, for $k=1,\cdots, l$
  \begin{eqnarray*}
  &&(q-2\nu_{0})T_{w^k}(r, \omega_{FS})  \\
&\leq& \frac{k_0}{(k_0+1)\nu_k}N(r, S)+\sum_{j=1}^q\frac{1}{k_j+1}T_{w^k}(r, \omega_{FS})+o\big{(}T_{w^k}(r,\omega_{FS})\big{)}, 
  \end{eqnarray*}
which yields  from Lemma \ref{lem1} that
  \begin{eqnarray*}  
  && \bigg(\sum_{j=1}^q\frac{k_j}{k_j+1}-2\nu_{0}\bigg)T_{w^k}(r, \omega_{FS})  \\
&\leq& \frac{k_0}{(k_0+1)\nu_{k}}N(r, S)+o\big{(}T_{w^k}(r,\omega_{FS})\big{)} \\
&\leq&  \frac{k_0\nu_1\cdots\nu_l}{(k_0+1)\nu_{k}}\sum_{k=1}^lT_{w^{k}}(r, \omega_{FS})+o\big{(}T_{w^k}(r,\omega_{FS})\big{)}. 
  \end{eqnarray*}
  Thus, we get 
   \begin{eqnarray*}
&&\bigg(\sum_{j=1}^q\frac{k_j}{k_j+1}-\frac{k_0\nu_1\cdots\nu_l}{k_0+1}
\sum_{k=1}^l\frac{1}{\nu_k}-2\nu_{0}\bigg)\sum_{k=1}^lT_{w^k}(r,\omega_{FS}) \\
&\leq&
  o\bigg{(}\sum_{k=1}^lT_{w^k}(r,\omega_{FS})\bigg{)}, 
  \end{eqnarray*} 
  which contradicts with 
  $\gamma>0.$
This proves  the theorem. 
\end{proof}
Set  
$$\mathscr W_\varsigma =\big\{\text{all $\nu$-valued algebroid functions in }   \mathscr W  \text{satisfying} \ \nu\leq\varsigma \big\}.$$
Fix any $\mu$-valued algebroid function  $w_0\in\mathscr W_\varsigma .$ We say that  $\{a_j\}_{j=1}^q$ is   
\emph{generic} with respect to $w_0,$  if 
 $w_0(M)\cap a_s\not=\emptyset$
 for at least one $s\in\{1,\cdots,q\}.$  
 Assume that   $\{a_j\}_{j=1}^q$ is generic with respect to $w_0.$  
  Let $\mathscr W_{\varsigma }^0$ be the set of all $w\in\mathscr W_\varsigma$ such that $w=w_0$ on  $S.$
  Set
 $$\gamma_0=\sum_{j=1}^q\frac{k_j}{k_j+1}-\frac{2k_0\varsigma}{k_0+1}-2\varsigma .$$

 \begin{theorem}\label{t1}
 If $\gamma_0>0,$ then $\mathscr W_\varsigma^0$ contains exactly  one element. 
 \end{theorem}
 \begin{proof} Let $w$ be a $\nu$-valued algebroid function in $\mathscr W^0_\varsigma.$
 It suffices to show that $w\equiv w_0$ on $M.$  Set   
 $\phi=w\times w_0.$
  Let $\Delta$ denote  the diagonal of $\mathbb P^1(\mathbb C)\times \mathbb P^1(\mathbb C).$
  Since $w=w_0$ on $S,$ we obtain   $\phi(S)\subseteq \Delta.$  
   Let $\pi_k: \mathbb P^1(\mathbb C)\times \mathbb P^1(\mathbb C) \rightarrow \mathbb P^1(\mathbb C)$ be the natural projection on the $k$-th factor with  $k=1,2.$ We 
   consider the holomorphic line bundle $\tilde{\mathscr O}(1)=\pi_1^*\mathscr O(1)\otimes\pi_2^*\mathscr O(1)$ over  $\mathbb P^1(\mathbb C)\times \mathbb P^1(\mathbb C),$
    where $\mathscr O(1)$ is the point  line bundle over  $\mathbb P^1(\mathbb C).$ 
     If  $w\not\equiv w_0$ on $M,$  
  then  there exists   a  holomorphic  section $\tilde\sigma$ of $\tilde{\mathscr O}(1)$ 
     over $\mathbb P^1(\mathbb C)\times \mathbb P^1(\mathbb C)$ such that   $\phi^*\tilde\sigma\not=0$ and $\Delta\subseteq{\rm{Supp}}(\tilde\sigma=0).$ 
      Take $\Sigma={\rm{Supp}}(\tilde\sigma=0),$ then one  can have    $\phi(S)\subseteq\Sigma$ and $\phi(M)\not\subseteq\Sigma.$  
On the other hand,  $\gamma_0>0$ implies that 
 $$\gamma=\sum_{j=1}^q\frac{k_j}{k_j+1}-\frac{2k_0\nu_{0}}{k_0+1}-2\nu_{0}>0,$$
where
$\nu_{0}=\max\{\mu, \nu\}.$
With the help of  Theorem \ref{uni1}, 
we obtain   $\phi(M)\subseteq\Sigma,$  which  is a contradiction.   
This proves the theorem.
 \end{proof}

\subsection{Five-Value Type Theorem}~

\noindent\textbf{A1}   \emph{$M$ has non-negative Ricci curvature}
 
Assume that $M$ is non-parabolic with  volume growth satisfying (\ref{cond}). 

 \begin{theorem}\label{above} Let $u, v$ be transcendental  $\mu$-valued,  $\nu$-valued algebroid  functions on $M$ respectively.  
  Let $a_1,\cdots,a_q$ be distinct values in $\mathbb P^1(\mathbb C).$  Then  
 
 $(a)$ Assume ${\rm{Supp}}(u^*a_j)={\rm{Supp}}(v^*a_j)\not=\emptyset$ for all $j.$ If $q\geq 4\max\{\mu,\nu\}+1,$ then $u\equiv v;$
 
 $(b)$ Assume  ${\rm{Supp}}_1(u^*a_j)={\rm{Supp}}_1(v^*a_j)\not=\emptyset$ for all $j.$ If $q\geq 6\max\{\mu,\nu\}+1,$ then $u\equiv v.$
  \end{theorem}
  
 \begin{proof}
Set $\varsigma =\max\{\mu, \nu\}$   
  and $k_j=+\infty$ with  $j=1,\cdots,q.$   Then
   $$\gamma_0=q-4\max\{\mu, \nu\}=q-4\varsigma.$$
 By  Theorem \ref{t1},  we have  $(a)$ holds. 
 Similarly,  when $k_j=1$ for $j=1,\cdots,q,$    we deduce that  
$$\gamma_0=\sum_{j=1}^q\frac{1}{1+1}-\frac{2\varsigma}{1+1}-2\varsigma
= \frac{q}{2}-3\varsigma,
$$     which   implies that  $(b)$ holds. 
 \end{proof}
 \begin{cor}[Five-Value Type Theorem]\label{above11} Let $u, v$ be nonconstant   $\nu$-valued algebroid  functions on $M.$   
  If $u, v$  share $4\nu+1$ distinct values ignoring multiplicities  in $\mathbb P^1(\mathbb C),$  then $u\equiv v.$  
 \end{cor}

\noindent\textbf{B1}   \emph{$M$ has non-positive  sectional  curvature} 
 
Assume that $M$ has   curvature satisfying (\ref{curvature}). 
 \begin{theorem}\label{above1} Let $u, v$ be nonconstant $\mu$-valued, $\nu$-valued algebroid  functions on $M$ respectively.  
  Let $a_1,\cdots,a_q$ be distinct values in $\mathbb P^1(\mathbb C).$  Assume that $u, v$ satisfy  $(\ref{xjing}).$
  Then  
 
 $(a)$ Assume  ${\rm{Supp}}(u^*a_j)={\rm{Supp}}(v^*a_j)\not=\emptyset$ for all $j.$ If $q\geq 4\max\{\mu,\nu\}+1,$ then $u\equiv v;$
 
 $(b)$ Assume  ${\rm{Supp}}_1(u^*a_j)={\rm{Supp}}_1(v^*a_j)\not=\emptyset$ for all $j.$ If $q\geq 6\max\{\mu,\nu\}+1,$ then $u\equiv v.$
 \end{theorem}
 \begin{proof}
The  argument  is similar as in the proof of Theorem \ref{above}. 
 \end{proof}
 
 \begin{cor}[Five-Value Type Theorem]\label{above12} Let $u, v$ be nonconstant  $\nu$-valued algebroid  functions on $M.$   
 Assume that $u, v$ satisfy  $(\ref{xjing}).$
  If $u, v$  share $4\nu+1$ distinct values ignoring multiplicities  in $\mathbb P^1(\mathbb C),$  then $u\equiv v.$   
 \end{cor}

\noindent\textbf{A2}     \emph{$M$ has non-negative Ricci curvature}  

Assume that $M$ is non-parabolic with  volume growth satisfying (\ref{cond}). 

 \begin{theorem}\label{aboveqq} Let $u, v$ be nonconstant   $\mu$-valued,  $\nu$-valued algebroid  functions on $M$ respectively.  
  Let $a_1,\cdots,a_q$ be distinct values in $\mathbb P^1(\mathbb C).$    
  Assume that ${\rm{Supp}}_{k_j}(u^*a_j)={\rm{Supp}}_{k_j}(v^*a_j)\not=\emptyset$ for all $j.$ 
   If $q$ satisfies  $(\ref{q}),$ 
   then $u\equiv v.$
 \end{theorem}
\begin{proof}
Assume on the contrary that $u\not\equiv v.$ The condition implies that  
\begin{eqnarray*}
{\rm{Supp}}(u^*a_j)&\leq&\frac{k_j}{k_j+1}{\rm{Supp}}_{k_j}(u^*a_j)+\frac{1}{k_j+1}u^*a_j \\
&=&\frac{k_j}{k_j+1}S_j+\frac{1}{k_j+1}u^*a_j
\end{eqnarray*}
in the sense of currents, where ${\rm{Supp}}(u^*a_j)$ and ${\rm{Supp}}_{k_j}(u^*a_j)$ are understood as  divisors. 
Thus,  it yields from $\mu\overline N_u(r, a_j)=\overline N(r, u^*a_j)$ that 
\begin{eqnarray*}
\mu\sum_{j=1}^q\overline N_u(r, a_j) &=& 
\sum_{j=1}^q\frac{k_j}{k_j+1}N(r, S_j)+\sum_{j=1}^q\frac{1}{k_j+1}N(r, u^*a_j) \\
&\leq& \frac{k_0}{k_0+1}N(r, S)+\sum_{j=1}^q\frac{\mu}{k_j+1}N_u(r, a_j) \\
&\leq& \frac{k_0}{k_0+1}N(r, S)+\sum_{j=1}^q\frac{\mu}{k_j+1}T_u(r, \omega_{FS})+O(1).
\end{eqnarray*}
In further, it follows  from  Lemma \ref{lem1}   that 
\begin{eqnarray*}
 \mu\sum_{j=1}^q\overline N_u(r, a_j) 
&\leq& \frac{k_0\mu\nu}{k_0+1}\big(T_u(r, \omega_{FS})+T_v(r, \omega_{FS})\big) \\
&&+\sum_{j=1}^q\frac{\mu}{k_j+1}T_u(r, \omega_{FS})+O(1). 
\end{eqnarray*}
Similarly, we have 
\begin{eqnarray*}\label{hh2}
\nu\sum_{j=1}^q\overline N_v(r, a_j)&\leq& \frac{k_0\mu\nu}{k_0+1}\big(T_u(r, \omega_{FS})+T_v(r, \omega_{FS})\big) \\
&& +\sum_{j=1}^q\frac{\nu}{k_j+1}N_v(r, a_j)+O(1). \nonumber
\end{eqnarray*}
Combining  the above and using  Theorem \ref{sec3}, we are led to   
\begin{eqnarray*}
& &\mu(q-2\mu)T_u(r,\omega_{FS})+\nu(q-2\nu)T_v(r,\omega_{FS}) \\
&\leq&  \frac{2k_0\mu\nu}{k_0+1}\big(T_u(r, \omega_{FS})+T_v(r, \omega_{FS})\big)
+\sum_{j=1}^q\frac{\mu}{k_j+1}T_u(r, \omega_{FS}) \\
&& + \sum_{j=1}^q\frac{\nu}{k_j+1}T_v(r, \omega_{FS})+o\big(T_u(r,\omega_{FS})+T_v(r,\omega_{FS})\big). 
\end{eqnarray*}
It is therefore 
\begin{eqnarray*}
& & \mu\Big(q-2\mu-\frac{2k_0\nu}{k_0+1}-\sum_{j=1}^q\frac{1}{k_j+1}\Big)T_u(r,\omega_{FS}) \\
&& + \nu\Big(q-2\nu-\frac{2k_0\mu}{k_0+1}-\sum_{j=1}^q\frac{1}{k_j+1}\Big)T_v(r,\omega_{FS}) \\
&\leq& o\big(T_u(r,\omega_{FS})+T_v(r,\omega_{FS})\big),
\end{eqnarray*}
which  contradicts with the condition (\ref{q}) for $q.$ This completes  the proof.
\end{proof}

 \begin{cor}\label{coro33} Let $u, v$ be nonconstant  $\mu$-valued,  $\nu$-valued algebroid  functions on $M,$ respectively.  
  Let $a_1,\cdots,a_q$ be distinct values in $\mathbb P^1(\mathbb C).$  Then  
 
 $(a)$ Assume ${\rm{Supp}}(u^*a_j)={\rm{Supp}}(v^*a_j)\not=\emptyset$ for all $j.$ If $q\geq 2\mu+2\nu+1,$ then $u\equiv v;$
 
 $(b)$ Assume  ${\rm{Supp}}_1(u^*a_j)={\rm{Supp}}_1(v^*a_j)\not=\emptyset$ for all $j.$ If 
 $q\geq \max\{2\mu+4\nu,2\nu+4\mu\}+1,$ then $u\equiv v.$
  \end{cor}
 \begin{proof}
It is not hard   to show  that  $(\ref{q})$ holds  with  $k_1=\cdots=k_q=+\infty$ under  the condition for $q.$
 Employing  Theorem \ref{aboveqq}, we have $(a)$ holds.  Similarly,  we have  $(\ref{q})$ holds  with  $k_1=\cdots=k_q=1.$
 Hence, we  prove $(b).$ 
 \end{proof}
 
\begin{cor} [Five-Value Type Theorem]\label{aass}
Let $u, v$ be nonconstant  $\mu$-valued, $\nu$-valued algebroid  functions on $M$ respectively.  
  If $u, v$  share $2\mu+2\nu+1$ distinct values ignoring multiplicities  in $\mathbb P^1(\mathbb C),$  then $u\equiv v.$  
\end{cor}

\noindent\textbf{B2}    \emph{$M$ has non-positive  sectional  curvature}
 
 Assume that $M$ has   curvature satisfying  (\ref{curvature}). 

 \begin{theorem}\label{} Let $u, v$ be nonconstant $\mu$-valued,  $\nu$-valued algebroid  functions on $M$  satisfying  $(\ref{xjing}),$ respectively. 
  Let $a_1,\cdots,a_q$ be distinct values in $\mathbb P^1(\mathbb C).$    
  Assume that ${\rm{Supp}}_{k_j}(u^*a_j)={\rm{Supp}}_{k_j}(v^*a_j)\not=\emptyset$ for all $j.$ 
   If $q$ satisfies  $(\ref{q}),$
  then $u\equiv v.$
 \end{theorem}
 \begin{proof}
 The  argument is similar as  in the proof of Theorem \ref{aboveqq}.
 \end{proof}
 
 \begin{cor}\label{cor66} Let $u, v$ be nonconstant $\mu$-valued,  $\nu$-valued algebroid  functions on $M$  satisfying  $(\ref{xjing}),$ respectively. 
  Let $a_1,\cdots,a_q$ be distinct values in $\mathbb P^1(\mathbb C).$  Then  
 
 $(a)$ Assume ${\rm{Supp}}(u^*a_j)={\rm{Supp}}(v^*a_j)\not=\emptyset$ for all $j.$ If $q\geq 2\mu+2\nu+1,$ then $u\equiv v;$
 
 $(b)$ Assume  ${\rm{Supp}}_1(u^*a_j)={\rm{Supp}}_1(v^*a_j)\not=\emptyset$ for all $j.$ If 
 $q\geq \max\{2\mu+4\nu,2\nu+4\mu\}+1,$ then $u\equiv v.$
 \end{cor}
 
\begin{cor} [Five-Value Type Theorem]
Let $u, v$ be nonconstant $\mu$-valued, $\nu$-valued algebroid  functions on $M$  satisfying 
$(\ref{xjing}),$ respectively.   If $u, v$  share $2\mu+2\nu+1$ distinct values ignoring multiplicities  in $\mathbb P^1(\mathbb C),$  then $u\equiv v.$   
\end{cor}


\vskip\baselineskip

\end{document}